\RequirePackage{fix-cm}
\documentclass[smallextended]{svjour3}
\makeatletter
\newcommand{\algname}[1]{\def\@currentlabelname{#1}}
\makeatother
\smartqed
\usepackage{graphicx}
\usepackage{algorithm}
\usepackage{algpseudocode} 
\usepackage{comment} 
\usepackage{tcolorbox}
\usepackage{footnote}
\usepackage[misc]{ifsym}
\usepackage{graphicx}
\usepackage{amssymb}
\usepackage{tikz}
\usepackage{amsmath}
\usepackage{multirow}
\usepackage{longtable}
\usepackage{xcolor,colortbl}
\usepackage{lscape}
\usepackage{siunitx}
\usepackage{url}
\usepackage{epstopdf}
\usepackage{appendix}
\usepackage[utf8]{inputenc}
\usepackage{booktabs} 

\usepackage[caption=false]{subfig}
\usepackage[english]{babel}
 \usepackage[final, colorlinks = true,
            linkcolor = blue,
            urlcolor  = blue,
            citecolor = blue,
            anchorcolor = blue]{hyperref}

\DeclareMathOperator*{\argmin}{\arg\!\min}
\DeclareMathOperator*{\lmo}{lmo}
\newcommand{\C}{\mathcal{C}}

\usepackage[a4paper,margin=1in,footskip=0.25in]{geometry}

\begin{document}

\title{Adaptive Conditional Gradient Descent}
\titlerunning{Adaptive Conditional Gradient Descent}
\author{Abbas Khademi \and Antonio Silveti-Falls} 
\authorrunning{A. Khademi and A. Silveti-Falls}
\institute{A. Khademi \at 
School of Mathematics, Institute for Research in Fundamental Sciences (IPM), P.O.Box 19395-5746, Tehran, Iran.\\
School of Mathematics and Computer Science, Iran University of Science and Technology, Tehran, Iran.
\\
\email{\href{abbaskhademi92@gmail.com}{abbaskhademi92@gmail.com}}\\
https://orcid.org/0000-0002-8276-6821
\and
A. Silveti-Falls \emph{(Corresponding Author)} \at 
CVN, CentraleSupélec, Université Paris-Saclay, Inria, France.\\ \email{\href{tonys.falls@gmail.com}{tonys.falls@gmail.com}}\\
https://orcid.org/0000-0002-8165-6348
}
\date{Received: date / Accepted: date}
\maketitle
\abstract{
Selecting an effective step-size is a fundamental challenge in first-order optimization, especially for problems with non-Euclidean geometries.
This paper presents a novel adaptive step-size strategy for optimization algorithms that rely on linear minimization oracles, as used in the Conditional Gradient or non-Euclidean Normalized Steepest Descent algorithms. Using a simple heuristic to estimate a local Lipschitz constant for the gradient, we can determine step-sizes that guarantee sufficient decrease at each iteration.
More precisely, we establish convergence guarantees for our proposed Adaptive Conditional Gradient Descent algorithm, which covers as special cases both the classical Conditional Gradient algorithm and non-Euclidean Normalized Steepest Descent algorithms with adaptive step-sizes. Our analysis covers optimization of continuously differentiable functions in non-convex, quasar-convex, and strongly convex settings, achieving convergence rates that match state-of-the-art theoretical bounds.
Comprehensive numerical experiments validate our theoretical findings and illustrate the practical effectiveness of Adaptive Conditional Gradient Descent. The results exhibit competitive performance, underscoring the potential of the adaptive step-size for applications.
}

\section{Introduction}

\subsection{Problem Context}
First-order methods based on linear minimization oracles offer a powerful paradigm for structured optimization, providing either projection-free optimization for constrained problems or geometry-adaptive descent for unconstrained problems. These methods have found applications ranging from high-dimensional matrix completion \cite{shalev2011large}, sparse optimization \cite{cai2011orthogonal,locatello2018matching}, constrained multiobjective optimization problems \cite{assunccao2021conditional,gonccalves2024away}, submodular optimization \cite{calinescu2011maximizing,vondrak2008optimal}, convex mixed-integer optimization \cite{hendrych2025convex}, and traffic flow problems \cite{leblanc1975efficient} to training deep neural networks with billions of parameters \cite{pethick2025training,pethick2025generalized}, where the linear minmization oracle structure naturally encodes problem geometry through its extremal points.

This framework unifies two important algorithms that originate from different problem settings. In constrained optimization, it yields the Conditional Gradient algorithm, which sidesteps potentially expensive projections onto a constraint set. In unconstrained optimization, it produces the Normalized Steepest Descent algorithm, which adapts the direction of descent to the geometry of the problem (or our choosing). Both methods share a critical dependency: their convergence rates and practical performance hinge on appropriate step-size selection, which is a challenge that existing approaches address only partially.

While various step-size strategies have been proposed in the literature, ranging from conservative, descent-driven choices that require knowledge of the global Lipschitz constant of the gradient \cite{dem1967minimization} to unassuming, parameter-free open-loop schedules \cite{dunn1978conditional} and adaptive backtracking approaches \cite{pedregosa2020linearly}, none fully exploit the local geometry that the iterative nature of these algorithms naturally encounters. We develop an adaptive step-size strategy that estimates local Lipschitz constants of the gradient on-the-fly using previous gradient information followed by a backtracking correction procedure, combining the robustness of backtracking with the efficiency of curvature-aware steps.

To formalize our approach, we consider the following general problem setting
\begin{equation}\tag{P}\label{P}
    \min\limits_{x\in\mathcal{X}} f(x),
\end{equation}
where $\mathcal{X}\subseteq \mathbb{R}^n$ and $f\colon\mathbb{R}^n\to\mathbb{R}$ is a continuously differentiable function. 
We will assume that the gradient $\nabla f$ is Lipschitz-continuous with respect to a norm $\|\cdot\|$, i.e., there exists a constant $L > 0$ such that for all $x, y \in\mathcal{X}$,
$$
\|\nabla f(x) - \nabla f(y)\|_\ast \leq L \|x - y\|,
$$
which we will also sometimes refer to as $f$ being $L$-smooth. While we will assume that $L$ exists, we will not assume knowledge of the value of $L$ in our study of the problem.

For a non-empty compact convex set $\mathcal{C}$, most crucially, we will assume that we are able to solve a so-called Linear Minimization Oracle (LMO) of the form
\begin{equation}
\lmo_{\mathcal{C}}(x) := \argmin\limits_{v\in\mathcal{C}}\langle x,v\rangle,
\end{equation}
where $\lmo_{\mathcal{C}}(x)$ can be any minimizer in the case that there are many, The LMO always returns a boundary point of $\mathcal{C}$ (typically an extremal point) and this structure is what enables both the projection-free property of the Conditional Gradient algorithm and the adaptation to the geometry of the problem in the Normalized Steepest Descent algorithm. The specific structure of the LMO output, i.e., sparse for $\ell^1$-balls, dense for $\ell^\infty$-balls, rank-$1$ for nuclear norm-balls, full-rank for operator norm-balls, etc, endows these algorithms with their distinctive computational properties.

We will study a general algorithmic template, \nameref{alg:general}, based on the LMO for solving \eqref{P}, which is applicable to both the unconstrained and constrained variants; when applied to unconstrained problems it is sometimes referred to as Normalized Steepest Descent algorithm, the Continuous Greedy algorithm, or the Generalized Matching Pursuit algorithm while it is called the Frank-Wolfe algorithm or the Conditional Gradient algorithm when applied to constrained problems. The only difference in the template between the two types of problems is a slight modification to the update direction $d^k$ in the constrained case, to ensure feasibility. However, both algorithms share the same core structure in that they query an LMO at each iteration to obtain a direction, then take a step in that direction and it is this shared structure that allows us to develop a unified step-size strategy that works for both.

\paragraph{Conditional Gradient} The Conditional Gradient algorithm is a classical method to solve \eqref{P} when $\mathcal{X}=\mathcal{C}$. At every iteration $k$, the algorithm linearizes the objective function $f$ around the current iterate $x^k$ and then minimizes this linearization over the constraint set $\mathcal{C}$, which amounts to querying the LMO. Then, since the LMO is always in $\mathcal{C}$, taking a convex combination between the current feasible iterate $x^k$ and the LMO output $v^k$ is guaranteed to remain in $\mathcal{C}$ by convexity. Since the iterates never leave $\mathcal{C}$, there is no need to project back onto $\mathcal{C}$ and the method is thus \emph{projection-free}.

\paragraph{Normalized Steepest Descent} The Normalized Steepest Descent algorithm is also a classic method to solve \eqref{P} when $\mathcal{X} = \mathbb{R}^n$. It is analogous to doing normalized gradient descent but in a non-Euclidean norm, which is advantageous in many problems whose geometry is not necessarily Euclidean. By selecting a particular norm, we can enforce structured updates (sparse, dense, low-rank, full-rank, etc) or improve the convergence rate for problems with particular geometries. Similar to the Conditional Gradient algorithm, it requires computing a linearization of the objective function $f$ around the current iterate $x^k$ that is then minimized over a unit-ball $\mathcal{C}$ for some norm to get some update direction. Rather than taking a convex combination with the current iterate, though, Normalized Steepest Descent directly steps in the direction output by the LMO.

\subsection{Related Work}

Step-size selection in LMO-based methods has been extensively studied since the seminal work of Frank and Wolfe \cite{frank1956algorithm}. Classical approaches range from the short-step step-size rule, which requires knowledge of the global Lipschitz constant $L$ and sets $t_k = \min\{-\langle \nabla f(x^k), d^k \rangle/(L\|d^k\|^2), 1\}$ \cite{bomze2024frank,dunn1978conditional}, to parameter-free open-loop schedules such as $t_k = 2/(2+k)$ that sacrifice adaptivity for simplicity \cite{dunn1978conditional}. While exact line search maximizes per-iteration progress, its computational cost often outweighs its benefits, especially when gradients are expensive to evaluate \cite{dem1970approximate}. These classical methods either require problem-specific constants that are difficult to estimate in practice or fail to adapt to the local geometry of the objective function.

There is a closely related work \cite{pedregosa2020linearly} to ours, which introduced a unified backtracking framework for both Conditional Gradient and Normalized Steepest Descent that eliminates the need for knowing the global Lipschitz constant $L$. Their adaptive strategy maintains an estimate of the local Lipschitz constant that is decreased by a multiplicative factor (typically 0.9) when backtracking occurs. While this approach successfully adapts to the problem landscape without prior knowledge of $L$, it suffers from slow adaptation when the iterates move into regions of significantly lower curvature—requiring multiple iterations for the estimate to decrease substantially. Our work builds directly upon this foundation by incorporating gradient-based local Lipschitz constant estimation to accelerate this adaptation process.

Several other step-size strategies have been proposed for LMO-based methods. For the Conditional Gradient side, there has recently been introduced \cite{hendrych2025secant} a secant-based line search that adapts to local smoothness with reduced computational cost compared to exact line search, though it requires additional gradient evaluations. Other works \cite{wirth2025accelerated} have demonstrated accelerated convergence with open-loop step-sizes under specific geometric conditions, while optimistic step-size variants that go beyond the short-step rule have also been proposed \cite{martinez2025beyond}. On the Normalized Steepest Descent side, the connection to continuous greedy \cite{calinescu2011maximizing} and generalized matching pursuit has been formalized in \cite{locatello2017unified}, which provided a unified optimization view establishing explicit convergence rates for matching pursuit methods in an optimization sense, demonstrating the deep connection between these seemingly separate algorithms.

Our analysis holds for quasar-convex functions, which are a broader class of functions than convex functions or even star-convex functions, and which have gotten more attention recently \cite{guminov2023accelerated,guminov2019primal,martinez-rubio2025smooth,wang2023continuized}. The convergence of the Conditional Gradient algorithm under star-convexity has been analyzed in a recent work \cite{millan2025frank} under a few different step-size schedules which include open-loop, short-step, and backtracking. Of note is that their analysis of the adaptive backtracking algorithm \cite{pedregosa2020linearly} uses the star-convexity assumption, attaining similar rates to ours for the more general quasar-convex functions.

Another relevant line of work \cite{malitsky2019adaptive,malitsky2024adaptive} developed gradient-based heuristics for estimating local smoothness constants in proximal gradient methods. Their approach uses consecutive gradients and iterates to estimate the local Lipschitz constant without requiring additional function evaluations. While their work focused on proximal methods, the principle of exploiting gradient information for local curvature estimation has broader applicability, as we demonstrate in this work.

\paragraph{Contributions} 
This paper presents an adaptive step-size method for \nameref{alg:general} that estimates the local Lipschitz constant of the gradient through a simple heuristic. Unlike methods that rely on the global Lipschitz constant (as in the short-step) or backtracking alone \cite{pedregosa2020linearly}, our approach adapts efficiently to varying problem landscapes where global properties are difficult to estimate or exhibit significant variation.
Our theoretical contributions are rigorous convergence analyses of both variants of \nameref{alg:general} for \eqref{P} under convex, quasar-convex, and non-convex settings. Notably, under an additional quasar-convexity assumption, we establish improved convergence rates in the non-convex case compared to \cite{pedregosa2020linearly}. Our use of local Lipschitz constants furthermore yields simplified proofs throughout.
We validate our approach's superiority through comprehensive synthetic experiments designed to isolate and demonstrate the benefits of local adaptation across diverse problem geometries and conditioning. By warm-starting the backtracking procedure with gradient-based estimates of the local Lipschitz constant, our algorithm rapidly adapts to changes in the problem landscape, which addresses a limitation of pure backtracking approaches. These experiments confirm our hypothesis that leveraging local curvature information leads to faster practical convergence. 

\subsection{Outline}
The structure of the remainder of the paper is as follows.
Section~\ref{sec:preliminaries} reviews the Conditional Gradient method, Normalized Steepest Descent and greedy methods, their connection through the LMO, and classic step-size choices for both methods: the short-step rule, open-loop schedules, and line-search.
Section~\ref{sec:adaptive_step} presents our adaptive step-size strategy with local Lipschitz constant estimation, detailing our methods for overcoming the slow adaptation of existing pure backtracking schemes.
Section~\ref{sec:algorithms} describes the proposed optimization algorithms in full and provides a convergence analysis that covers strongly convex, convex, quasar-convex, and non-convex settings.
Section~\ref{sec:numerical} presents numerous numerical experiments, evaluating the Adaptive Conditional Gradient and Normalized Steepest Descent algorithms.
Finally, Section~\ref{sec:conclusion} concludes the paper with a summary of findings and possible avenues to continue this research direction.

\subsection*{Notation and Conventions}
We will denote the natural numbers $\mathbb{N}$, which includes $0$, and use $\mathbb{N}^*$ to denote $\mathbb{N}\setminus\{0\}$. For a Hilbert space, the inner product is denoted by $\langle \cdot, \cdot \rangle$. 
The dual norm of $\|\cdot\|$ is denoted by $\|\cdot\|_{\ast}$ and is defined by $\|v\|_{\ast} := \max_{\|x\| \leq 1} \langle v, x \rangle$. We denote the closed ball of radius $\rho > 0$ centered at the origin by $\mathcal{B}(\rho) := \{x : \|x\| \leq \rho\}$ and, when $\rho$ is omitted, we will take by convention the unit-ball.
Moreover, we define functional-value gap at $x \in \mathcal{X}$ as $h(x):=f(x) - f^\star$, where $f^\star=\inf_{x\in\mathcal{X}}f(x)$ is the optimal objective value, and the Frank-Wolfe gap as $\max_{v \in \mathcal{X}} \langle \nabla f(x), x - v \rangle$, a measure of non-stationarity.
For $ x \in \mathbb{R}^n $ and $ q \geq 1 $, let $ \|x\|_q $ denote the $\ell^q$-norm of $ x $, defined as $ \|x\|_q = (\sum_{i=1}^n |x_i|^q)^{1/q} $.
For $ X \in \mathbb{R}^{n\times m}$ and $q\in[1,+\infty]$, let $\|X\|_{\mathcal{S}^q}$ denote the Schatten-$q$ norm of $X$, defined by applying the $\ell^q$ norm to the singular values of $X$.
For $ a \in \mathbb{R} $, define $ \mathrm{sign}(a) $ as $ \mathrm{sign}(a) = 1 $ if $ a > 0 $, $ -1 $ if $ a < 0 $, and $ 0 $ if $ a = 0 $. 
For $ v \in \mathbb{R}^n $, $ \mathrm{sign}(v) \in \{-1, 0, 1\}^n $ denotes the vector of signs of the elements of $ v $. We will denote the indicator function for a set $A$ by $\iota_{A}\colon x\mapsto\begin{cases}0, & x\in A\\ +\infty, & x\not\in A\end{cases}$ and the normal cone $N_A$ to be its subdifferential in the convex sense. Finally, we denote the linear span of a set $V$ by $\mathrm{span}(V)$ and the convex hull of $V$ by $\mathrm{conv}(V)$.
We will denote the Minkowski multiplication of a set $A\subset\mathbb{R}^n$ by a real number $\lambda\in\mathbb{R}$ by $\lambda A := \{\lambda a\colon a\in A\}$. A function $f\colon\mathbb{R}^n\to\mathbb{R}$ will be said to be \emph{strongly convex} with respect to the norm $\|\cdot\|$ if $\exists \mu>0$ such that $f-\mu\|\cdot\|^2$ is convex. Moreover, a differentiable function $f$ will be called \emph{$\eta$-quasar-convex} if $\exists x^\star\in\argmin\limits_{x\in\mathcal{X}}f(x),~ \eta\in]0,1]\colon \forall y\in\mathbb{R}^n, f(x^\star) - f(y) \geq \frac{1}{\eta}\langle \nabla f(y),x^\star-y\rangle$.

\section{Preliminaries}\label{sec:preliminaries}

We being this section with the generic LMO template in \nameref{alg:general}, which was also studied in \cite{pedregosa2020linearly}. This is a very general template that encompasses both the Conditional Gradient algorithm and the Normalized Steepest Descent algorithm, depending on whether $\mathcal{X} = \mathcal{C}$ or $\mathcal{X} = \mathbb{R}^n$.

\begin{algorithm}[hbpt!]
\algname{GenCG}
\caption{Generic Conditional Gradient (GenCG)}\label{alg:general}
\textbf{Input}: Initial point $x^0 \in \mathcal{X}$, compact convex set $\mathcal{C}$ (typically a closed unit-ball $\mathcal{B}_{\lmo}(1)$), step-size sequence $\{t_k\}_{k \geq 0} \subseteq [0, +\infty[$.\\
\textbf{Initialization}:
 Set iteration counter $k\leftarrow 0$.\\
\textbf{Repeat}: For $k = 0, 1, 2, \dots$, perform the following steps:
\begin{itemize}
\item[] 
Compute  $v^k\leftarrow \lmo\limits_{\mathcal{C}}(\nabla f(x^k))$.
\item[]
Set $d^k \leftarrow \begin{cases}
    v^k, & \mbox{if $\mathcal{X}=\mathbb{R}^n$ \quad (Normalized Steepest Descent)}\\
   v^k-x^k. & \mbox{if $\mathcal{X}=\mathcal{C}$ \ \ \quad (Conditional Gradient)}
\end{cases}$
\item[] Set $x^{k+1} \leftarrow x^k + t_k d^k$.
\item[] Update $k\leftarrow k + 1$.
\end{itemize}
\textbf{Stop}: Terminate when convergence criteria is satisfied.
\end{algorithm}
While the LMO is accessible for a multitude of sets $\mathcal{C}$, we will focus especially on the case where $\mathcal{C} = \mathcal{B}_{\lmo}(1)$ is the closed unit-ball in some norm that we will denote $\|\cdot\|_{\lmo}$ for clarity.\footnote{All of the convergence results we will present readily extend to the more general case where $\mathcal{C}$ is a level-set of some general Minkowski gauge function.} In this case, the LMO returns an element of the subdifferential (in the convex sense) of the dual norm $\|\cdot\|_{\lmo\ast}$, i.e.,
\begin{equation*}
	d\in\lmo_{\mathcal{C}}(x)\iff-d\in\partial \|x\|_{\lmo\ast}.
\end{equation*}
Moreover, the output of the LMO is always on the boundary of $\mathcal{C}$, typically an extremal point/atom of the set in a convex-analytic sense, and this endows the output of the LMO with a certain structure based on what the extremal points of $\mathcal{C}$ are. For instance, when $\mathcal{C}$ is the unit-ball for the $\ell^1$ norm (i.e., $\|\cdot\|_{\lmo}=\|\cdot\|_1$), the atoms are the $1$-sparse vectors and the output of the LMO is then typically $1$-sparse for non-zero inputs. In contrast, when $\mathcal{C}$ is the unit-ball for the $\ell^{\infty}$ norm (i.e., $\|\cdot\|_{\lmo}=\|\cdot\|_{\infty}$), the atoms are dense sign vectors and so the output of the LMO is then typically dense for non-zero inputs. This extends to spectral norms on matrices (also known as Schatten $p$-norms); for the nuclear norm or Schatten $1$-norm, the atoms of $\mathcal{C}$ are rank-$1$ matrices and so the output of the LMO in this case is typically rank-$1$ while for the operator norm or Schatten $\infty$-norm, the atoms of $\mathcal{C}$ are orthogonal matrices and the output of the LMO is typically full-rank. 

The LMO is also familiar in the constrained optimization community due to its use in the Frank-Wolfe \cite{frank1956algorithm} or Conditional Gradient \cite{levitin1966constrained} algorithms, as well as in the greedy optimization community where it has been used for greedy and matching pursuit problems. As we will soon describe, the algorithms we consider for solving \eqref{P} are centered around the LMO and are inspired by the Conditional Gradient algorithm. We will consider two distinct cases for \eqref{P}, for which we will develop two corresponding algorithms and analyses: either $\mathcal{X} = \mathcal{C}\subset \mathbb{R}^n$ for some non-empty compact convex set $\mathcal{C}$ (and the problem is constrained) or $\mathcal{X} = \mathbb{R}^n$ (and the problem is unconstrained). While most of our results in the unconstrained setting hold for general non-empty sets $\mathcal{C}$ so long as they are compact and convex, we will sometimes state our results specifically for the case that $\mathcal{C}$ is the closed unit-ball $\mathcal{B}_{\lmo}(1)$ induced by some norm of our choosing $\|\cdot\|_{\lmo}$, as this is often the most relevant for problems in practice and it greatly simplifies the statement of the results in the unconstrained setting.

In both unconstrained and constrained settings, \nameref{alg:general} provides a certificate of criticality for \eqref{P} in the form of the quantity
\begin{equation}
\mathrm{Gap}(x^k):=-\langle \nabla f(x^k), d^k\rangle = \begin{cases}\langle \nabla f(x^k), x^k-v^k\rangle, & \text{(Conditional Gradient)}\\ \|\nabla f(x^k)\|_{\lmo\ast}. & \text{(Normalized Steepest Descent)}
\end{cases}
\end{equation}
Since $\nabla f(x^k)$ and $d^k$ are already computed to run the algorithm, the cost to compute this certificate is cheap and can be used a stopping criterion. For the unconstrained case where $\mathcal{X} = \mathbb{R}^n$, a point $x^\star$ is a critical point if and only if $\|\nabla f(x^\star)\|_{\lmo\ast} = 0$, which is equivalent to $\mathrm{Gap}(x^\star)=0$. For the constrained case where $\mathcal{X} = \mathcal{C}$, criticality is characterized by $-\nabla f(x)\in N_{\mathcal{C}}(x)$ which is equivalent to the so-called Frank-Wolfe gap $\max_{v \in \mathcal{C}} \langle \nabla f(x), x - v \rangle$ being zero, which is also equivalent to $\mathrm{Gap}(x^k)$ being zero and inspires our choice of name for this quantity. We finally remark that the choice of norm $\|\cdot\|_{\lmo}$ for the LMO directly relates to the convergence guarantees we will get, which are all in terms fo $\mathrm{Gap}$.

\subsection{Conditional Gradient Method}

The Conditional Gradient method, originally proposed by Frank and Wolfe \cite{frank1956algorithm} in 1956 for quadratic programming and later generalized by Levitin and Polyak \cite{levitin1966constrained} to broader settings, has become more popular since its introduction to machine learning in \cite{clarkson2010coresets,jaggi2013revisiting}.
Its appeal lies in its simplicity, relying only on gradient evaluations and the LMO instead of projections, which may be more costly for certain choices of $\mathcal{C}$ \cite{combettes2021complexity}. This makes it particularly applicable to many large-scale structured optimization problems that arise in machine learning and beyond \cite{braun2022conditional,jaggi2013revisiting}.
At each iteration $k$, the algorithm approximates $f$ around the current point $x^k$ using the linear model:
\begin{equation}
f(x^k) + \langle \nabla f(x^k), v - x^k \rangle,
\end{equation}
where $v \in \mathcal{C}$. Minimizing this linear model over $\mathcal{C}$ then corresponds to querying the LMO for $\mathcal{C}$ at the point $\nabla f(x^k)$, which gives a direction $v^k \in \mathcal{C}$. The next iterate $x^{k+1}$ is finally constructed as a convex combination of the current point $x^k$ and the oracle-selected direction $v^k$, parameterized by a step-size $t_k\in[0,1]$:
\begin{equation}
x^{k+1} = (1 - t_k)x^k + t_k v^k,
\end{equation}
or equivalently,
\begin{equation}
x^{k+1} = x^k + t_k d^k,
\end{equation}
where $d^k:=v^k-x^k=\lmo_{\mathcal{C}} (\nabla f(x^k))-x^k$. This update ensures that $x^{k+1}$ remains feasible within $\mathcal{C}$, thanks to the convexity of $\mathcal{C}$, as illustrated in Figure~\ref{fig:FW}. Consequently, no projection onto $\mathcal{C}$ is needed, making the Conditional Gradient method a projection-free method.

\begin{figure}[htpb!]
\centering
\begin{tikzpicture}[font=\small]
\coordinate (A) at (2,4); 
\coordinate (B) at (2,2.75); 
\coordinate (C) at (2,1.625); 
\coordinate (D) at (4.25,2.55); 
\coordinate (E) at (0.75,2.5); 
\coordinate (F) at (3.5,4.25); 
\coordinate (G) at (4.75,3.57); 
\coordinate (H) at (2.58,1.75); 
\coordinate (I) at (2.24,3.05); 
\coordinate (J) at (0.5,1.25); 
\coordinate (K) at (5,2.3); 
\draw [thick, fill=gray!20] (E) -- (A) -- (F) -- (G) -- (D) -- (C) -- cycle;
\draw [dashdotdotted] (A) -- (C);
\draw [dotted] (A) -- (H);
\draw [dashed, thick, blue] (J) -- (K);
\draw [->, red, >=stealth, thick] (A) -- (I);
\draw [->, >=stealth, thick] (A) -- (B);
\fill (A) circle (1.2pt) node[above] {$x^k$}; 
\fill (B) circle (1.2pt) node[left] {$x^{k+1}$}; 
\fill (C) circle (1.2pt) node[below] {$v^k$}; 
\node [above right, red] at (I) {$-\nabla f(x^k)$};
\node at (4,3) {$\mathcal{C}$};
\end{tikzpicture}
\caption{Illustration of one step of the Conditional Gradient algorithm. Starting at $x^k$, the gradient $\nabla f(x^k)$ is computed and used to compute $v^k$, an output of the LMO. The next iterate $x^{k+1}$ is obtained by moving along the line segment connecting $x^k$ to $v^k$. The blue dashed line represents a supporting hyperplane to the set $\mathcal{C}$ dictated by its normal direction $-\nabla f(x^k)$.}
\label{fig:FW}
\end{figure}
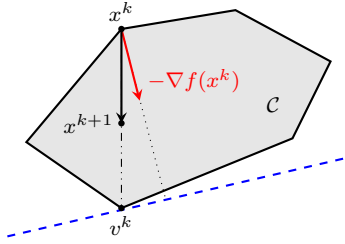

\subsection{Normalized Steepest Descent, Continuous Greedy, and Generalized Matching Pursuit}
While the Conditional Gradient method discussed in the previous part is suited to solving constrained problems, by a slight modification to the definition of the update direction $d^k$ in \nameref{alg:general}, one can arrive at a variant capable of solving unconstrained problems. This algorithm has appeared several times throughout the literature under different names for different cases, e.g., the Normalized Steepest Descent method \cite{boyd2004convex} when $\mathcal{C}$ is the closed unit-ball for some norm $\|\cdot\|_{\lmo}$, the continuous greedy method \cite{calinescu2011maximizing,vondrak2008optimal}, or generalized matching pursuit \cite{locatello2017unified} when $\mathcal{C}$ is the convex hull of some set of atoms or a dictionary to be optimized over. 

\paragraph{Normalized Steepest Descent} Recall that the typical gradient descent algorithm with fixed step-size $t$ can be written,
\begin{equation}
x^{k+1} = x^k - t \nabla f(x^k) = \argmin\limits_{x\in\mathbb{R}^n}f(x^k) + \langle \nabla f(x^k), x-x^k\rangle + \frac{1}{2t}\|x-x^k\|_2^2.
\end{equation}
If one replaces the Euclidean term $\|x-x^k\|_2^2$ by an arbitrary norm $\|x-x^k\|_{\lmo}^2$ then the resulting algorithm is called \emph{Steepest Descent with respect to the norm} $\|\cdot\|_{\lmo}$ \cite{kelner2014almost,nesterov2012efficiency,nutini2015coordinate} and is given by
\begin{equation}
x^{k+1} = x^k + t \|\nabla f(x^k)\|_{\lmo\ast} \lmo\limits_{\mathcal{C}}(\nabla f(x^k)),
\end{equation}
where $\|\cdot\|_{\lmo\ast}$ is the dual norm to $\|\cdot\|_{\lmo}$ and $\mathcal{C}$ is the closed unit-ball induced by the norm $\|\cdot\|_{\lmo}$. When $\|\cdot\|_{\lmo}$ is the Euclidean norm, the dual norm $\|\cdot\|_{\lmo\ast}$ is Euclidean again and Steepest Descent reduces to the familiar gradient descent algorithm. When the function $f$ has a gradient $\nabla f$ that is Lipschitz-continuous with respect to the norm $\|\cdot\|_{\lmo}$, taking $t=\frac{1}{L}$ recovers the analog of the short-step in the unconstrained setting. Indeed, by definition of the dual norm it follows that
\begin{equation}
t \|\nabla f(x^k)\|_{\lmo\ast} = t \max\limits_{v\in\mathcal{C}}\langle \nabla f(x^k), v\rangle = t \left\langle \nabla f(x^k), -\lmo_{\mathcal{C}}(\nabla f(x^k))\right\rangle.
\end{equation}
On the other hand, if we replace $\|x-x^k\|_2^2$ by $\iota_{t \mathcal{C}}(x-x^k)$, then the resulting method can be interpreted as a non-Euclidean trust region method, as noted in \cite{pethick2025training}, with an update to $x^k$ that is always normalized in the norm $\|\cdot\|_{\lmo}$. This method is the so-called \emph{Normalized Steepest Descent with respect to the norm} $\|\cdot\|_{\lmo}$ and is illustrated for an example in Figure~\ref{fig:NSD}. When $t_k$ is chosen like in the open-loop step-size setting, the method is exactly the unconstrained analog to the Conditional Gradient algorithm.

\begin{figure}[htpb!]
\centering
\begin{tikzpicture}[scale=1, font=\small]
\coordinate (A) at (0,0); 
\coordinate (C) at (-3,0); 
\coordinate (B) at (-2,0); 
\coordinate (H) at (-1.38, 0.92); 
\coordinate (I) at (-.6,.4); 
\coordinate (J) at (-3,-1.5); 
\coordinate (K) at (-.6,2.1); 
\draw[thick, fill=gray!20] (-2,0) -- (0,2) -- (2,0) -- (0,-2) -- cycle;
\draw[dashdotdotted] (A) -- (B);
\draw[dotted] (A) -- (H);
\draw[dashed, thick, blue] (J) -- (K);
\draw[->, red, thick] (A) -- (I);
\draw[->, thick] (A) -- (C);
\fill (A) circle (1.2pt) node[below] {$x^k$};
\fill (B) circle (1.2pt) node[above] {$v^k$};
\fill (C) circle (1.2pt) node[above] {$x^{k+1}$};
\node[above right, red] at (I) {$-\nabla f(x^k)$};
\node[rotate=45] at (0.5,-1.2) {$x^k+\mathcal{B}_{\ell^1}(1)$};
\end{tikzpicture}
\caption{One step of the Normalized Steepest Descent method using the $\ell^1$ unit-ball centered at $x^k$. Starting at $x^k$, the gradient $\nabla f(x^k)$ is computed and used to compute $v^k$, an output of the LMO over $\mathcal{C}=\mathcal{B}_{\ell^1}(1)$. The next iterate $x^{k+1}$ is then found by adding $t_kv^k$ to $x^k$. The blue dashed line represents the supporting hyperplane to the set $x^k+\mathcal{B}_{\ell^1}(1)$ dictated by its normal direction $-\nabla f(x^k)$.}
\label{fig:NSD}
\end{figure}
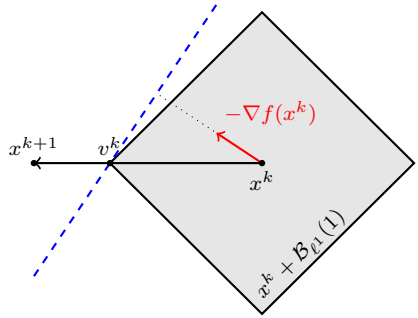


Beyond Normalized Steepest Descent with standard norm-balls, several related methods follow similar LMO-based frameworks. The continuous greedy method \cite{calinescu2011maximizing,vondrak2008optimal} applies this approach to submodular maximization, using a specific step-size schedule $t_k = 1/N$ that ensures the iterates sum to one after $N$ iterations. Generalized matching pursuit \cite{locatello2017unified} extends the framework to arbitrary sets of atoms $\mathcal{A}$, seeking minimizers over $\text{span}(\mathcal{A})$ by querying the LMO over $\text{conv}(\mathcal{A} \cup -\mathcal{A})$. When $\mathcal{A}$ consists of standard basis vectors, this reduces to greedy coordinate descent \cite{locatello2018matching}. 

While these methods demonstrate the versatility of the LMO framework, they typically employ fixed or problem-specific step-sizes. Our work focuses on adaptive step-size strategies for the core Normalized Steepest Descent method with norm-ball constraints, where the local Lipschitz constant estimation can be efficiently computed and provides clear computational benefits.


\subsection{Classic Step-size Strategies}

The convergence and practical performance of LMO-based methods critically depend on step-size selection. We review the main strategies that have been studied for both Conditional Gradient and Normalized Steepest Descent, highlighting connections that are often overlooked in the literature.

\paragraph{Short-Step} The short-step can be found by denoting $t_{\max} = \begin{cases} 1, & \text{(Conditional Gradient)}\\ +\infty, & \text{(Normalized Steepest Descent)}\end{cases}$
and then minimizing a quadratic upper bound on the progress from $x^k$ to $x^{k+1}$ from the descent lemma over $[0,t_{\max}]$:
\begin{equation*}
\argmin\limits_{t \in [0,t_{\max}]} f(x^k) + t \langle \nabla f(x^k),d^k\rangle + \frac{L}{2} t^2\|d^k\|^2.
\end{equation*}
It is given in closed-form by
\begin{equation}
t_k = \min\left\{-\frac{\langle \nabla f(x^k),d^k\rangle}{L\|d^k\|^2},t_{\max}\right\}=\begin{cases}
\min\left\{-\frac{\langle \nabla f(x^k), d^k \rangle}{L\|d^k\|^2}, 1\right\}, & \text{(Conditional Gradient)} \\
\frac{\|\nabla f(x^k)\|_{\lmo\ast}}{L \|d^k\|^2}. & \text{(Normalized Steepest Descent)}
\end{cases}
\end{equation}

For Conditional Gradient, this appears as a quasi-adaptive step-size that depends on the Frank-Wolfe gap and the Lipschitz constant $L$ of the gradient. For Normalized Steepest Descent with $\mathcal{C} = \mathcal{B}(\rho)$ and $\|\cdot\|=\|\cdot\|_{\lmo}$, this is precisely Steepest Descent with fixed step-size $1/L$ in the norm $\|\cdot\|_{\lmo}$. For example, in the $\ell^2$ geometry, the update becomes 
\begin{equation*}
x^{k+1} = x^k + \frac{\|\nabla f(x^k)\|_2}{L} \lmo\limits_{\mathcal{B}_{\ell^2}(1)}(\nabla f(x^k)) = x^k + \frac{\|\nabla f(x^k)\|_2}{L} \left(\frac{-\nabla f(x^k)}{\|\nabla f(x^k)\|_2}\right) = x^k - \frac{1}{L}\nabla f(x^k).
\end{equation*}
This reveals that standard gradient descent with step-size $1/L$ can be interpreted to be Normalized Steepest Descent with the short-step. The short-step was also used in Steepest Descent under the name \emph{sharp operator} in stochastic spectral descent \cite{carlson2015stochastic}.

\paragraph{Open-loop Schedule} An open-loop schedule requires no knowledge of the parameter of the problem (e.g., $L$) and is given, for all $k\in\mathbb{N}$, by
\begin{equation}
t_k = \frac{2}{k+2}.
\end{equation}

This schedule works for both Conditional Gradient and Normalized Steepest Descent without requiring knowledge of $L$ but, on the other hand, it is not adaptive and will remain large even if the iterates are close to a critical point. While widely used in the Conditional Gradient literature \cite{jaggi2013revisiting,wirth2025accelerated}, its applicability to unconstrained optimization is less appreciated. For a ``parameter-free'' gradient descent method, the LMO perspective suggests using open-loop schedules with Normalized Steepest Descent, with no line search or Lipschitz constant estimation needed. Not only is this method parameter-free, it is even affine-invariant \cite{lacoste2013affine}.

\paragraph{Line Search} A line-search schedule maximizes per-iteration progress along the current direction and is given by
\begin{equation}
t_k \in \argmin_{t \in [0,t_{\max}]} f(x^k + td^k).
\end{equation}

Though optimal per iteration when the minimization is performed exactly, the computational cost of multiple function evaluations can outweigh the benefits, motivating adaptive strategies like those we develop in this work.

\section{Adaptive Step-Size Selection with Local Lipschitz Estimation}\label{sec:adaptive_step}

In search of a step-size that is as effective as the line-search schedule without the added computational overhead, one can start by adapting the short-step strategy discussed in Section~\ref{sec:preliminaries}. Normally, this requires knowledge of the global Lipschitz constant $L$ that is often unknown or an overly conservative upper bound of the curvature when known. We now review and develop adaptive strategies that estimate local smoothness properties to achieve better practical performance with this step-size.

\paragraph{Pure Backtracking}

In \cite{pedregosa2020linearly} the authors introduced an elegant backtracking scheme that maintains an estimate $L_k$ of the local Lipschitz constant. At each iteration, the algorithm computes a candidate step-size using $L_k$ in place of $L$ in the short-step formula, then checks the sufficient decrease condition
\begin{equation*}
f(x^k + t_k d^k) \leq f(x^k) + t_k \langle \nabla f(x^k), d^k \rangle + \frac{L_k}{2} t_k^2 \|d^k\|_2^2.
\end{equation*}
We warn the reader that the norm used to compute the LMO (if it is over a unit-ball) and the norm used to perform the backtracking procedure do \emph{not} have to match. The backtracking procedure as in \cite{pedregosa2020linearly} always uses the Euclidean norm $\|\cdot\|_2$ for backtracking, even when the LMO is used over a general compact convex set $\mathcal{C}$.

If the sufficient decrease condition is satisfied, the algorithm accepts the step-size and decreases the estimate of the local Lipschitz constant for the next iteration, $L_{k+1} = \eta L_k$ where $\eta \in ]0,1[$ (typically 0.9). If the sufficient decrease condition is not satisfied, the algorithm increases $L_k$ by a factor $\beta > 1$ (typically 2) and recomputes the step-size, repeating until the condition holds.

While this approach successfully eliminates dependence on the global constant $L$ and adapts to changes in the geometry, it suffers from slow adaptation when moving to regions of lower curvature. Since $L_k$ decreases by only a factor of 0.9 per iteration, transitioning from a high-curvature region (where $L_k \approx 100$, say) to a flat region (where the local Lipschitz constant is $1$) requires over 40 iterations. Figure~\ref{fig:L} illustrates this slow decay on an example problem on video colocalization, showing how the estimate often oscillates between two regimes rather than quickly adapting to the local geometry.

\begin{figure}[htpb!]
\centering
\includegraphics[width=0.6\linewidth]{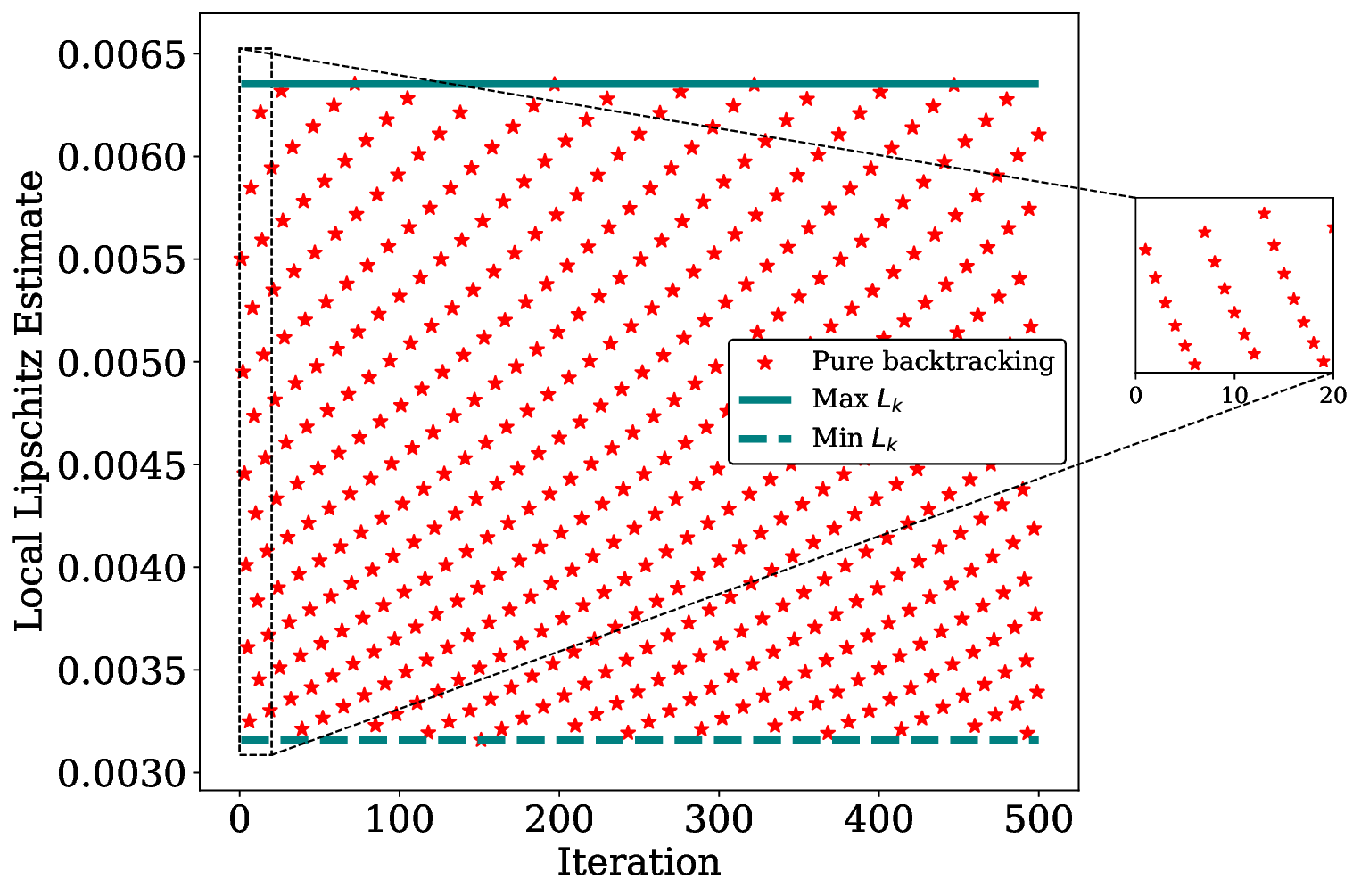}%
\caption{The local Lipschitz constant found at each iteration with the pure backtracking procedure used in \cite{pedregosa2020linearly} is oscillating heavily between two values, roughly $0.0065$ and $0.0030$. }
\label{fig:L}
\end{figure}

\paragraph{Local Lipschitz Estimation via Gradient Differences}

Malitsky and Mishchenko \cite{malitsky2019adaptive} proposed estimating the local Lipschitz constant directly from consecutive gradients by
\begin{equation*}
L_k = \frac{\|\nabla f(x^k) - \nabla f(x^{k-1})\|_2}{\|x^k - x^{k-1}\|_2} + \delta,
\end{equation*}
where $\delta > 0$ is a small constant that we have added for numerical stability. Again, we use Euclidean norms for consistency with the backtracking framework. This estimate directly captures the local rate of gradient change, providing immediate adaptation to the local geometry rather than waiting for slow multiplicative decay. The global Lipschitz property ensures $L_k \leq L + \delta$, and this estimate typically underestimates the global $L$ in smooth regions, which is desirable for exploiting local flatness. This heuristic is also computationally efficient, requiring no additional gradient evaluations beyond those already computed, and completely eliminates the dependence on $L$. However, it does require one to store the previous iteration's gradient $\nabla f(x^{k-1})$ and iterate $x^{k-1}$ in order to compute $L_k$.

\subsection{Adaptive Step-Size Algorithm}

Integrating the heuristic of \cite{malitsky2019adaptive} into \nameref{alg:general} is not so simple when $\mathcal{C}$ is not Euclidean. One could take the maximum between the current $L_k$ at each iteration compared with the previous, so that the estimate becomes a tighter and tighter approximation of the constant $L$ as the algorithm progresses, but this does not take advantage of the local curvature, which might be very low and amenable to large steps. Instead, we will use the heuristic as a warm-start for an optimistic candidate step-size, followed by a modification of the backtracking procedure of \cite{pedregosa2020linearly}, if sufficient decrease is not met. We propose two adaptive step-size scaling strategies, called constant scaling and adjustable scaling, that leverage local smoothness combined with improved backtracking. The algorithm is outlined as \nameref{alg:adaptive-stepsize}. 

\begin{algorithm}[hbpt!]
\algname{AdaStep}
\caption{Adaptive Step-Size (AdaStep)}\label{alg:adaptive-stepsize}
\textbf{Input}: $x^{k-1},x^{k} \in\mathcal{X}$, direction $d^k$, $\delta>0$, and some $\beta>1$.\\
\textbf{General Steps}:
\begin{itemize}
    \item[]\textbf{(Step 1)} 
    Estimate the local Lipschitz constant:
    $$ 
    L_k \leftarrow \frac{\|\nabla f(x^k) - \nabla f(x^{k-1})\|_{\ast}}{\|x^k - x^{k-1}\|} + \delta.
     $$
    \item[]\textbf{(Step 2)}
    Compute the initial step-size:
    $$ 
    t_k \leftarrow 
    \begin{cases}
    \min\left\{-\frac{\langle \nabla f(x^k), d^k \rangle}{L_k\|d^k\|^2}, 1\right\}, & \text{(Conditional Gradient)} \\
    \frac{\|\nabla f(x^k)\|_{\lmo\ast}}{L_k\|d^k\|^2}. & \text{(Normalized Steepest Descent)}
    \end{cases}
     $$
    \item[]\textbf{(Step 3)}
    While the sufficient decrease condition is not satisfied, i.e.,
    $$ 
    f(x^k + t_k d^k) > f(x^k) + t_k \langle \nabla f(x^k), d^k \rangle + \frac{L_k}{2} t_k^2 \|d^k\|^2,
     $$
    update the Lipschitz estimate and step-size:
    $$ 
    L_k \leftarrow \beta L_k \quad\text{and}\quad  
    t_k \leftarrow \begin{cases}
    \min\left\{-\frac{\langle \nabla f(x^k), d^k \rangle}{L_k\|d^k\|^2}, 1\right\}, & \text{(Conditional Gradient)} \\
    \frac{\|\nabla f(x^k)\|_{\lmo\ast}}{L_k\|d^k\|^2}. & \text{(Normalized Steepest Descent)}
    \end{cases}
     $$
    Repeat until the condition is met.
\end{itemize}
\textbf{Output}: Step-size $t_k$.
\end{algorithm}
Upon termination, the updated iterate $x^{k+1} = x^k + t_k d^k$ satisfies the local descent property:
\begin{equation}
f(x^{k+1}) \leq f(x^k) + t_k \langle \nabla f(x^k), d^k \rangle + \frac{L_k}{2} t_k^2 \|d^k\|^2.
\end{equation}
This ensures decrease in either the norm of the gradient $\nabla f$ or the Frank-Wolfe gap without knowing the Lipschitz constant $L$ of the gradient $\nabla f$. It also allows for the step-size to increase if a larger step will provide more progress.

\nameref{alg:adaptive-stepsize} will continue until the termination criterion is met, which differs depending on whether or not the problem is non-convex and whether or not the minimal functional-value $\inf\limits_{x\in\mathcal{X}}f(x)$ is known. When $f$ is convex (or quasar-convex) and $\inf\limits_{x\in\mathcal{X}}f(x)$ is (approximately) known, we will use as a termination criterion the functional-value gap
\begin{equation*}
    f(x^k) - \inf\limits_{x\in\mathcal{X}}f(x).
\end{equation*}
Meanwhile when $f$ is non-convex or if the $\inf\limits_{x\in\mathcal{X}}f(x)$ is not known, we will use the general LMO gap discussed in the beginning of Section~\ref{sec:preliminaries},
\begin{equation*}
    \mathrm{Gap}(x^k):=-\langle \nabla f(x^k),d^k\rangle = \begin{cases} \langle \nabla f(x^k),x^k-v^k\rangle, & \text{(Conditional Gradient)}\\ \|\nabla f(x^k)\|_{\lmo\ast}, & \text{(Normalized Steepest Descent)}\end{cases}
\end{equation*}
as the termination criterion in \nameref{alg:general}. We recall that this quantity is relatively cheap to compute, since $d^k$ and $\nabla f(x^k)$ are both already computed to run the algorithm.

\subsubsection{Adaptive Step-Size with Constant Scaling}

The constant scaling method enhances \nameref{alg:adaptive-stepsize} by applying an initial scaling to $L_k$ to promote larger step-sizes. After estimating $L_k$ in step 1, we adjust by scaling:
$$ 
L_k \leftarrow \gamma L_k,
$$
where $\gamma \in ]0,1]$ is a fixed scaling factor. This reduction in $L_k$ potentially increases $t_k$. The algorithm then proceeds to step 3, checking the sufficient decrease condition. If the condition is not met, $L_k$ is iteratively increased by the backtracking factor $\beta\in]1,+\infty[$,
$$ 
L_k \leftarrow \beta L_k,
$$
and $t_k$ is recomputed. This method, termed \emph{adaptive constant scaling}, uses fixed parameters $\gamma$ and $\beta$ to balance aggressive step-size selection with robustness.

\subsubsection{Adaptive Step-Size with Adjustable Scaling}

The adjustable scaling method builds on the constant scaling approach by dynamically updating the scaling factor $\gamma$ to adapt to the function's local curvature. After estimating $L_k$ and applying the initial scaling $L_k \leftarrow \gamma L_k$, the algorithm proceeds as in the constant scaling method. However, $\gamma$ is periodically updated, e.g., every $r$ iterations, based on the total number of backtracking steps in the previous period:
\begin{itemize}
    \item \textbf{Conservative}: If no backtracking occurs, indicating a conservative $L_k$, reduce $\gamma$:
    $$ 
    \gamma \leftarrow 0.9 \gamma.
     $$
    \item \textbf{Aggressive}: If backtracking exceeds $r$ iterations, indicating an aggressive $L_k$, increase $\gamma$:
    $$ 
    \gamma \leftarrow 1.1 \gamma.
     $$
    \item \textbf{Balanced}: If backtracking occurs but totals $r$ or fewer steps, $\gamma$ remains unchanged, as the algorithm is well-calibrated.
\end{itemize}
This \emph{adaptive adjustable scaling} method enhances flexibility by adjusting $\gamma$ to match the function's smoothness, potentially accelerating convergence in problems with varying curvature. Both scaling methods are computationally efficient, requiring only scalar updates. In our experiments, we used consistent initial parameters for both methods: $r=10$, $\gamma = 1/4$, $\delta = 10^{-10}$, and $\beta = 2$.

\section{Algorithm and Convergence Analysis}\label{sec:algorithms}

We present here the \emph{Adaptive Conditional Gradient Descent} algorithm (\nameref{alg:ACGD}) which is simply the template from \nameref{alg:general} with our adaptive step-size strategy \nameref{alg:adaptive-stepsize}. Afterwards, we give some remarks followed by its convergence analysis under a variety of assumptions on the regularity of $f$, ranging from non-convex to strongly convex.

\subsection{Algorithm}

\begin{algorithm}[hbpt!]
\algname{ACGD}
\caption{Adaptive Conditional Gradient Descent (ACGD)}\label{alg:ACGD}
\textbf{Input}: Compact convex set $\mathcal{C}$ (typically a closed unit-ball $\mathcal{B}_{\lmo}(1)$), feasible set $\mathcal{X}$ (either $\mathbb{R}^n$ or $\mathcal{C}$), parameters 
$\gamma \in ]0,1]$, $\delta>0$, $t_{\max} > 0$ and $\beta>1$.\\
\textbf{Initialization}:
Set $k = 0$ and choose $x^{k-1}, x^k \in \mathcal{X}$.\\
\textbf{General Steps}: 
\begin{itemize}
    \item[] \textbf{(Step 1)} 
    Compute  
    $$v^k\leftarrow \lmo_{\mathcal{C}}(\nabla f(x^k)).$$
    \item[] \textbf{(Step 2)} 
    Set descent direction 
    $$d^k \leftarrow \begin{cases}
    v^k, & \mbox{if $\mathcal{X}=\mathbb{R}^n$}\\
   v^k-x^k. & \mbox{if $\mathcal{X}=\mathcal{C}$}
    \end{cases}$$
    \item[]\textbf{(Step 3)} 
    Estimate the local Lipschitz constant:
    $$ 
    L_k \leftarrow \frac{\|\nabla f(x^k) - \nabla f(x^{k-1})\|_{\ast}}{\|x^k - x^{k-1}\|} + \delta.
    $$
    \item[]\textbf{(Step 4)} Scale local Lipschitz 
    $$ 
    L_k \leftarrow \gamma L_k. 
    $$
    \item[]\textbf{(Step 5)}
    Compute the initial step-size:
    $$ 
    t_k \leftarrow 
     \min\left\{-\frac{\langle \nabla f(x^k),d^k\rangle}{L_k\|d^k\|^2},t_{\max}\right\}.
     $$
    \item[]\textbf{(Step 6)}
    While 
     $ 
    f(x^k + t_k d^k) > f(x^k) + t_k \langle \nabla f(x^k), d^k \rangle + \frac{L_k}{2} t_k^2 \|d^k\|^2,
    $ 
    do:
    $$ 
    L_k \leftarrow \beta L_k \quad\text{and}\quad  
    t_k \leftarrow  \min\left\{-\frac{\langle \nabla f(x^k),d^k\rangle}{L_k\|d^k\|^2},t_{\max}\right\}.
     $$
    \item[] \textbf{(Step 7)} Update iterate 
    $$x^{k+1} \leftarrow x^k + t_k d^k  \quad \text{and} \quad  k\leftarrow k + 1.$$
\end{itemize}
\textbf{Stop}: Terminate when a convergence criterion is met.
\end{algorithm}

The backtracking procedure used terminates in a finite number of iterations because the estimated local Lipschitz constant $L_k$ is bounded above by $L + \delta$, where $L$ is the global Lipschitz constant of $\nabla f$. Moreover, the exponential growth of $L_k$ (via $\beta > 1$) ensures that $L_k$ will eventually exceed $L+\delta$, guaranteeing termination.

\begin{remark}
    In the unconstrained case ($\mathcal{X}=\mathbb{R}^n$), the norm $\|\cdot\|_{\lmo}$ used in Steps 3 to 6 does \emph{not} have to be the same norm used to define the closed unit-ball $\mathcal{B}_{\lmo}(1)$ taken for $\mathcal{C}$. This means it is possible to compute the normalized steepest descent direction with respect to an arbitrary norm $\|\cdot\|_{\lmo}$ (Step 1) in combination with backtracking in a \emph{different} norm $\|\cdot\|$. As we will explain the numerical section, it turns out that doing the backtracking procedure in the $\ell^2$ norm works well even when $\mathcal{C}$ is not the $\ell^2$ closed unit-ball, although in the theoretical section we will also show that the convergence guarantee is for $\|\nabla f(x^k)\|_{\lmo\ast}$.
\qed
\end{remark}

After termination of Step 6 in \nameref{alg:ACGD}, the updated iterate satisfies the following local descent property
\begin{equation}\label{eq:local_descent}
f(x^k + t_k d^k) \leq f(x^k) + t_k \langle \nabla f(x^k), d^k \rangle + \frac{L_k}{2} t_k^2 \|d^k\|^2,
\end{equation}
which is the key inequality that we will use to prove the convergence rates.

\begin{remark}
In our adaptive step-size method, we define the update direction as $v^k = \lmo_{\mathcal{B}_{\lmo}(1)}(\nabla f(x^k))$, where the LMO is taken over the closed unit-ball $\mathcal{B}_{\lmo}(1)$ induced by some norm that we denote $\|\cdot\|_{\lmo}$. Formally, $\lmo_{\mathcal{C}}(v) = \argmin_{\|s\|_{\lmo} \leq 1} \langle s, v \rangle$, which guarantees $\|v^k\|_{\lmo} = 1$. To estimate the local Lipschitz constant $L_k$ efficiently, we could use the same norm in both the $\lmo$ computation and the backtracking, i.e., $\|\cdot\|_{\lmo}=\|\cdot\|$. In that case, the calculation for $L_k$ simplifies as follows
\begin{equation*}
L_k = \frac{\|\nabla f(x^k) - \nabla f(x^{k-1})\|_{\ast}}{\|x^k - x^{k-1}\|}+ \delta = \frac{\|\nabla f(x^k) - \nabla f(x^{k-1})\|_{\ast}}{t_{k-1}}+ \delta, 
\end{equation*}
because $\|x^k - x^{k-1}\| = \|t_{k-1}v^{k-1}\| = t_{k-1}$ since, under this assumption, $\|v^{k-1}\|=\|v^{k-1}\|_{\lmo}=1$.
This formulation reduces computational cost by leveraging the fact that the output of the LMO has a fixed norm, so the quantity $\|x^k - x^{k-1}\|$ can be known in advance, eliminating the need to compute distances between iterates explicitly and thus avoiding having to store $x^{k-1}$. However, we will see that it is eventually useful to allow the norms used to calculate $L_k$ and the norm used to calculate the LMO to be \emph{decoupled}, since estimating the local Lipschitz constant using $\ell^2$ norms performs well in practice with directions coming from non-Euclidean LMOs.
\qed
\end{remark}

%

\begin{corollary}\label{L1}
Let $k\in\mathbb{N}$ and consider the step-size $t_k$, iterate $x_k$, and direction $d^k$ generated by \nameref{alg:ACGD}. Then,
\begin{equation*}
t_k \langle \nabla f(x^k), d^k \rangle +\frac{L_k}{2} t_k^2\|d^k\|^2\leq 0.
\end{equation*}
\end{corollary}
\begin{proof}
For $t_k=0$ the claim obviously holds so consider $t_k>0$.

Starting from the inequality:
For all $k\in\mathbb{N}$, by the construction of the step-size $t_k$, it holds
\begin{equation*}
0 < t_k \leq -\dfrac{\langle \nabla f(x^k), d^k \rangle}{L_k \| d^k \|^2} \leq -2\dfrac{\langle \nabla f(x^k), d^k \rangle}{L_k \| d^k \|^2}. 
\end{equation*}
Multiplying both sides by $\frac{L_k\|d^k\|^2}{2}$ and rearranging, we find
\begin{equation*}
    \langle \nabla f(x^k),d^k\rangle + \frac{L_k}{2}t_k\|d^k\|^2\leq 0,
\end{equation*}
which, after multiplying by $t_k\geq 0$, gives the desired claim.
\qed 
\end{proof}

\begin{corollary}
Let $k\in\mathbb{N}$ and consider consecutive iterates $x^k$ and $x^{k+1}$ generated by \nameref{alg:ACGD}. Then, $f(x^{k+1}) \leq f(x^k)$, i.e., the functional-values are non-increasing.
\end{corollary}
\begin{proof}
From Corollary \ref{L1}, we know that:
$$
t_k \langle \nabla f(x^k), d^k \rangle +\frac{L_k}{2} t_k^2\|d^k\|^2\leq 0.
$$
Thus, it follows that $f(x^k+t_k d^k) \leq f(x^k)$.
This ensures that the functional-values are non-increasing, guaranteeing descent at each iteration.
\qed
\end{proof}

\begin{remark}
The main difference between our adaptive step-size method and the step-size method proposed in \cite{pedregosa2020linearly} lies in the handling of the Lipschitz constant estimate. Our method computes $L_k$ directly from gradient differences at each iteration and then modifies it if needed, allowing it to  adapt more quickly to the local curvature of a region. 
In contrast, the step-size method proposed in \cite{pedregosa2020linearly} initializes $L_{-1}$ heuristically as  
$$
\frac{\|\nabla f(x^0) - \nabla f(x^0 + \epsilon d^0)\|_{2}}{{\epsilon \|d^0\|_2}},
$$ 
where $\epsilon$ is a small constant, and updates it iteratively by increasing it during backtracking or decreasing it when the sufficient decrease condition is satisfied. This approach avoids recomputing the Lipschitz estimate from scratch but may result in conservative step-sizes, particularly in iterations where the sufficient decrease condition does not require adjustment through the backtracking loop. Note also that, technically speaking, there is no guarantee that this heuristic initialization for $L_{-1}$ is non-zero.
\qed
\end{remark}

\subsection{Unconstrained Problems ($\mathcal{X} = \mathbb{R}^n$)}

In the subsection, we will always assume that $\mathcal{X}=\mathbb{R}^n$, so that \eqref{P} is unconstrained. We start with a lemma that characterizes how much the functional-values must decrease across iterations.

\begin{lemma}\label{lem:unconstrained_local_descent}
    Let $k\in\mathbb{N}$ and consider the step-size $t_k$, iterate $x^k$, and direction $d^k$ generated by \nameref{alg:ACGD}. Then,
    \begin{equation*}
        f(x^{k+1}) \leq f(x^k) -\frac{1}{2 L_k\|d^k\|^2}\|\nabla f(x^k)\|_{\lmo\ast}^2.
    \end{equation*}
\end{lemma}
\begin{proof}
    By the local descent property of \nameref{alg:ACGD} given in \eqref{eq:local_descent} and noting that $d^k = \lmo_{\mathcal{C}}(\nabla f(x^k))$ in the unconstrained case, we get
    \begin{equation*}
        f(x^{k+1})\leq f(x^k) - t_k\|\nabla f(x^k)\|_{\lmo\ast} + \frac{L_k}{2}t_k^2\|d^k\|^2.
    \end{equation*}
    By definition of $t_k$ in \nameref{alg:ACGD}, $t_k= \frac{\|\nabla f(x^k)\|_{\lmo\ast}}{L_k\|d^k\|^2}$, which leads to
    \begin{equation*}
        \begin{aligned}
            f(x^{k+1})&\leq f(x^k) - \left( \frac{\|\nabla f(x^k)\|_{\lmo\ast}}{L_k\|d^k\|^2}\right) \|\nabla f(x^k)\|_{\lmo\ast} + \frac{L_k }{2} \left( \frac{\|\nabla f(x^k)\|_{\lmo\ast}}{L_k\|d^k\|^2}\right)^2\|d^k\|^2 \\
            &
            = f(x^k) -\frac{1}{2 L_k\|d^k\|^2}\|\nabla f(x^k)\|_{\lmo\ast}^2.
        \end{aligned}
    \end{equation*}
\qed
\end{proof}

The following theorem establishes the convergence rate (to a critical point) for \nameref{alg:ACGD} in the unconstrained setting under the assumption that $f$ is bounded from below by $\inf\limits_{x\in\mathbb{R}^n}f(x) > -\infty$.
\begin{theorem}[Unconstrained Non-convex Convergence Rate]\label{thm:steepest_nonconvex}
Let $\mathcal{C}=\mathcal{B}_{\lmo}(1)$ be the closed unit-ball for some norm  $\|\cdot\|_{\lmo}$ and let $f: \mathbb{R}^n \to \mathbb{R}$ be a continuously differentiable function, bounded from below by $f^\star>-\infty$, with a gradient that is Lipschitz-continuous in the norm $\|\cdot\|_{\ast}$ with constant $L>0$. Let $\{x^k\}_{k\in\mathbb{N}}$ be the sequence of iterates generated by \nameref{alg:ACGD}. Then, for any $N \in\mathbb{N}^*$, 
\begin{equation*}
\min\limits_{0\leq k \leq N}\|\nabla f(x^k)\|_{\lmo\ast}^2 \leq \frac{2\beta (L+\delta)\zeta^2(f(x^0)-f^\star)}{N+1},
\end{equation*}
where $\zeta\in]0,+\infty[$ is the norm-equivalence constant such that, for all $x\in\mathbb{R}^n$, $\|x\|\leq \zeta\|x\|_{\lmo}$.
\end{theorem}
\begin{proof}
Using the local descent property given by Lemma~\ref{lem:unconstrained_local_descent} for $f$, it holds for all $k\in\mathbb{N}$,
\begin{equation}
f(x^{k+1}) \leq f(x^k) -  \frac{1}{2 L_k\|d^k\|^2}\|\nabla f(x^k)\|_{\lmo\ast}^2.
\end{equation}
Letting $N\in\mathbb{N}^*$ and summing from $k=0$ to $N$ yields
\begin{equation*}
\begin{aligned}
    \sum\limits_{k=0}^{N}\|\nabla f(x^k)\|_{\lmo\ast}^2 &\leq \sum\limits_{k=0}^{N} 2 L_k \|d^k\|^2\left(f(x^{k})-f(x^{k+1})\right)\\
    & \leq 2 \beta (L+\delta)\zeta^2\sum\limits_{k=0}^{N}\left(f(x^{k})-f(x^{k+1})\right) \\
    & \leq 2 \beta (L+\delta)\zeta^2\left(f(x^0)-f^\star\right),
\end{aligned}
\end{equation*}
where $\zeta\in]0,+\infty[$ is the norm-equivalence constant such that, for all $x\in\mathbb{R}^n$, $\|x\|\leq \zeta \|x\|_{\lmo}$.
Lower bounding by taking the minimum over $k\in\{0,\ldots,N\}$ finally gives 
\begin{equation*}
    \min\limits_{0\leq k \leq N}\|\nabla f(x^k)\|_{\lmo\ast}^2 \leq \frac{2\beta (L+\delta)\zeta^2(f(x^0)-f^\star)}{N+1}.
\end{equation*}
\qed 
\end{proof}
\begin{remark}
    The previous theorem gives a convergence guarantee in terms of $\|\nabla f(x^k)\|_{\lmo\ast}$ where the norm $\|\cdot\|_{\lmo\ast}$ is the dual norm to $\|\cdot\|_{\lmo}$, the one used to compute $v^k$ via the LMO. This norm does \emph{not} have to be the same norm that is used in the backtracking procedure. In practice, we find that using the $\ell^2$ for the backtracking procedure works well even when the LMO is computed with respect to a different norm (e.g., $\ell^1$ or $\ell^\infty$).
\qed
\end{remark}

The next theorem guarantees a stronger convergence rate (to the optimal functional-value) for \nameref{alg:ACGD} when minimizing a quasar-convex function. It requires the additional hypothesis that the sub-level sets of $f$ are bounded. Note that a sub-level set for a continuous function is a closed set, so that the boundedness assumption ensures compactness of sub-level sets of $f$.

\begin{theorem}[Unconstrained Quasar-convex Convergence Rate]\label{thm:steepest_convex}
Let $\mathcal{C} = \mathcal{B}_{\lmo}(1)$ be a closed unit-ball for some norm $\|\cdot\|_{\lmo}$ and let $f: \mathbb{R}^n \to \mathbb{R}$ be a continuously differentiable function with a gradient that is Lipschitz-continuous in the norm $\|\cdot\|_{\ast}$ with constant $L>0$. Assume also that $f$ is $\eta$-quasar-convex and that the level set $\mathcal{S}:=\{x\in\mathbb{R}^n:f(x) \leq f(x^0)\}$ associated with the initialization $x^0$ is bounded and let $\{x^k\}_{k\in\mathbb{N}}$ the sequence of iterates generated by \nameref{alg:ACGD}. Then, for any $N \in\mathbb{N}^*$, 
\begin{equation}\label{CCC}
f(x^N)-f^\star \leq \frac{2 \beta(L+\delta)R^2}{(\frac{2 \beta(L+\delta)R^2}{f(x^0)-f^\star}) + N\eta^2}.
\end{equation}
\end{theorem}
\begin{proof}
First we note that, for any $N \in\mathbb{N}^*$, if $f(x^N)=f^{\star}$ then inequality \eqref{CCC} clearly holds. So, we assume that $f(x^N)\neq f^{\star}$, which means that, for $k\in\{0,\ldots,N\}$, we have $h_k>0$.

From the descent property in Lemma~\ref{lem:unconstrained_local_descent} and the definition of the step-size $t_k$, it holds for all $k\in\{0,\ldots,N-1\}$,
\begin{equation*}
f(x^{k+1}) \leq f(x^k)-\frac{\|\nabla f(x^k)\|_{\lmo\ast}^2}{2 L_k \|d^k\|^2}.
\end{equation*}
Define $h_k:=f(x^k)-f^{\star} \geq 0$; subtracting $f^\star$ from both sides of the above inequality gives
\begin{equation*}
h_{k+1} \leq h_k-\frac{\|\nabla f(x^k)\|_{\lmo\ast}^2}{2 L_k\|d^k\|^2}.
\end{equation*}
Since each $x^k \in \mathcal{S}$ and $\mathcal{S}$ is bounded by assumption, define
\begin{equation*}
R:=\max\left\{\max_{ x \in \mathcal{S}} \|x-x^{\star}\|_{\lmo},1 \right\}<\infty.
\end{equation*}
By $\eta$-quasar-convexity of $f$, it follows that
\begin{equation*}
h_k \leq \frac{1}{\eta}\langle\nabla f(x^k), x^k-x^{\star}\rangle \leq\frac{1}{\eta}\|\nabla f(x^k)\|_{\lmo\ast}\|x^k-x^{\star}\|_{\lmo} \Rightarrow
\|\nabla f(x^k)\|_{\lmo\ast} \geq \frac{h_k}{\|x^k-x^{\star}\|_{\lmo}}
 \geq h_k\frac{\eta}{R},
\end{equation*}
since  $x^k \in \mathcal{S}$, i.e., $\left\|x^k-x^{\star}\right\|_{\lmo} \leq R$ for all $k\in\mathbb{N}$. It therefore follows that
\begin{equation*}
    h_{k+1} \leq h_k-\frac{\|\nabla f(x^k)\|_{\lmo\ast}^2}{2 L_k\|d^k\|^2}\leq h_k-\frac{1}{2 L_k\|d^k\|^2} \left(h_k\frac{\eta}{R}\right)^2
\leq h_k-\frac{h_k^2\eta^2}{2 \beta(L+\delta)\zeta^2R^2},
\end{equation*}
since $L_k \leq \beta (L + \delta)$ and $\zeta\in]0,+\infty[$ is the norm-equivalence constant such that, for all $x\in\mathbb{R}^n, \|x\|\leq \zeta\|x\|_{\lmo}$.
Manipulating the above we find
\begin{equation*}
\frac{h_{k+1}}{h_{k+1}h_k} \leq \frac{h_k - \frac{h_k^2\eta^2}{2\beta(L+\delta)\zeta^2R^2}}{h_{k+1}h_k}
\implies
\frac{1}{h_k} \leq \frac{1}{h_{k+1}} - \frac{h_k\eta^2}{2\beta(L+\delta)\zeta^2R^2 h_{k+1}}.
\end{equation*}
Since $h_{k+1} \leq h_k$, we know that $\frac{h_k}{h_{k+1}} \geq 1$ and thus
\begin{equation*}
-\frac{h_k\eta^2}{2\beta(L+\delta)\zeta^2R^2 h_{k+1}} \leq -\frac{\eta^2}{2\beta(L+\delta)\zeta^2R^2}.
\end{equation*}
Substituting this into our earlier equation, we find that
\begin{equation*}
\frac{1}{h_k} \leq \frac{1}{h_{k+1}} - \frac{\eta^2}{2\beta(L+\delta)\zeta^2R^2},
\end{equation*}
which we rewrite as
\begin{equation*}
\frac{1}{h_{k+1}} \geq \frac{1}{h_k} + \frac{\eta^2}{2\beta(L+\delta)\zeta^2R^2}.
\end{equation*}
Summing this telescoping inequality from $k = 0$ to $N-1$,
\begin{equation*}
\frac{1}{h_{N}} \geq \frac{1}{h_0} + \frac{N\eta^2}{2\beta(L+\delta)\zeta^2R^2}.
\end{equation*}
Inverting and simplifying gives
\begin{equation*}
h_{N} \leq \frac{2 \beta(L+\delta)\zeta^2R^2}{\left(\frac{2 \beta(L+\delta)\zeta^2R^2}{h_0}\right) + N\eta^2},
\end{equation*}
which completes the proof.
\qed 
\end{proof}

\begin{theorem}[Unconstrained Strongly Convex Convergence Rate]\label{thm:steepest_strongly_convex}
Let $\mathcal{C} = \mathcal{B}_{\lmo}(1)$ be a closed ball of radius $\rho>0$ for some norm $\|\cdot\|_{\lmo}$ and let $f: \mathbb{R}^n \to \mathbb{R}$ be continuously differentiable with a gradient that is Lipschitz-continuous in the norm $\|\cdot\|_{\ast}$ with constant $L>0$. Assume also that $f$ is $\mu$-strongly convex with respect to $\|\cdot\|_{\lmo}$ and let $\{x^k\}_{k\in\mathbb{N}}$ be the sequence of iterates generated by \nameref{alg:ACGD}. Then, for any $N \in \mathbb{N}^*$,
\begin{equation}
    f(x^N) - f^\star \leq \left(1 - \frac{\mu}{\beta (L + \delta)\zeta^2}\right)^N (f(x^0) - f^\star),
\end{equation}
where $\zeta\in]0,+\infty[$ is the norm-equivalence constant such that, for all $x\in\mathbb{R}^n, \|x\|\leq\zeta\|x\|_{\lmo}$.
\end{theorem}
\begin{proof}
From the local descent property established in Lemma~\ref{lem:unconstrained_local_descent} for $f$ with the definition of the step-size $t_k$, we have, for all $k\in\mathbb{N}$,
\begin{equation}\label{eq:unconstrained_local_descent}
f(x^{k+1}) \leq f(x^k) - \frac{\|\nabla f(x^k)\|_{\lmo\ast}^2}{2 L_k\|d^k\|^2}.
\end{equation}
By the assumed strong convexity of $f$ with constant $\mu>0$, we have the Polyak-\L ojasiewicz inequality (or more precisely the Kurdyka-\L ojasiewicz inequality with exponent $1/2$)
\begin{equation*}
\|\nabla f(x^k)\|_{\lmo\ast}^2 \geq 2\mu (f(x^k) - f^\star).
\end{equation*}
Applying the above inequality to \eqref{eq:unconstrained_local_descent}, we find
\begin{equation*}
f(x^{k+1}) \leq f(x^k) - \frac{\mu (f(x^k)-f^\star)}{L_k\|d^k\|^2}.
\end{equation*}
Subtracting $f^\star$ from both sides and letting $h_k := f(x^k) - f^\star \geq 0$ we can rewrite to get
\begin{equation*}
h_{k+1} \leq h_k - \frac{\mu h_k}{L_k\|d^k\|^2} = \left(1 - \frac{\mu}{L_k\|d^k\|^2}\right)h_k.
\end{equation*}
From the backtracking procedure (c.f., \nameref{alg:adaptive-stepsize}), we know that $L_k \leq \beta (L + \delta)$ upon termination, i.e.,
\begin{equation*}
1 - \frac{\mu}{L_k\|d^k\|^2} \geq 1 - \frac{\mu}{\beta (L + \delta)\zeta^2},
\end{equation*}
where $\zeta\in]0,+\infty[$ is the norm-equivalence constant such that, for all $x\in\mathbb{R}^n, \|x\|\leq\zeta\|x\|_{\lmo}$. Thus,
\begin{equation*}
h_{k+1} \leq \left(1 - \frac{\mu}{\beta (L + \delta)\zeta^2}\right)h_k.
\end{equation*}
Applying this recursively from $k=0$ to $N-1$ yields
\begin{equation*}
h_{N} \leq \left(1 - \frac{\mu}{\beta (L + \delta)\zeta^2}\right)^{N} h_0,
\end{equation*}
which completes the proof.
\qed
\end{proof}


\subsection{Constrained Problems ($\mathcal{X} = \mathcal{C}$)}

In this subsection, we will always assume that $\mathcal{X} = \mathcal{C}$ is some non-empty compact convex set. This makes \eqref{P} a constrained problem that can be solved using the Conditional Gradient algorithm with our adaptive step-size, i.e., \nameref{alg:ACGD}. 

The next theorem establishes a convergence rate for the Frank-Wolfe gaps of \nameref{alg:ACGD} when minimizing a smooth function over a non-empty compact convex set to a stationary point in the general non-convex setting. The rate is given on $\mathrm{Gap}(x^k)$, which we recall is the Frank-Wolfe gap when $\mathcal{X}=\mathcal{C}$ and which, when bounded above by $\epsilon$, certifies $\epsilon$-criticality of $x^k$ for \eqref{P}. We make use of the fact that $\mathcal{C}$ is compact and thus admits a finite diameter $D = \max\limits_{x,y\in\mathcal{C}}\|x-y\|<+\infty$.
\begin{theorem}[Constrained Non-convex Convergence Rate]\label{thm:frank_nonconvex}
Let $\mathcal{C} \subset \mathbb{R}^n$ be a non-empty compact convex set and let $f: \mathbb{R}^n \to \mathbb{R}$ be continuously differentiable with $\{x^k\}_{k\in\mathbb{N}}$ the sequence of iterates generated by \nameref{alg:ACGD}. Then, for any $N\in\mathbb{N}^*$,
\begin{equation*}
\min\limits_{0 \leq k \leq N} \, \mathrm{Gap}(x^k) \leq \sqrt{\frac{2C(f(x^0)-f^\star)}{N+1}},
\end{equation*}
where $C := \left(\max\left\{\beta (L+\delta), 1, \max_{x\in \mathcal{C}}\|\nabla f(x)\| \right\} \cdot (\max\{D, 1\})^2\right)$.
\end{theorem}
\begin{proof}
 Using the local descent property for $f$ given by \eqref{eq:local_descent} we have, for all $k\in\mathbb{N}$,
\begin{equation*}
f(x^{k+1}) \leq f(x^k) + t_k \langle \nabla f(x^k), d^k \rangle + \frac{L_k }{2} t_k^2\|d^k\|^2=f(x^k) - t_k \mathrm{Gap}(x^k) + \frac{L_k }{2} t_k^2\|d^k\|^2
\end{equation*}
which can be rearranged to get a bound on the gap,
\begin{equation}\label{eq:gapbound}
    t_k\mathrm{Gap}(x^k) \leq f(x^k)-f(x^{k+1}) + \frac{L_k}{2}t_k^2\|d^k\|^2
\end{equation}
By design, we have
$t_k= \min\left\{ -\frac{\langle \nabla f(x^k), d^k \rangle}{L_k \|d^k\|^2}, 1 \right\}= \min\left\{ \frac{\mathrm{Gap}(x^k)}{L_k \|d^k\|^2}, 1 \right\}$, which gives two cases that we examine separately.
\\
\textbf{Case one:}
 $t_k=\frac{\mathrm{Gap}(x^k)}{L_k \|d^k\|^2}\leq 1$. In this case we find 
Starting with \eqref{eq:gapbound} and using the definition of $t_k$ we find
 \begin{equation*}
     \frac{(\mathrm{Gap}(x^k))^2}{L_k \|d^k\|^2} \leq f(x^k) - f(x^{k+1})  + \frac{(\mathrm{Gap}(x^k))^2 }{2 L_k \|d^k\|^2}
 \end{equation*}
 which is equivalent to
 \begin{equation*}
     \frac{(\mathrm{Gap}(x^k))^2}{2L_k \|d^k\|^2} \leq f(x^k) - f(x^{k+1}).
 \end{equation*}
\textbf{Case two:}
 $t_k=1, \tfrac{\mathrm{Gap}(x^k)}{L_k \|d^k\|^2}\geq 1$. In this case, starting from \eqref{eq:gapbound} and using $t_k=1$ we find
\begin{equation*}
\mathrm{Gap}(x^k) \leq f(x^k) - f(x^{k+1})  + \frac{L_k }{2} \|d^k\|^2 \leq f(x^k) - f(x^{k+1})  + \frac{\mathrm{Gap}(x^k) }{2}.
\end{equation*}
Rearranging this yields
\begin{equation*}
 \frac{\mathrm{Gap}(x^k)}{2} \leq f(x^k) - f(x^{k+1}).
\end{equation*}
Combining both cases we have, for all $k\in\mathbb{N}$,
\begin{equation*}
\min\left\{\frac{\mathrm{Gap}(x^k)}{2}, \frac{(\mathrm{Gap}(x^k))^2}{2L_k \|d^k\|^2} \right\}\leq f(x^k) - f(x^{k+1}).
\end{equation*}
Set $C := \left(\max\left\{\beta (L+\delta), 1, \max_{x\in \mathcal{C}}\|\nabla f(x)\| \right\} \cdot (\max\{D, 1\})^2\right)>0$, then 
$$
\min\left\{\frac{\mathrm{Gap}(x^k)}{2}, \frac{(\mathrm{Gap}(x^k))^2}{2C} \right\}\leq f(x^k) - f(x^{k+1}).
$$
Now, we claim that $\min\left\{\frac{\mathrm{Gap}(x^k)}{2}, \frac{(\mathrm{Gap}(x^k))^2}{2C} \right\}=\frac{(\mathrm{Gap}(x^k))^2}{2C} $, i.e., that $\frac{\mathrm{Gap}(x^k)}{2}\geq \frac{(\mathrm{Gap}(x^k))^2}{2C}$ since
\begin{equation*}
\frac{\mathrm{Gap}(x^k)}{2}-\frac{(\mathrm{Gap}(x^k))^2}{2C}=\underbrace{\frac{\mathrm{Gap}(x^k)}{2C}}_{\geq0}\underbrace{(C-\mathrm{Gap}(x^k))}_{\geq 0}\geq 0,
\end{equation*}
and
\begin{equation*}
    0\leq \mathrm{Gap}(x^k)= -\langle\nabla f(x^k), d^k\rangle \leq \| \nabla f(x^k)\| \|d^k\|<C.
\end{equation*}
Thus, it holds for all $k\in\mathbb{N}$ that
\begin{equation*}
\frac{(\mathrm{Gap}(x^k))^2}{2C}\leq f(x^k) - f(x^{k+1}).
\end{equation*}
Letting $N\in\mathbb{N}^*$ and summing from $k=0$ to $N$ we find 
\begin{equation*}
\sum_{k=0}^{N} \frac{(\mathrm{Gap}(x^k))^2}{2C}\leq f(x^0) - f(x^{N+1})\leq f(x^0) - f^\star.
\end{equation*}
Finally, 
\begin{equation*}
\begin{aligned}
&(N+1)\frac{\left(\min\limits_{0\leq k\leq N}\mathrm{Gap}(x^k)\right)^2}{2C}\leq\sum_{k=0}^{N} \frac{(\mathrm{Gap}(x^k))^2}{2C}\leq f(x^0)-f^\star \\
\Rightarrow&
\left(\min\limits_{0\leq k\leq N}\mathrm{Gap}(x^k)\right)^2\leq \frac{2Ch_0}{N+1}
\\
\Rightarrow&
\min\limits_{0\leq k\leq N}\mathrm{Gap}(x^k)\leq \sqrt{\frac{2C(f(x^0)-f^\star)}{N+1}},
\end{aligned}
\end{equation*}
which is the desired claim.
\qed 
\end{proof}

We now examine the convergence rate of the functional-value gap when \nameref{alg:ACGD} is employed for constrained problems with $\eta$-quasar-convex objective functions.

\begin{theorem}[Constrained Quasar-convex Convergence Rate]\label{thm:local_convergence_rate}
Let $\mathcal{C}\subset\mathbb{R}^n$ be a non-empty compact convex set and let $f\colon\mathbb{R}^n\to\mathbb{R}$ be continuously differentiable and $\eta$-quasar-convex with $\{x^k\}_{k\in\mathbb{N}}$ the sequence of iterates generated by \nameref{alg:ACGD}.
Then, for any $N\in\mathbb{N}^*$,
\begin{equation*}
f(x^{N}) - f^\star \leq \frac{2C}{N\eta^2+(\frac{2C}{f(x^0)-f^\star})},
\end{equation*}
where $C := \max\left\{\beta (L+\delta), \max\limits_{x\in\mathcal{C}}\|\nabla f(x)\|, 1\right\} \cdot (\max\{D, 1\})^2$.
\end{theorem}
\begin{proof}
Using the local descent property of $f$ from \eqref{eq:local_descent},
\begin{equation}\label{eq:constrained_local_descent}
f(x^{k+1}) \leq f(x^k) + t_k \langle \nabla f(x^k), d^k \rangle + \frac{L_k }{2} t_k^2\|d^k\|^2.
\end{equation}
Recall that the step-size is defined in this case to be
$t_k= \min\left\{ -\frac{\langle \nabla f(x^k), d^k \rangle}{L_k \|d^k\|^2}, 1 \right\}$, which gives two cases.
\\
\textbf{Case one:} $t_k=-\frac{\langle \nabla f(x^k), d^k \rangle}{L_k \|d^k\|^2}$. Substituting $t_k$ into \eqref{eq:constrained_local_descent}, we have:
\begin{equation*}
\begin{aligned}
   f(x^{k+1}) &\leq f(x^k) + \left(-\frac{\langle \nabla f(x^k), d^k \rangle}{L_k \|d^k\|^2}\right) \langle \nabla f(x^k), d^k \rangle + \frac{L_k}{2}  \left(-\frac{\langle \nabla f(x^k), d^k \rangle}{L_k \|d^k\|^2}\right)^2\|d^k\|^2\\
&
=f(x^k) -\frac{\langle \nabla f(x^k), d^k \rangle^2}{L_k \|d^k\|^2} + \frac{\langle \nabla f(x^k), d^k \rangle^2}{2L_k \|d^k\|^2}
\\
 &= f(x^k) - \frac{\langle \nabla f(x^k), d^k \rangle^2}{2L_k \|d^k\|^2}.
\end{aligned}
\end{equation*}
By $\eta$-quasar-convexity of $f$, we have
\begin{equation*}
f(x^\star)\geq f(x^k)+ \frac{1}{\eta}\langle \nabla f(x^k), x^\star-x^k \rangle \geq f(x^k)+\frac{1}{\eta}\langle \nabla f(x^k), d^k \rangle,
\end{equation*}
where $x^\star\in\argmin\limits_{x\in\C}f(x)$ and $\eta>0$. Thus, denoting $h_k:=f(x^k)-f^\star$, it follows $\langle \nabla f(x^k), d^k \rangle \leq -\eta h_k \leq 0$. Squaring this and multiplying by $-1$ gives
\begin{equation*}
-\langle \nabla f(x^k), d^k \rangle^2\leq -\eta^2h_k^2.
\end{equation*}
Therefore, using \eqref{eq:constrained_local_descent} and since $ \|d^k\| \leq D:= \max_{x,y \in \mathcal{C}} \|x - y\| $, it follows that
\begin{equation*}
h_{k+1} \leq h_k - \frac{\eta^2h_k^2}{2L_k D^2}\leq h_k - \frac{\eta^2h_k^2}{2C},
\end{equation*}
where $C := \left(\max\left\{\beta (L+\delta),\max_{x\in \mathcal{C}}\|\nabla f(x)\|, 1 \right\} \cdot (\max\{D, 1\})^2\right)>0$. 
\\
\textbf{Case two:} $t_k=1$, so  $  -\frac{\langle \nabla f(x^k), d^k \rangle}{L_k \|d^k\|^2}\geq 1 $. Substituting $t_k = 1$ into \eqref{eq:constrained_local_descent}, we have
\begin{equation*}
\begin{aligned}
f(x^{k+1}) &\leq f(x^k) +  \langle \nabla f(x^k), d^k \rangle + \frac{L_k}{2} \|d^k\|^2\\
&
\leq f(x^k) +  \langle \nabla f(x^k), d^k \rangle - \frac{1}{2}\langle \nabla f(x^k), d^k \rangle\\
&
=f(x^k) + \frac{1}{2}\langle \nabla f(x^k), d^k \rangle
\\
&
\leq f(x^k) - \frac{\eta h_k}{2},
\end{aligned}
\end{equation*}
so that, finally,
\begin{equation*}
h_{k+1}\leq h_k- \frac{\eta h_k}{2}= \frac{(1-\eta)h_k}{2}.
\end{equation*}
In both cases, we obtain
\begin{equation*}
h_{k+1} \leq \max\left\{ h_k - \frac{\eta^2h_k^2}{2C}, \frac{(1-\eta)h_k}{2} \right\}.
\end{equation*}
Also, 
$$\eta h_k\leq -\langle \nabla f(x^k), d^k \rangle \leq \|\nabla f(x^k)\| \|d^k\|
\leq  C < 2C.
$$
Now we show that $\max\left\{ h_k - \frac{\eta^2h_k^2}{2C}, \frac{(1-\eta) h_k}{2} \right\}=h_k - \frac{\eta^2 h_k^2}{2C}$, i.e., $h_k - \frac{\eta^2 h_k^2}{2C}\geq  \frac{(1-\eta) h_k}{2}$:
\begin{equation*}
h_k - \frac{\eta^2h_k^2}{2C}- \frac{(1-\eta) h_k}{2}=  \frac{(1+\eta)h_k}{2}- \frac{\eta^2h_k^2}{2C}=\underbrace{\frac{h_k}{2C}}_{\geq0}\underbrace{(C(1+\eta)-\eta^2h_k )}_{\geq 0}\geq 0.
\end{equation*}
Therefore  
\begin{equation*}
h_{k+1} \leq  h_k - \frac{\eta^2h_k^2}{2C}.
\end{equation*}

Since $h_k \geq 0$ for all $k\in\mathbb{N}$, the sequence $\{h_k\}_{k\in\mathbb{N}}$ is non-increasing and bounded below by 0.  Thus, $\{h_k\}_{k\in\mathbb{N}}$ converges to some limit $h^\star \geq 0$.  Taking the limit as $k \to \infty$ of the inequality above gives
\begin{equation*}
h^\star \leq h^\star - \frac{\eta^2(h^\star)^2}{2C},
\end{equation*}
which implies that $\frac{(h^\star)^2}{2C} \leq 0$. Since $C > 0$, this can only happen if $h^\star = 0$. Therefore, $\lim_{k\to\infty} h_k = 0$. Furthermore, for all $k\in\mathbb{N}$,
\begin{equation*}
\frac{h_{k+1}}{h_{k+1}h_k} \leq \frac{h_k - \frac{\eta^2h_k^2}{2C}}{h_{k+1}h_k}
\implies
\frac{1}{h_k} \leq \frac{1}{h_{k+1}} - \frac{\eta^2h_k}{2C h_{k+1}}.
\end{equation*}
Since $h_{k+1} \leq h_k$, we know that $\frac{h_k}{h_{k+1}} \geq 1$. Therefore,
\begin{equation*}
-\frac{\eta^2h_k}{2C h_{k+1}} \leq -\frac{\eta^2}{2C}.
\end{equation*}
Substituting this into our earlier equation, we get
\begin{equation*}
\frac{1}{h_k} \leq \frac{1}{h_{k+1}} - \frac{\eta^2}{2C}.
\end{equation*}
Adding $\frac{\eta^2}{2C}$ to both sides, we get:
\begin{equation*}
\frac{1}{h_{k+1}} \geq \frac{1}{h_k} + \frac{\eta^2}{2C}.
\end{equation*}
We can apply this repeatedly and unroll the recurrence to find
\begin{equation*}
\frac{1}{h_{k+1}} \geq \frac{1}{h_0} + \frac{(k+1)\eta^2}{2C}.
\end{equation*}
To get $h_{k+1}$ back, we take the reciprocal of both sides of the inequality and arrive at
\begin{equation*}
h_{k+1} \leq \frac{1}{\frac{1}{h_0} + \frac{(k+1)\eta^2}{2C}}
= \frac{1}{\frac{2C + (k+1)\eta^2h_0}{2C h_0}}=\frac{2C h_0}{2C + (k+1)\eta^2h_0}= \frac{2C}{\frac{2C}{h_0} + (k+1)\eta^2}.
\end{equation*}
\qed 
\end{proof}

The above theorem applies also to convex objective functions $f$ and establishes that the sequence of function values converges to the optimal value at a sub-linear rate. As with all of our convergence guarantees, while the global Lipschitz gradient condition is employed in the analysis, from a computational perspective, it is not necessary to know the value of the global Lipschitz constant for the gradient.


\section{Numerical Experiments}\label{sec:numerical}

We performed extensive numerical experiments to assess the efficacy of \nameref{alg:ACGD}, benchmarking its performance against some other step-size strategies for both constrained and unconstrained optimization problems. All experiments were conducted using Julia \cite{Julia} on the Google Colaboratory cloud-based computational platform (Google Colab) \cite{google}.

In what follows, we present these numerical results for \nameref{alg:ACGD}. To ensure a fair comparison, all methods were initialized from the same starting point and were run for a maximum of 3,000 iterations. The complete dataset and source code are publicly accessible using this \href{https://github.com/abbaskhademi/Adaptive-Step-Size}{link}.

\subsection{Constrained Problems}
To evaluate the impact of step-size selection, we compare several strategies across a range of problems: the adaptive step-size schemes proposed in this work (i.e., \nameref{alg:ACGD}), the pure backtracking step-size from \cite{pedregosa2020linearly}, the classical short-step rule, and the open-loop step-size schedule. This comparison aims to assess the practical efficiency and convergence behavior of our adaptive methods relative to established step-sizes as baselines.


Convergence was determined using a termination criterion of $10^{-5}$. We remind the reader that, for constrained problems, this criterion is chosen to be the functional-value gap when the optimal value is known; otherwise, the Frank-Wolfe gap is used.
\subsubsection{Lasso Problem}
%

We first consider a rewritten variant of the classical Lasso problem, which employs an $\ell^1$-norm constraint and is widely used for sparse statistical recovery \cite{tibshirani1996regression}. The problem is formulated as
\begin{equation}
\begin{array}{cl}
\min\limits_{x\in \mathbb{R}^n} & \|b - Ax\|_2^2 \\
\mathrm{s.t.} & \|x\|_1 \leq \tau
\end{array}
\end{equation}
where $A \in \mathbb{R}^{m \times n}$ is the measurement matrix, $b \in \mathbb{R}^m$ is the observed signal, and $\tau > 0$ controls the sparsity of the solution $x \in \mathbb{R}^n$.

We generated a random problem instance to comprehensively assess step-size methods with
 $m = 2,000$, $n = 10,000$, and $\tau = 10$.
For each instance, we generate a random ground truth solution $x^{\text{true}}$ as a sparse vector with $\tau$ non-zero entries sampled from a standard Gaussian distribution. The matrix $A$ is also random with Gaussian entries, and the signal $b$ was computed as $b := A x^{\text{true}}$, resulting in an optimal objective value of zero.
The Lipschitz constant of the gradient was computed as $L = 2\lambda_{\max}(A^\top A)$.

Table \ref{tab:Lasso Problem Results} presents a comparative analysis of the performance of various step-size strategies on this Lasso instance. The adaptive adjustable, adaptive constant, and pure backtracking methods all successfully converged to the optimal solution, satisfying the functional-value gap stopping criterion. Bold values in the table indicate the best performance in each column (e.g., lowest runtime, smallest objective value, or smallest gap).

\begin{table}[htpb!]
\centering
\caption{Performance comparison of step-size strategies on Lasso problem}
\label{tab:Lasso Problem Results}
\begin{tabular}{lcccc}
\toprule
Step-size strategy & Iterations & Time (s) & Objective value & Functional-value gap \\
\midrule
Adaptive constant & 331 & 19.23 & \textbf{0.0000} & \textbf{0.0000} \\
Adaptive adjustable & 343 & \textbf{17.15} & \textbf{0.0000} & \textbf{0.0000} \\
Pure backtracking & 407 & 24.22 & \textbf{0.0000} & \textbf{0.0000} \\
Short-step & 3000 & 128.84 & 0.0042 & 0.0042 \\
Open-loop & 3000 & 132.39 & 0.1025 & 0.1025 \\
\bottomrule
\end{tabular}
\end{table}

Among these, the adaptive adjustable method demonstrated superior efficiency, achieving convergence in 17.15 seconds over 343 iterations. The adaptive constant method followed closely, requiring 19.23 seconds and 331 iterations. The pure backtracking method, while convergent, was significantly slower, taking 24.22 seconds over 407 iterations. In contrast, both the short-step and open-loop strategies failed to converge within the limit of 3,000 iterations, exhibiting significant sub-optimality. Their final objective values were 0.0042 and 0.1025, respectively, indicating a clear performance gap compared to adaptive approaches.

%

\begin{figure}[htpb!]
\centering
\subfloat[Gaps in Log Scale]{%
  \includegraphics[width=0.3\linewidth]{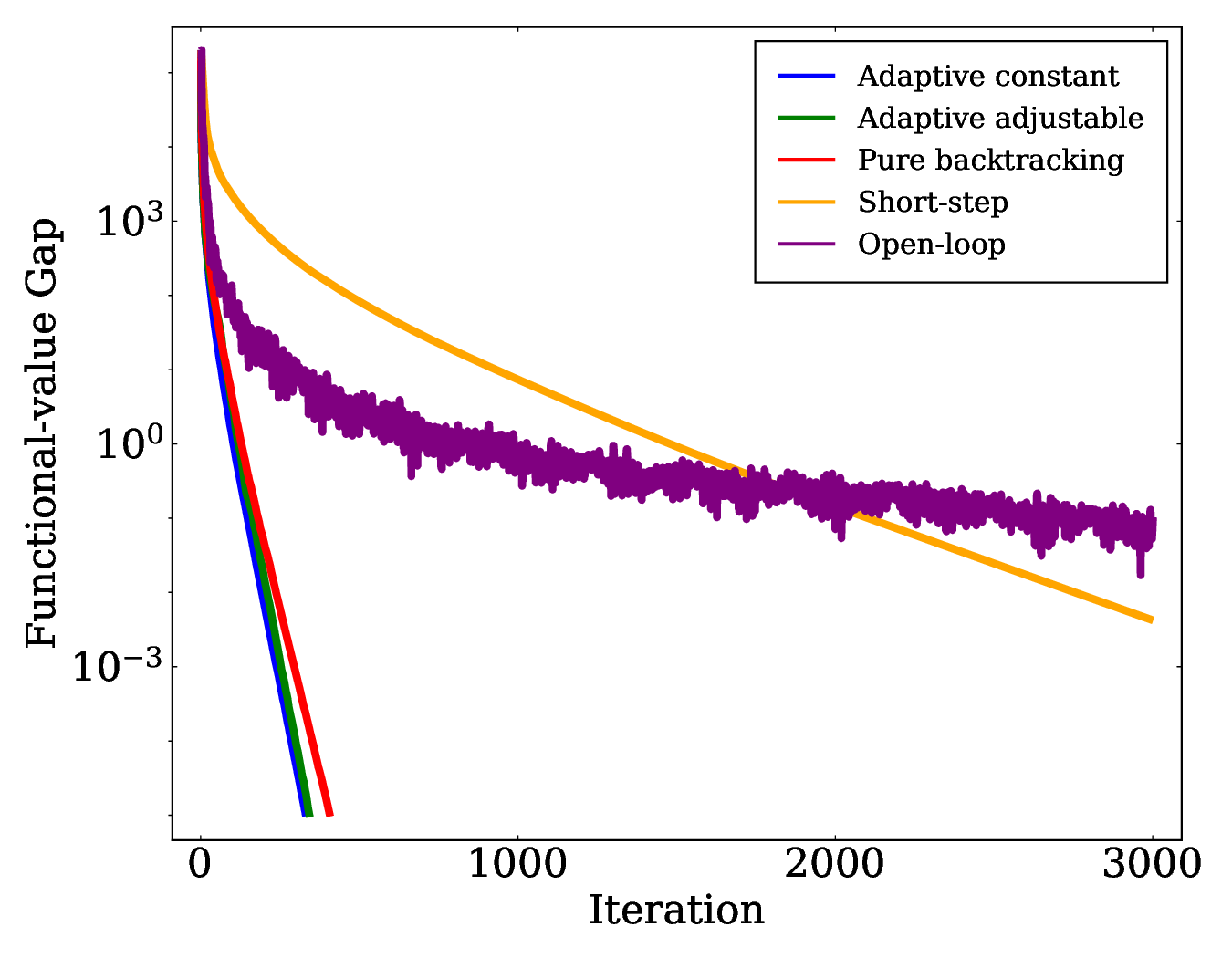}%
}
\hfill
\subfloat[Lipschitz Constant Evolution]{%
   \includegraphics[width=0.3\linewidth]{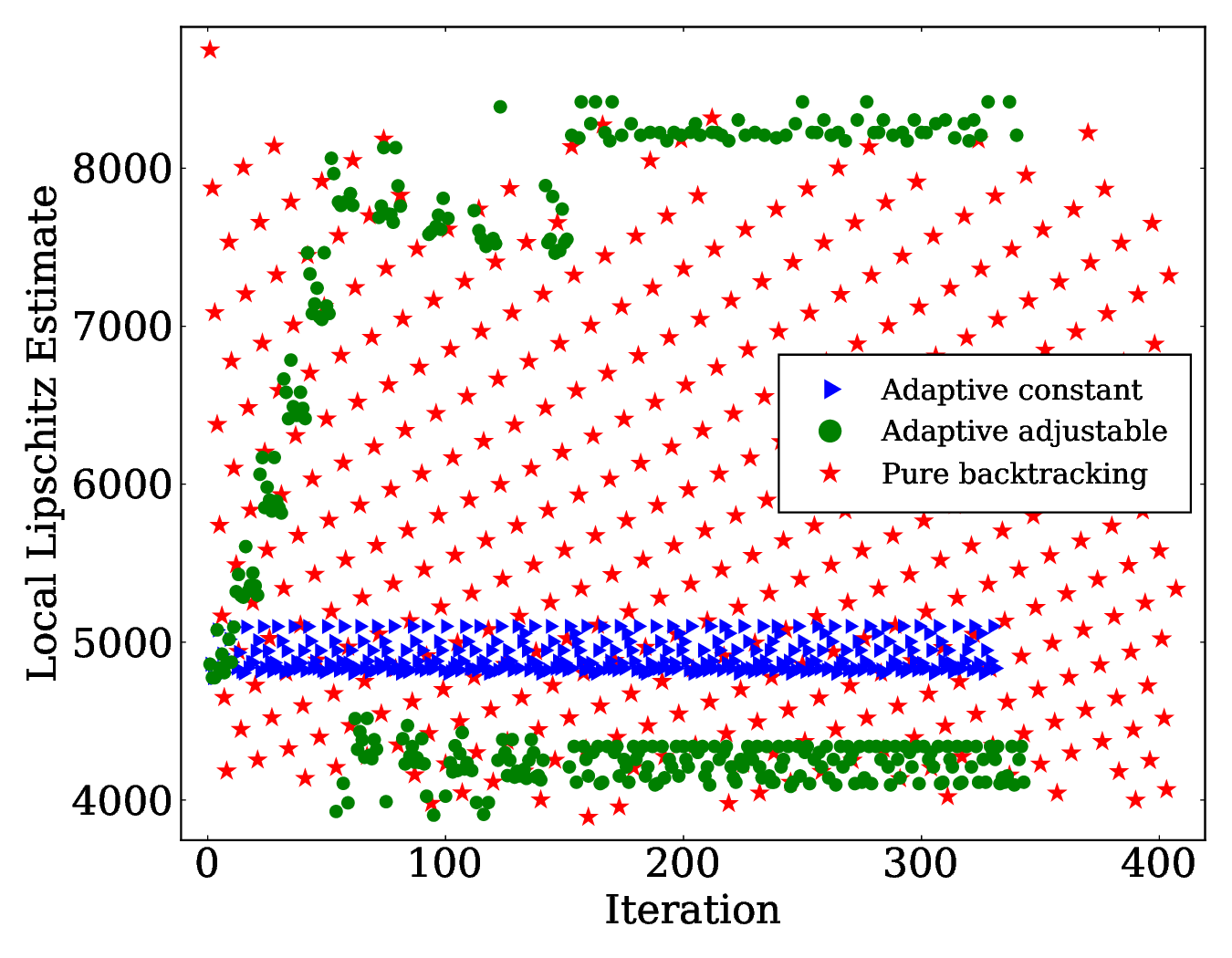}%
}
\hfill
\subfloat[Gaps over Time]{%
  \includegraphics[width=0.3\linewidth]{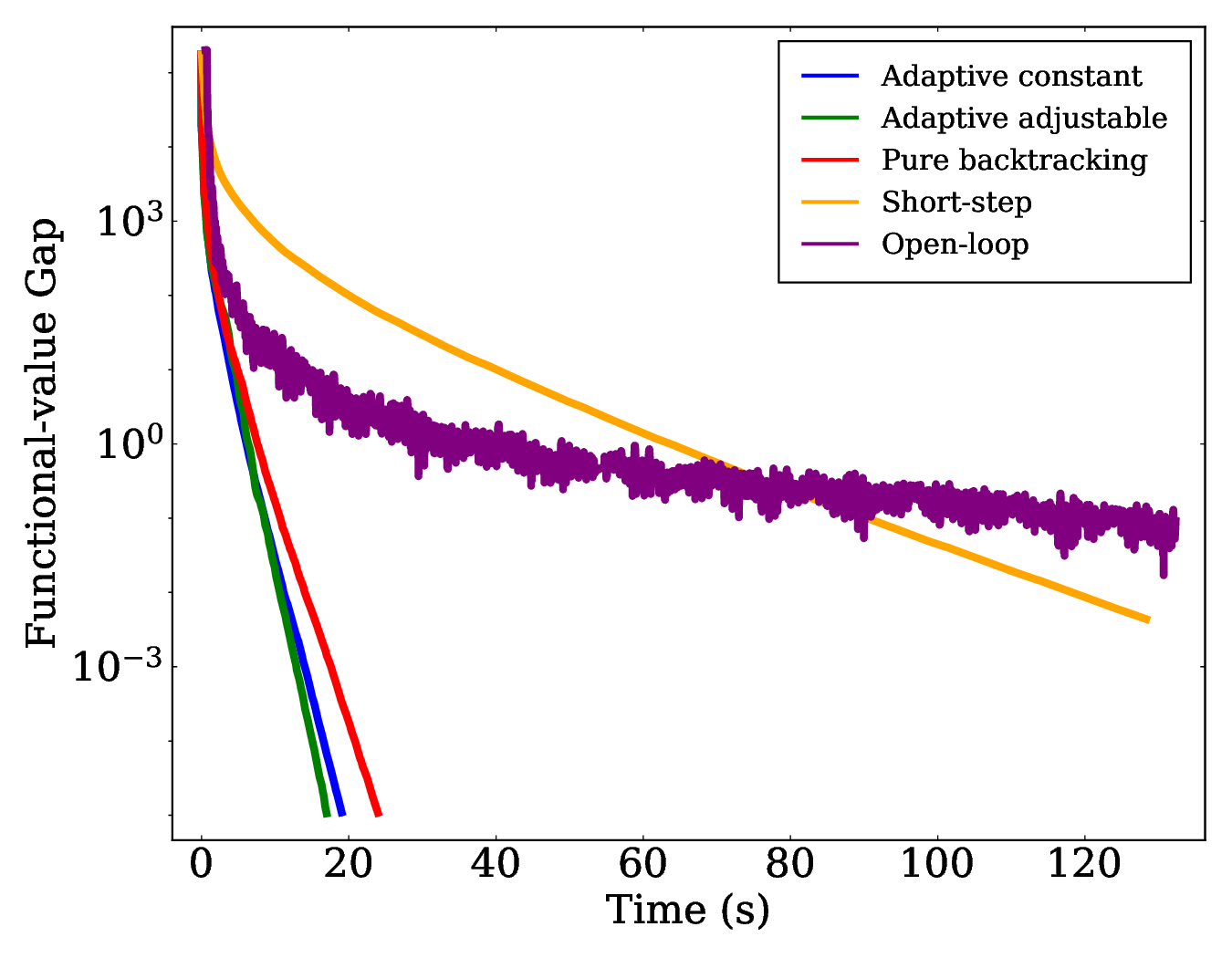}%
}
\caption{Convergence behavior and computational efficiency for Lasso problem}
\label{fig:lasso}
\end{figure}

Figure~\ref{fig:lasso} illustrates the convergence behavior of the step-size strategies. The open-loop and short-step methods exhibit the slowest convergence, as they fail to adapt to the local geometry of the problem. Among the adaptive methods, pure backtracking shows high variability in Lipschitz estimates, leading to inefficient progress. In contrast, the adaptive adjustable method achieves the fastest convergence by warm-starting the backtracking procedure, effectively balancing aggressiveness with stability.

As shown in Figure~\ref{fig:lasso}(b), the adaptive adjustable method (green dots) maintains Lipschitz estimates tightly concentrated in a low, effective range. In contrast, pure backtracking (red stars) exhibits wide fluctuations — ranging from below 4,000 to over 8,000 — resulting in either unstable or overly conservative steps. By consistently operating near the lower end of this spectrum, the adaptive adjustable method achieves faster and more stable convergence.

\subsubsection{Matrix Balancing Problem}
%

Matrix balancing is a fundamental problem in optimization with applications in network analysis and numerical linear algebra \cite{cohen2017matrix,kalantari1997complexity}. We consider the following constrained optimization problem
\begin{equation}
\begin{array}{cl}
\min\limits_{x \in \mathbb{R}^n} & \sum_{i,j=1}^n a_{ij} \exp(x_i - x_j) \\
\mathrm{s.t.} & x \in [a, b]^n
\end{array}
\end{equation}
where $A=[a_{ij}] \in \mathbb{R}^{n \times n}$ is a square matrix with non-negative entries, satisfying the balanced condition that the sum of each row equals the sum of its corresponding column \cite{cohen2017matrix}. While the gradient’s Lipschitz constant $L$ lacks a closed-form expression, a conservative upper bound can be derived over the feasible set $\mathcal{C}$.

We generate synthetic data with $n = 100$, $a = 1$, and $b = 10$. The matrix $A$ is constructed as a sparse, symmetric, and non-negative matrix. Specifically, we set a density factor $d = 5$ and initialize a random sparse matrix $A \in \mathbb{R}^{n \times n}$ using a sparsity level of $d/n$. To ensure symmetry and non-negativity, we compute $A = |A + A^\top| + 0.05 I_n$, where $I_n$ is the identity matrix, introducing a small diagonal shift to enhance numerical stability. 

Table~\ref{tab:MatrixProblem Results} compares step-size strategies. All adaptive methods (adaptive constant, adaptive adjustable, and pure backtracking) converge to the optimal value of 480.2835. The adaptive constant method is fastest (228 iterations, 75.44 s), followed by adaptive adjustable (291 iterations, 94.70 s). Pure backtracking requires 406 iterations (131.08 s). The open-loop method fails to converge within 3,000 iterations, yielding a primal gap of 0.0011. The short-step rule, reliant on a conservative Lipschitz bound, was prohibitively slow and is omitted.
\begin{table}[hpb!]
\centering
\caption{Performance comparison of step-size strategies on Matrix Balancing problem}
\label{tab:MatrixProblem Results}
{
\begin{tabular}{lccccccc}
\toprule
Step-size strategy & Iterations & Time (s) & Objective value & Functional-value gap
\\
\midrule
Adaptive constant & \textbf{228} & \textbf{75.44} & \textbf{480.2835} & \textbf{0.0000} 
\\
Adaptive adjustable & 291 & 94.70 & \textbf{480.2835} & \textbf{0.0000} 
\\
Pure backtracking & 406 & 131.08 & \textbf{480.2835} & \textbf{0.0000} %
\\
Open-loop & 3000 & 964.16 & 480.2846 & 0.0011 
\\
\bottomrule
\end{tabular}
}
\end{table}

Figure~\ref{fig:Matrix} illustrates convergence behavior. Adaptive methods maintain low, stable Lipschitz estimates $L_k$, enabling aggressive yet safe steps. In contrast, pure backtracking exhibits wider $L_k$ fluctuations, leading to slower progress. The open-loop method, lacking adaptivity, stagnates early.

\begin{figure}[htpb!]
\centering
\subfloat[Gaps in Log Scale]{%
   \includegraphics[width=0.3\linewidth]{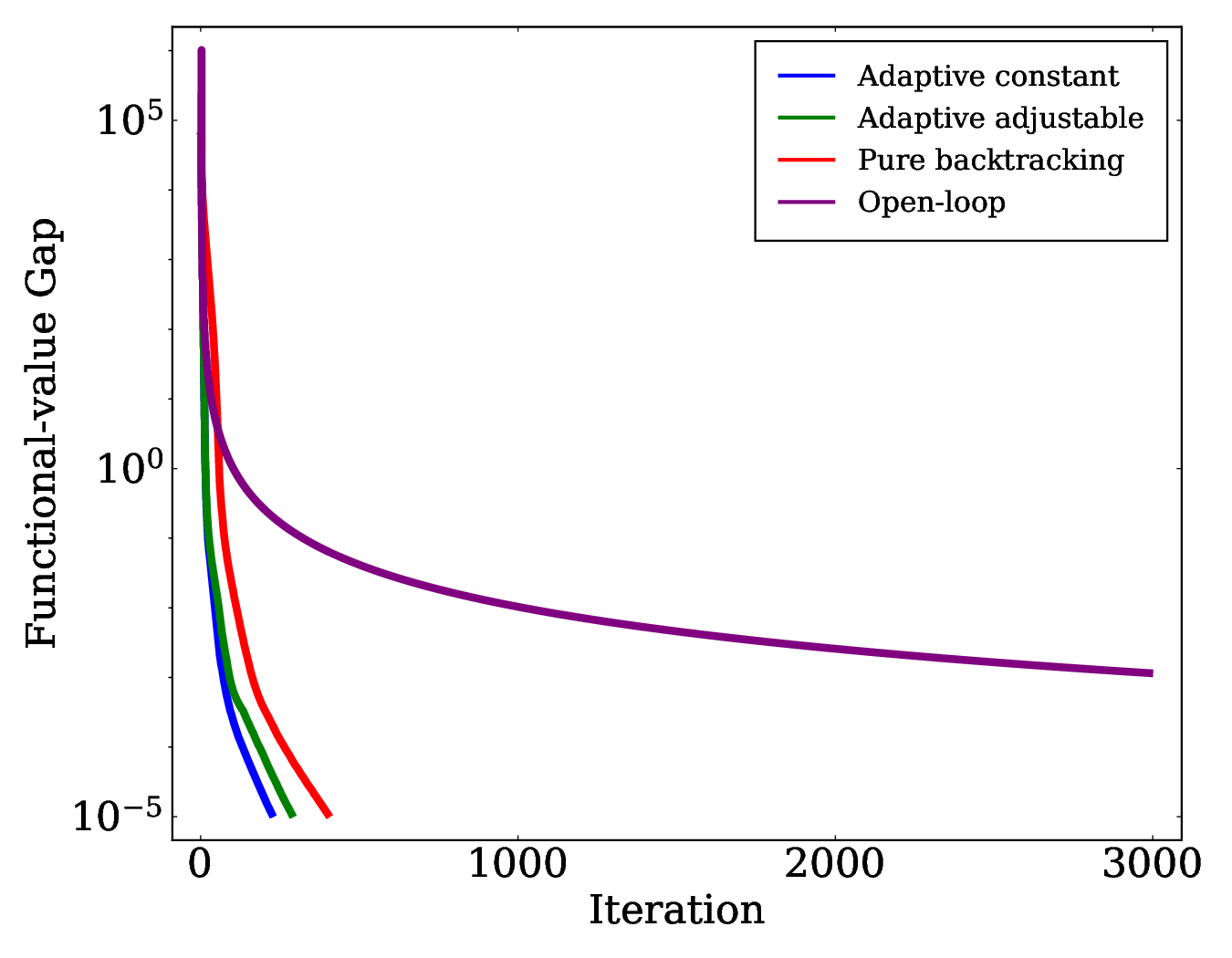}%
}
\hfill
\subfloat[Lipschitz Constant Evolution]{%
  \includegraphics[width=0.3\linewidth]{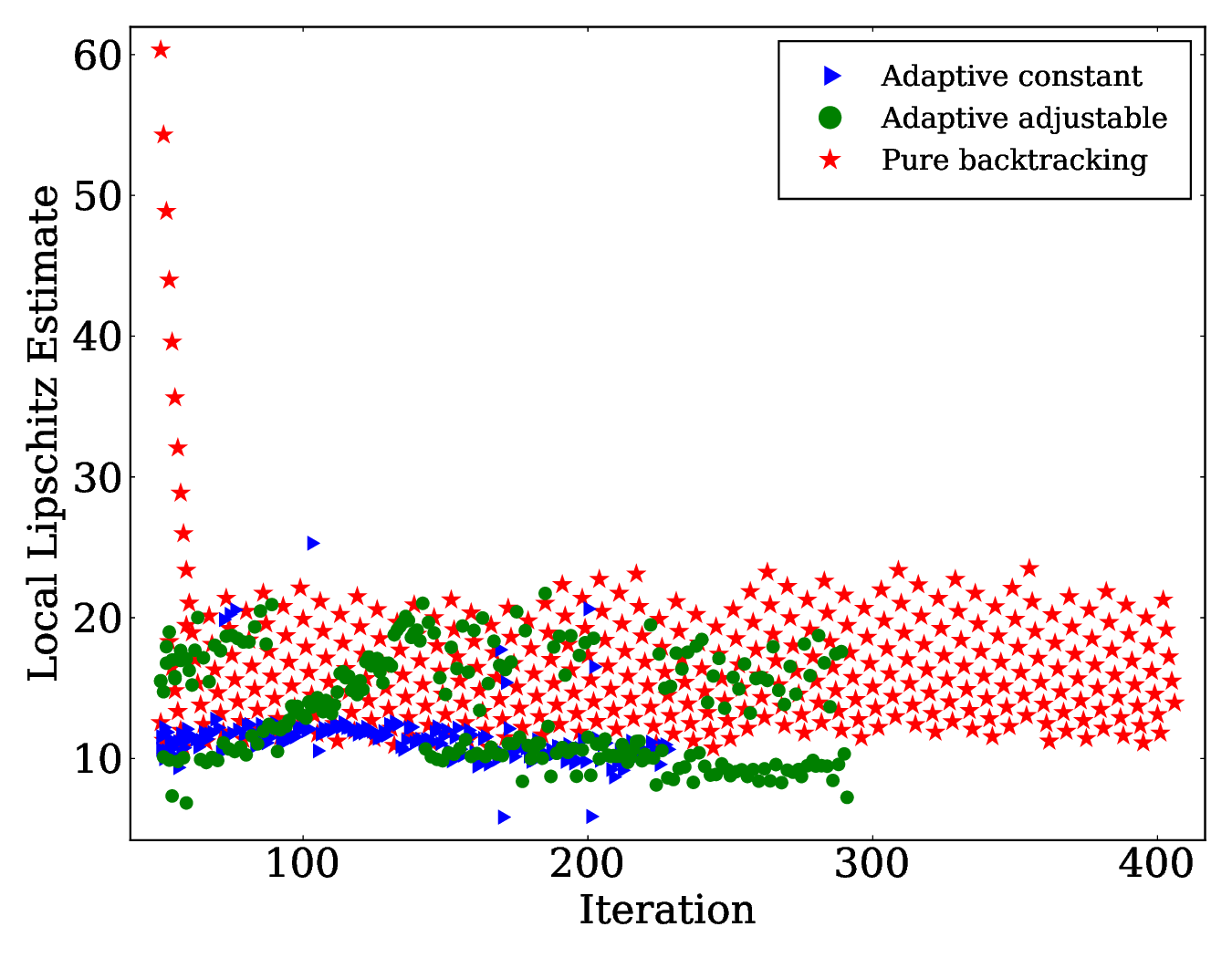}%
}
\hfill
\subfloat[Gaps over Time]{%
 \includegraphics[width=0.3\linewidth]{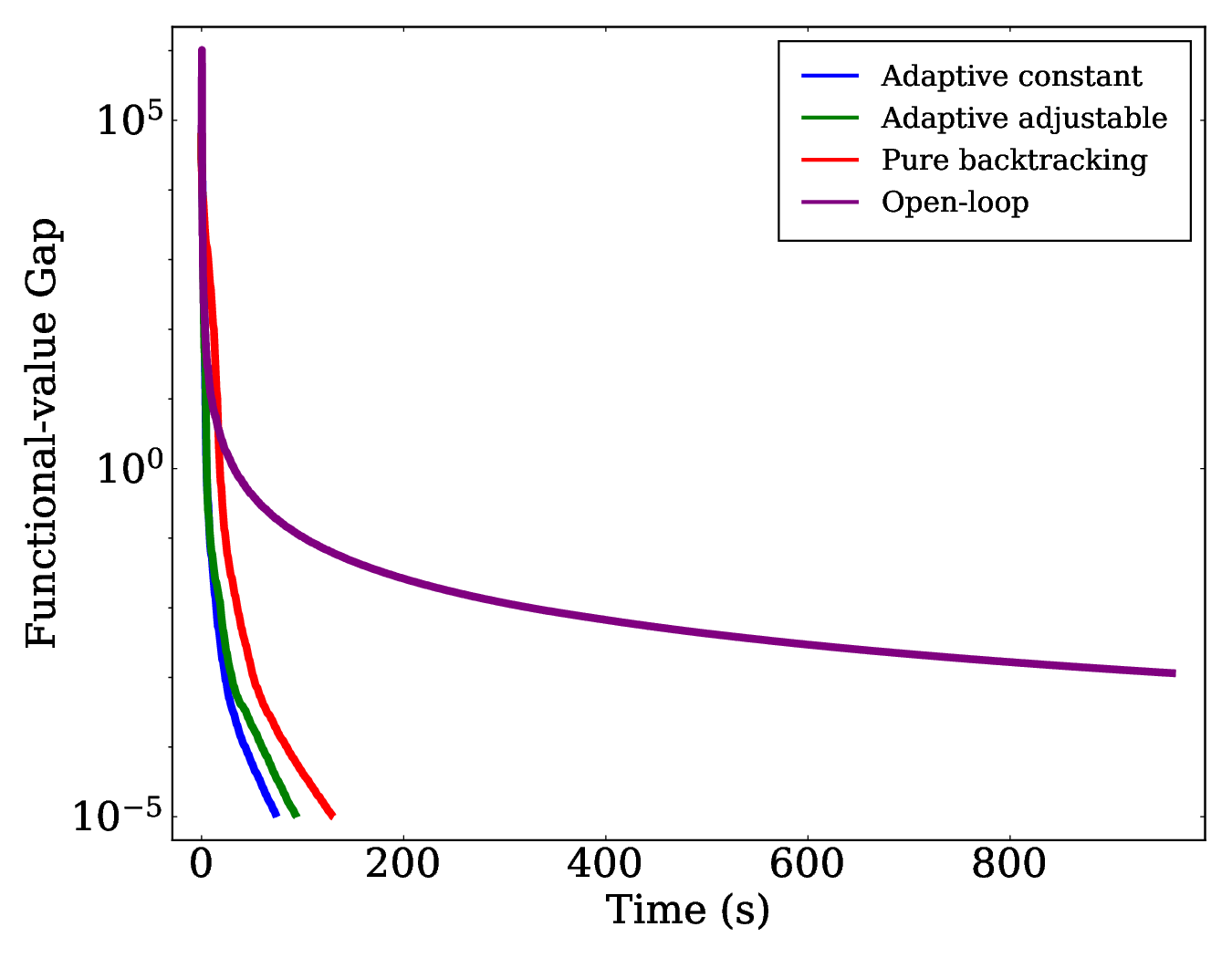}%
}
\caption{Convergence behavior and computational efficiency for Matrix Balancing problem}
\label{fig:Matrix}
\end{figure}

\subsubsection{Sparse Logistic Regression Problem}
%

Logistic regression is a cornerstone of binary classification in supervised learning \cite{lee2006efficient}. Let us consider the following logistic regression problem
\begin{equation}
\begin{array}{cl}
\min\limits_{x\in \mathbb{R}^n} & \dfrac{1}{m} \sum\limits_{i=1}^m \log \left(1 + \exp (-y_i x^\top a_i)\right) \\
\mathrm{s.t.} & \|x\|_1 \leq \tau
\end{array}
\end{equation}
where $A \in \mathbb{R}^{m \times n}$ is the feature matrix whose $i$-th row is $a_i \in \mathbb{R}^n$, $y \in \{-1,1\}^m$ is a vector whose entries are labels, and $\tau > 0$ controls solution sparsity. Here, The logistic loss function is $L$-smooth with 
$L = \|A\|_{\text{op}}^2 / (4m)$,
where $\|A\|_{\text{op}}$ denotes the operator norm (or spectral norm) of the matrix $A$, defined as the largest singular value of $A$.

We evaluate step-size strategies on LIBSVM \texttt{a1a} dataset \cite{chang2011libsvm} ($m=1605$ and $n=123$). All methods hit the 3,000-iteration limit and results are summarized in Table~\ref{tab:Logistic Regression a1a}.
Adaptive constant yields the best objective (0.3114) and smallest Frank-Wolfe gap (0.0086). Adaptive adjustable is fastest (1.35 s) with near-optimal accuracy (gap: 0.0112). Pure backtracking lags in both speed and accuracy. Non-adaptive methods underperform: short-step yields a poor objective (0.4522) and large gap (1.1077); open-loop is fastest (1.03 s) but plateaus at a suboptimal gap (0.0258).

\begin{table}[htpb!]
\centering
\caption{Performance comparison of step-size strategies on Logistic Regression problem}
\label{tab:Logistic Regression a1a}
{
\begin{tabular}{lccccccc}
\toprule
Step-size strategy & Iterations & Time & Objective value & Frank-Wolfe gap 
\\
\midrule
Adaptive constant & 3000 & 1.65 & \textbf{0.3114} & \textbf{0.0086} 
\\
Adaptive adjustable & 3000 & 1.35 & 0.3135 & 0.0112 
\\
Pure backtracking & 3000 & 1.73 & 0.3171 & 0.0190 
\\
Short-step & 3000 & 1.41 & 0.4522 & 1.1077 
\\
Open-loop & 3000 & \textbf{1.03} & 0.3176 & 0.0258 
\\
\bottomrule
\end{tabular}
}
\end{table}

Figure~\ref{fig:logistic_a1a_convergence} illustrates the convergence behavior across all methods. The left plot shows that adaptive strategies achieve significantly lower Frank-Wolfe gaps by the iteration limit, while non-adaptive methods stagnate at higher values. The center plot reveals that the adaptive adjustable method maintains stable and efficient local Lipschitz estimates, contributing to consistent progress. The right plot highlights the trade-off between speed and accuracy.

%
\begin{figure}[htpb!]
\centering
\subfloat[Gaps in Log Scale]{%
   \includegraphics[width=0.3\linewidth]{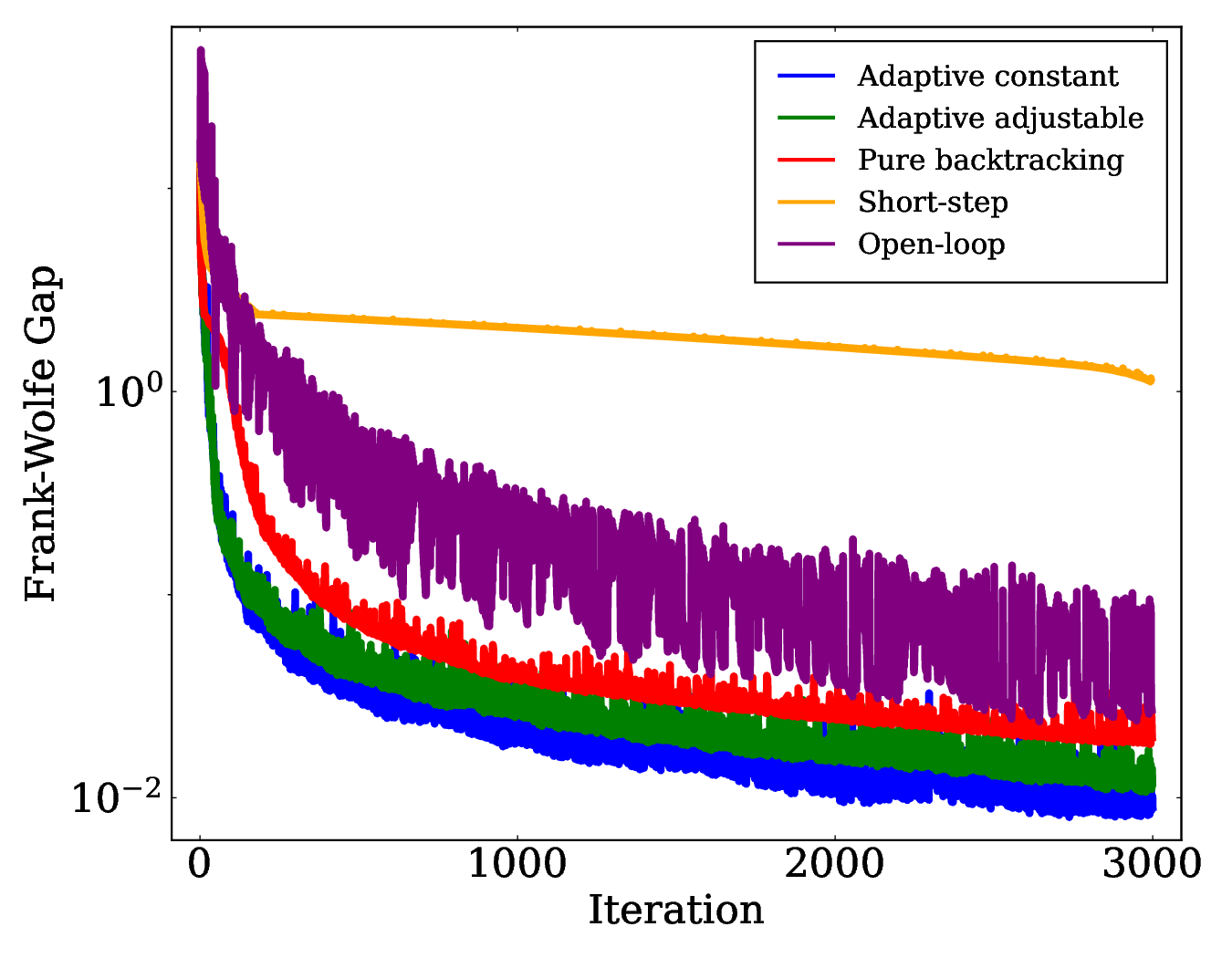}%
}\hfill
\subfloat[Lipschitz Constant Evolution]{%
    \includegraphics[width=0.3\linewidth]{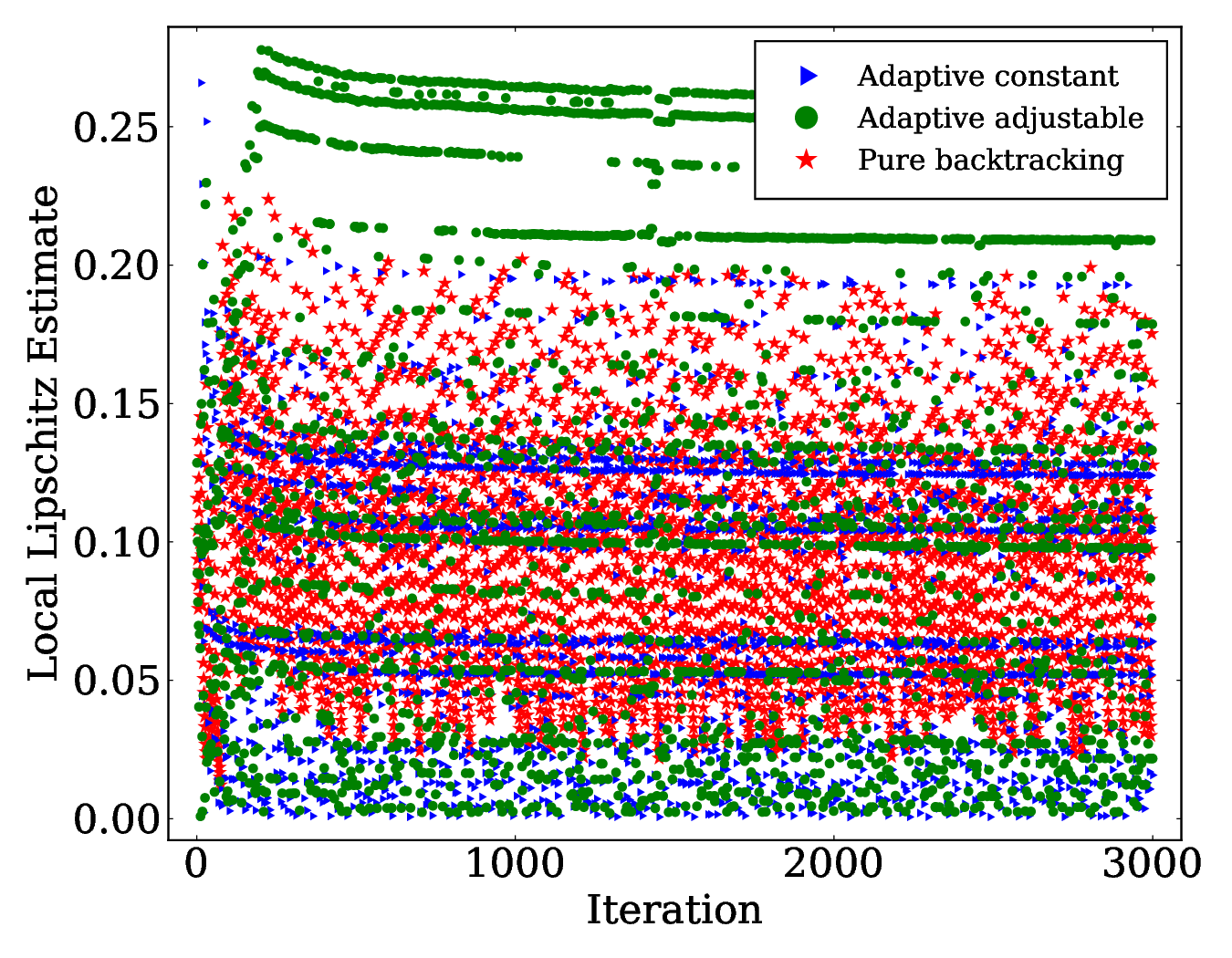}%
}\hfill
\subfloat[Gaps over Time]{%
   \includegraphics[width=0.3\linewidth]{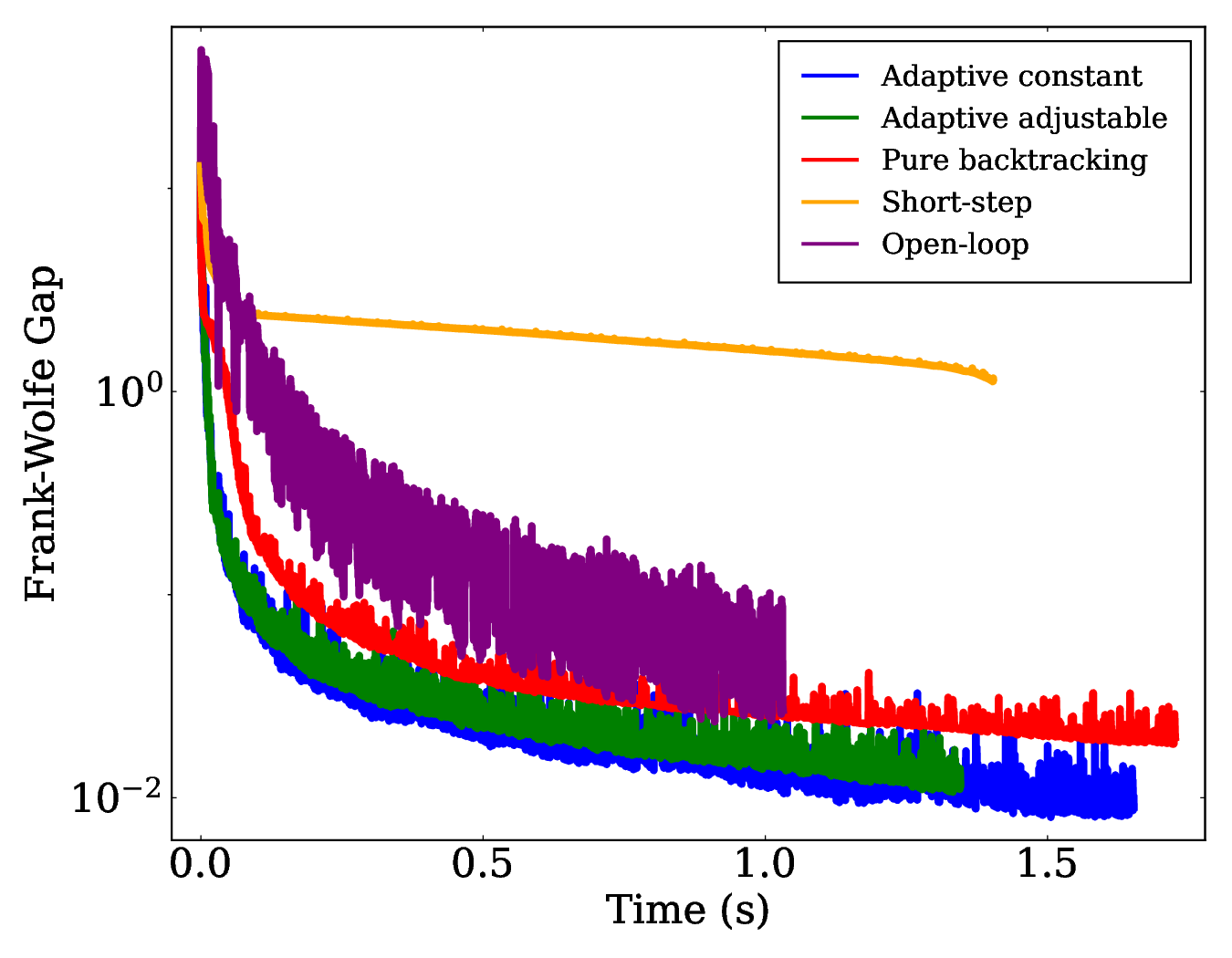}%
}
\caption{Convergence behavior and computational efficiency for Logistic Regression problem}
\label{fig:logistic_a1a_convergence}
\end{figure}

\subsubsection{Sparse Least Squares Sigmoid Regression Problem (non-convex)}

%
%

We now evaluate the step-size strategies on a non-convex optimization problem as follows 
\begin{equation}
\begin{array}{cl}
\min\limits_{x\in \mathbb{R}^n} & \dfrac{1}{m} \sum_{i=1}^m \left(y_i-\frac{1}{1+\exp \left(-x^{\top} a_i\right)}\right)^2 \\
\mathrm{s.t.} & \|x\|_1 \leq \tau,
\end{array}
\end{equation}
where the loss function is non-convex. The experiment is conducted on the LIBSVM \texttt{a1a} dataset \cite{chang2011libsvm}, which has $m=1,605$ samples and $n=123$ features. The regularization parameter is set to $\tau=50$.

The performance of the five step-size strategies is summarized in Table \ref{tab:Sigmoid_Regression_Problem} and visualized in Figure \ref{fig:Sigmoid_Regression_Problem}. Since the global optimum is unknown for this non-convex problem, we assess performance based on both the final objective value, the Frank-Wolfe gap as a measure of stationarity, and the total runtime.

\begin{table}[htpb!]
\centering
\caption{Performance comparison of step-size strategies on non-convex Sigmoid Regression problem}
\label{tab:Sigmoid_Regression_Problem}
{
\begin{tabular}{lcccc}
\midrule
Step-size strategy & Iterations & Time (s) & Objective value & Frank-Wolfe gap\\
\hline
Adaptive constant     & 3000 & 1.11 & \textbf{0.0989} & \textbf{0.0032} \\
Adaptive adjustable   & 3000 & \textbf{0.85} & 0.0995 & 0.0043   \\
Pure backtracking     & 3000 & 0.99 & 0.1007 & 0.0071  \\
Short-step            & 3000 & 0.69 & 0.1184 & 0.1418  \\
Open-loop             & 3000 & 0.72 & 0.1279 & 1.6103  \\
\bottomrule
\end{tabular}
}
\end{table}

The adaptive methods and pure backtracking demonstrate superior performance. The Adaptive constant strategy achieves the lowest final objective value and the smallest Frank-Wolfe gap, indicating the best approximately-stationary point. The Adaptive adjustable strategy is the most computationally efficient, exhibiting the lowest runtime while achieving a comparable solution quality. Pure backtracking also yields a competitive result, though it converges to a slightly higher Frank-Wolfe gap. As seen in Figure \ref{fig:Sigmoid_Regression_Problem}, these three methods effectively and rapidly reduce the Frank-Wolfe gap, both per iteration and over wall-clock time. In contrast, the non-adaptive strategies struggle. The Short-step method converges very slowly, resulting in a suboptimal solution with a high Frank-Wolfe gap. The open-loop strategy is particularly ill-suited for this problem; its Frank-Wolfe gap oscillates wildly and fails to decrease consistently, indicating a failure to converge.

Figure \ref{fig:Sigmoid_Regression_Problem}(b) offers insight into the dynamics of the adaptive methods, showing the evolution of their local Lipschitz constant estimates. The estimates fluctuate significantly throughout the optimization process, which is expected in a non-convex setting, yet they remain sufficiently bounded to guide the algorithms toward a solution. The relative stability of the adaptive methods highlights their robustness in complex optimization landscapes where a fixed step-size is inadequate.

\begin{figure}[htpb!]
\centering
\subfloat[Gaps in Log Scale]{%
   \includegraphics[width=0.3\linewidth]{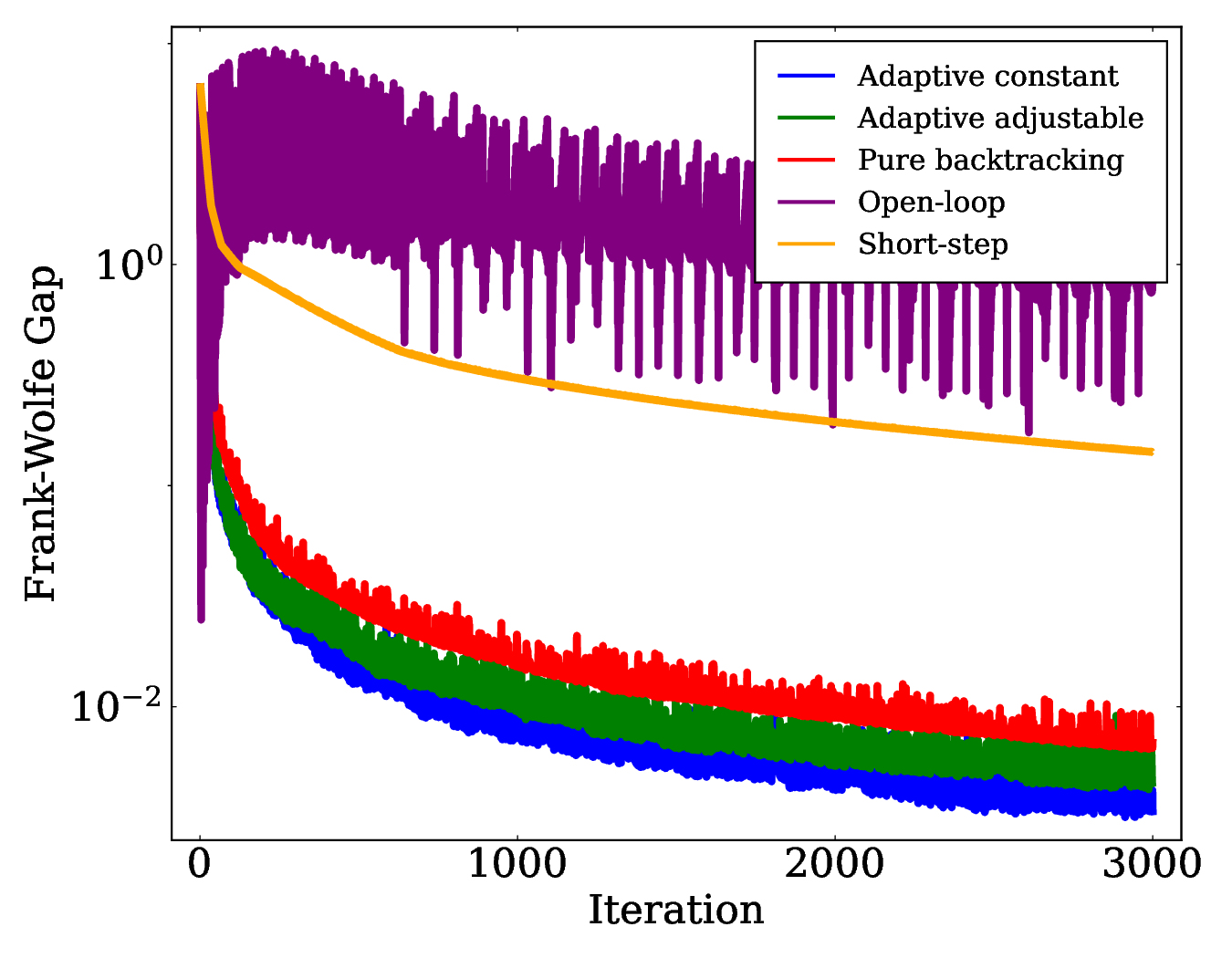}%
}\hfill
\subfloat[Lipschitz Constant Evolution]{%
    \includegraphics[width=0.3\linewidth]{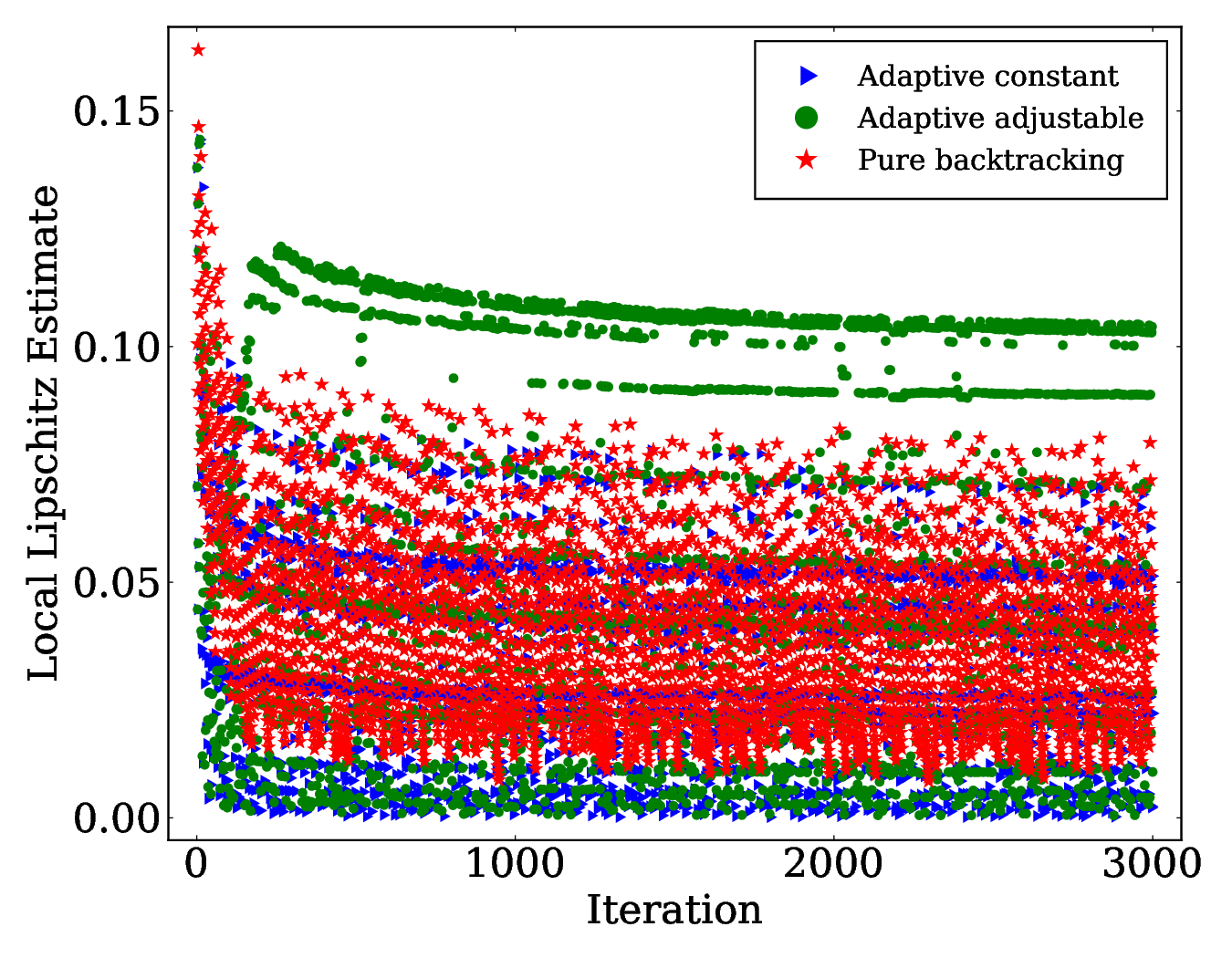}%
}\hfill
\subfloat[Gaps over Time]{%
    \includegraphics[width=0.3\linewidth]{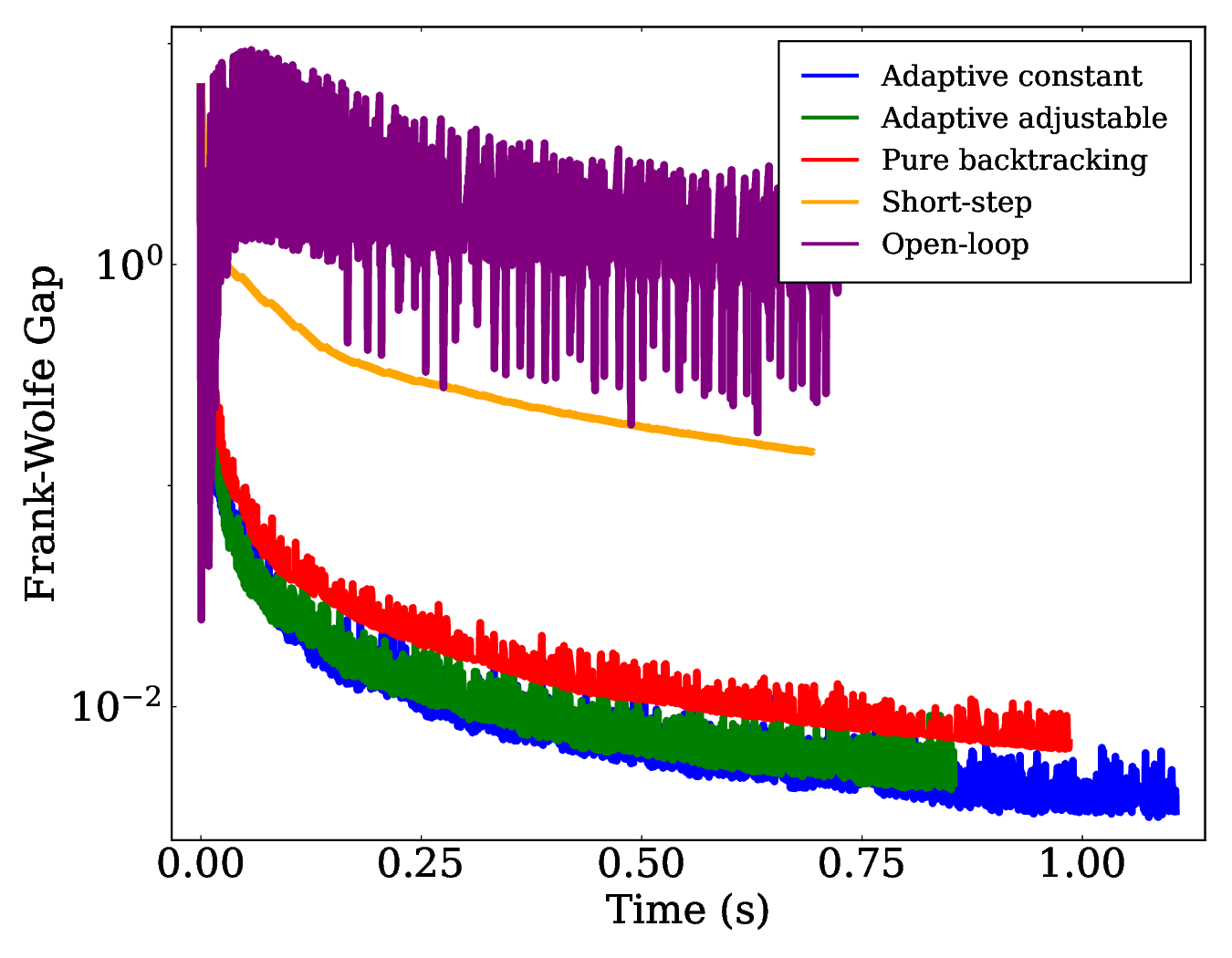}%
}
\caption{Convergence behavior and computational efficiency for non-convex Sigmoid Regression problem}
\label{fig:Sigmoid_Regression_Problem}
\end{figure}

\subsubsection{Video Co-Localization}
%

We investigate the video co-localization task using the aeroplane class from the YouTube-Objects dataset \cite{prest2012learning}, adopting the approach of  \cite{joulin2014efficient}. The aim is to detect and follow the aeroplane across video frames by computing bounding boxes.  Mathematically, this is expressed as the optimization problem
\begin{equation}\label{eq:VCL}\tag{VCL}
\begin{array}{cl}
\min\limits_{x \in \mathbb{R}^{660}} & \dfrac{1}{2} x^\top A x + b^\top x \\
\mathrm{s.t.} & x \in \mathcal{P}
\end{array}
\end{equation}
where $A \in \mathbb{R}^{660 \times 660}$ and $b \in \mathbb{R}^{660}$ are data-derived matrices, and $\mathcal{P}$ denotes a flow polytope encoding temporal consistency across frames. Note that, the LMO over $\mathcal{P}$ reduces to computing a shortest path in a corresponding directed acyclic graph.

Results are summarized in Table~\ref{tab:VideoCo-Localization}. The adaptive constant, adaptive adjustable, pure backtracking, and open-loop methods all converge to the optimal value with zero Frank-Wolfe gap. Only the short-step method exhibits slight sub-optimality. In terms of speed, the short-step method is fastest but sacrifices accuracy. Among fully convergent methods, adaptive adjustable is most efficient, followed by open-loop, pure backtracking, and adaptive constant.

\begin{table}[htpb!]
\centering
\caption{Performance comparison of step-size strategies on Video Co-Localization problem}
\label{tab:VideoCo-Localization}
{
\begin{tabular}{lcccc c c c}
\toprule
Step-size strategy & Iterations & Time  & Objective value & Frank-Wolfe gap\\
\midrule
Adaptive constant & 3000 & 1.66 & \textbf{0.0984} & \textbf{0.0000} \\
Adaptive adjustable & 3000 & 0.91 & \textbf{0.0984} & \textbf{0.0000} \\
Pure backtracking & 3000 & 1.16 & \textbf{0.0984} & \textbf{0.0000} \\
Short-step & 3000 & \textbf{0.80} & 0.0990& 0.0006 \\
Open-loop & 3000 & 1.02 & \textbf{0.0984} & \textbf{0.0000} \\
\bottomrule
\end{tabular}
}
\end{table}

Figure~\ref{fig:VideoCoLocalization} illustrates convergence behavior. The left plot confirms that all methods except short-step achieve near-zero duality gap, with adaptive methods converging fastest. The center plot shows that adaptive constant and adaptive adjustable maintain low, stable Lipschitz estimates ($L_k$), while pure backtracking exhibits higher variability. The right plot  highlights efficiency: adaptive adjustable and open-loop achieve the fastest initial progress, while short-step — though fastest overall — plateaus at a suboptimal level.

\begin{figure}[htpb!]
\centering
\subfloat[Gaps in Log Scale]{%
   \includegraphics[width=0.3\linewidth]{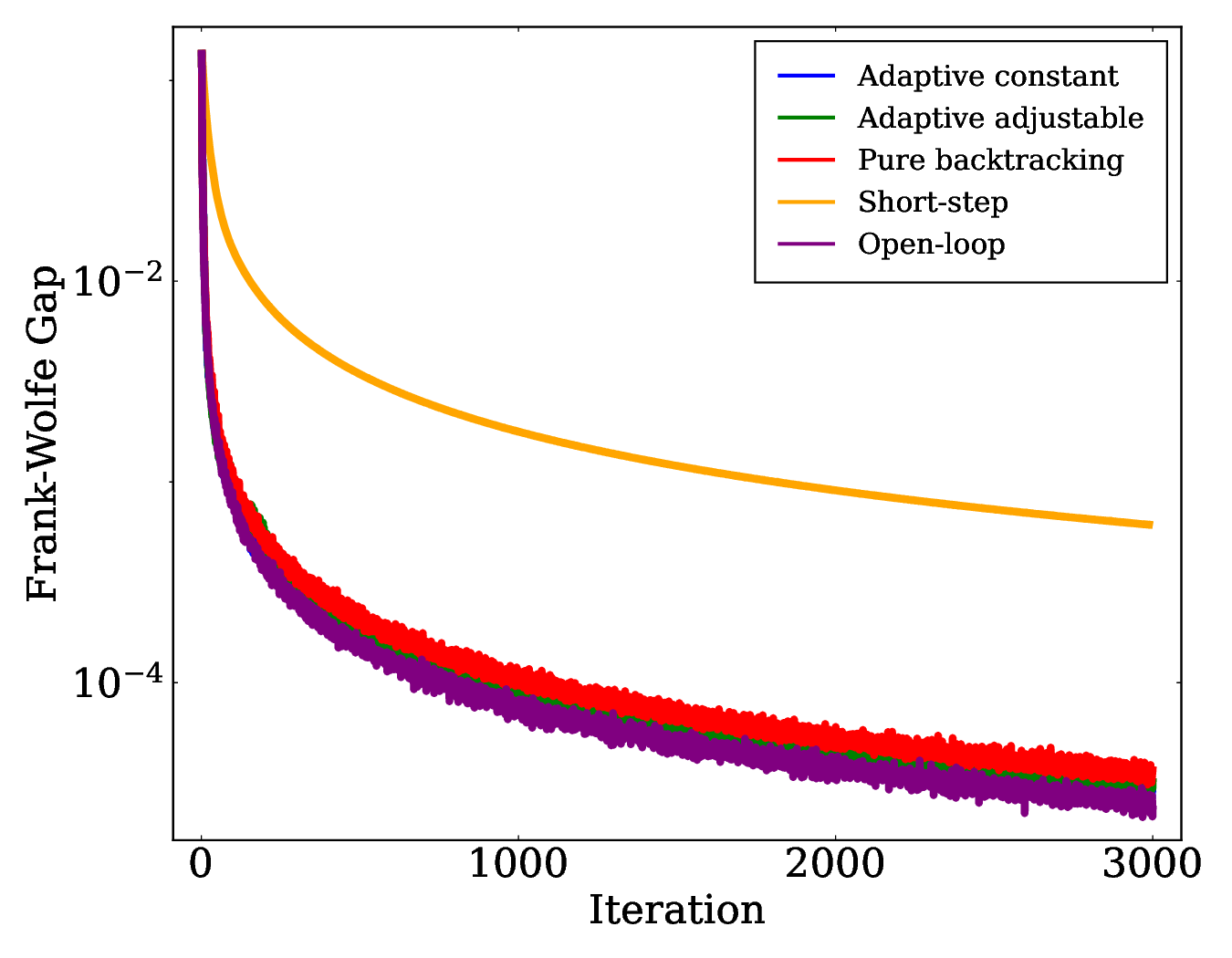}%
}\hfill
\subfloat[Lipschitz Constant Evolution]{%
  \includegraphics[width=0.3\linewidth]{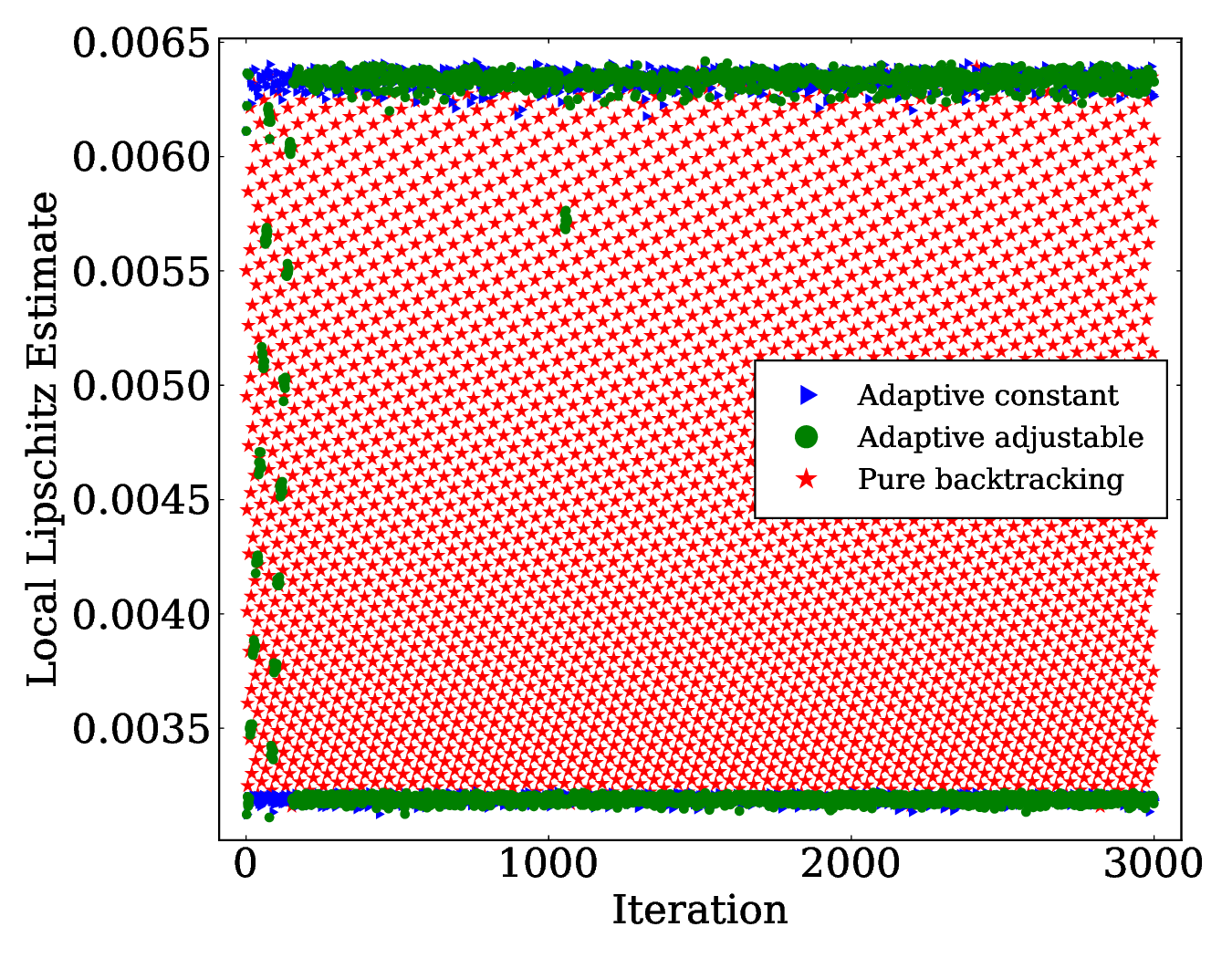}%
}\hfill
\subfloat[Gaps over Time]{%
   \includegraphics[width=0.3\linewidth]{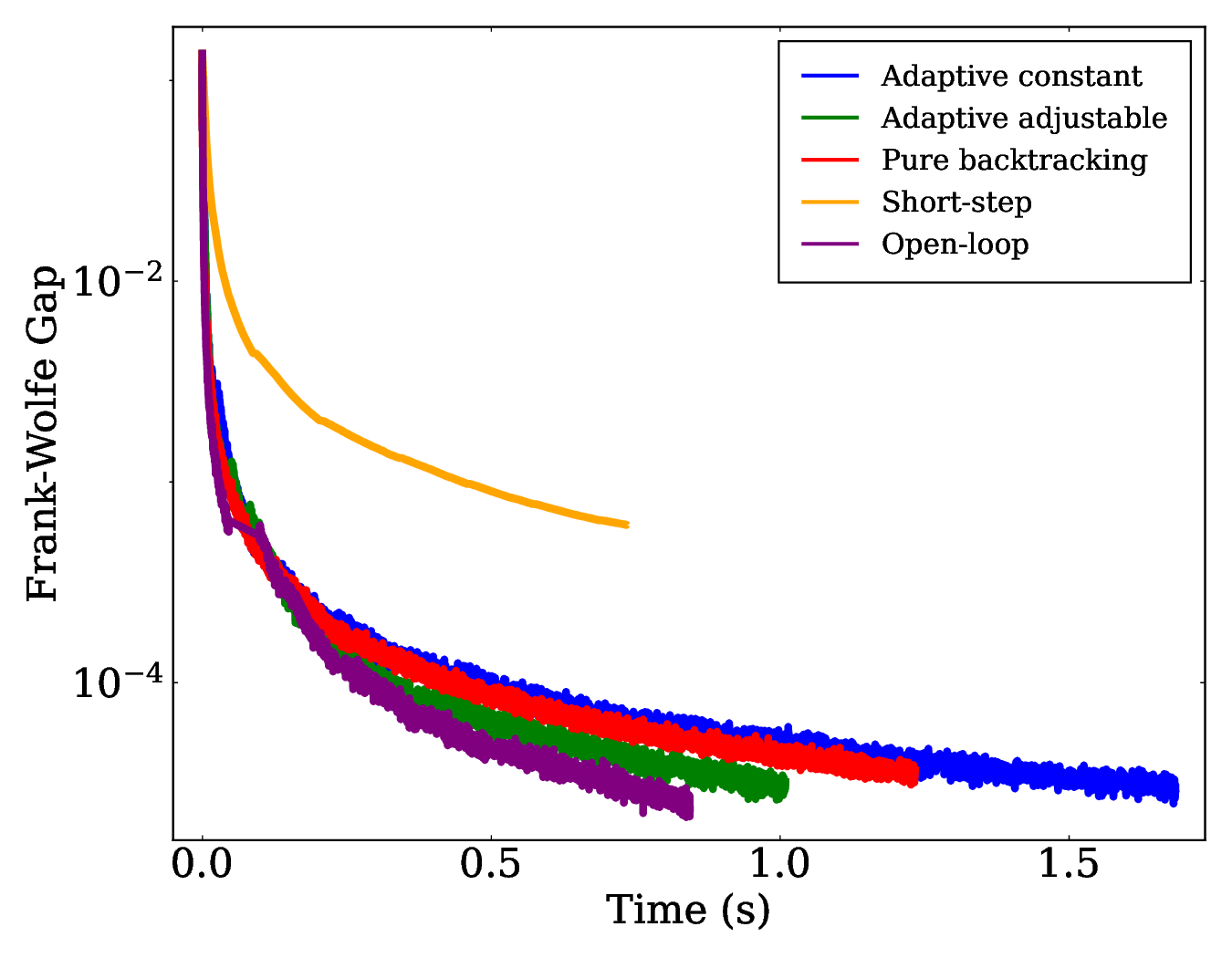}%
}
\caption{Convergence behavior and computational efficiency for Video Co-Localization problem}
\label{fig:VideoCoLocalization}
\end{figure}

\subsubsection{Collaborative Filtering}
%

Collaborative filtering lies at the heart of modern recommendation systems, leveraging patterns in user-item interactions to predict unobserved preferences. This task is formulated as the following constrained optimization problem
\begin{equation}\label{eq:Coll}
\begin{array}{cl}
\min\limits_{X \in \mathbb{R}^{m \times n}} & \dfrac{1}{|\mathcal{I}|} \sum\limits_{(i,j) \in \mathcal{I}} h_{\rho}(Y_{ij} - X_{ij}) \\
\mathrm{s.t.} & \|X\|_{\text{nuc}} \leq \tau
\end{array}
\end{equation}
where $Y \in \mathbb{R}^{m \times n}$ is the partially observed rating matrix, $\mathcal{I} \subset \{1,\dots,m\} \times \{1,\dots,n\}$ is the index set of observed entries, and $h_{\rho}$ is the Huber loss \cite{huber1992robust}:
\begin{equation}
h_{\rho}(\alpha) :=
\begin{cases} 
\frac{\alpha^2}{2}, & \text{if } |\alpha| \leq \rho \\
\rho (|\alpha| - \frac{\rho}{2}), & \text{if } |\alpha| > \rho
\end{cases}
\end{equation}
and $\|\cdot\|_{\text{nuc}}$ denotes the nuclear norm, given by
\begin{equation*}
\|X\|_{\text{nuc}} = \text{tr}\left(\sqrt{X^\top X}\right) = \sum_{i=1}^{\min\{m,n\}} \sigma_i(X)
\end{equation*}
where $\sigma_i(X)$ are the singular values of $X$.

We apply this formulation to the MovieLens 100K dataset \cite{harper2015movielens}. For this problem, the number of users is $m = 943$, the number of items is $n = 1,682$, and the number of observed entries is $|\mathcal{I}| = 10^5$. We set the parameters to $\rho = 1$ and $\tau = 5,000$.

Table~\ref{tab:CollaborativeFiltering} summarizes performance across step-size strategies. Open-loop remains the most efficient in practice, achieving the best objective value (0.1557) in the shortest time (1904.24s), despite not having the smallest Frank-Wolfe gap. Among adaptive methods, adaptive adjustable is the most accurate, attaining the smallest Frank-Wolfe gap (0.0134), and performs comparably to adaptive constant in objective value (0.1625 vs. 0.1622), while requiring less time. Pure backtracking is faster than both adaptive variants but yields a slightly worse objective and larger Frank-Wolfe gap. The short-step method lags significantly, producing a substantially worse objective and the largest Frank-Wolfe gap, indicating poor convergence. Figure~\ref{fig:CollaborativeFiltering} provides a visual illustration of the performance results and convergence behavior.

\begin{table}[htpb!]
\centering
\caption{Performance comparison of step-size strategies on Collaborative Filtering problem}
\label{tab:CollaborativeFiltering}
{
\begin{tabular}{lcccc}
\toprule
Step-size Strategy & Iterations & Time (s) & Objective Value & Frank-Wolfe gap\\
\midrule
Adaptive constant & 3000 & 2485.48 & 0.1622 & 0.0137 \\
Adaptive adjustable & 3000 & 2329.82 & 0.1625 & \textbf{0.0134} \\
Pure backtracking & 3000 & 2191.40 & 0.1638 & 0.0146 \\
Short-step & 3000 & 2250.36 & 0.1897 & 0.0422 \\
Open-loop & 3000 & \textbf{1904.24} & \textbf{0.1557} & 0.0155 \\
\bottomrule
\end{tabular}
}
\end{table}

It is important to note that, in problem~\eqref{eq:Coll} with the data we used, the objective function is smooth with gradient whose Lipschitz constant is $L = 1/|\mathcal{I}| = 10^{-5}$, which is an extremely small value. Step-size strategies that rely on the Lipschitz constant of the gradient become too large in this regime, whereas the open-loop rule $t_k = 2/(k+2)$ achieves faster and more effective convergence on this problem. This behavior is clearly reflected in Table~\ref{tab:CollaborativeFiltering} and Figure~\ref{fig:CollaborativeFiltering}.

\begin{figure}[htpb!]
\centering
\subfloat[Gaps in Log Scale]{
 \includegraphics[width=0.3\linewidth]{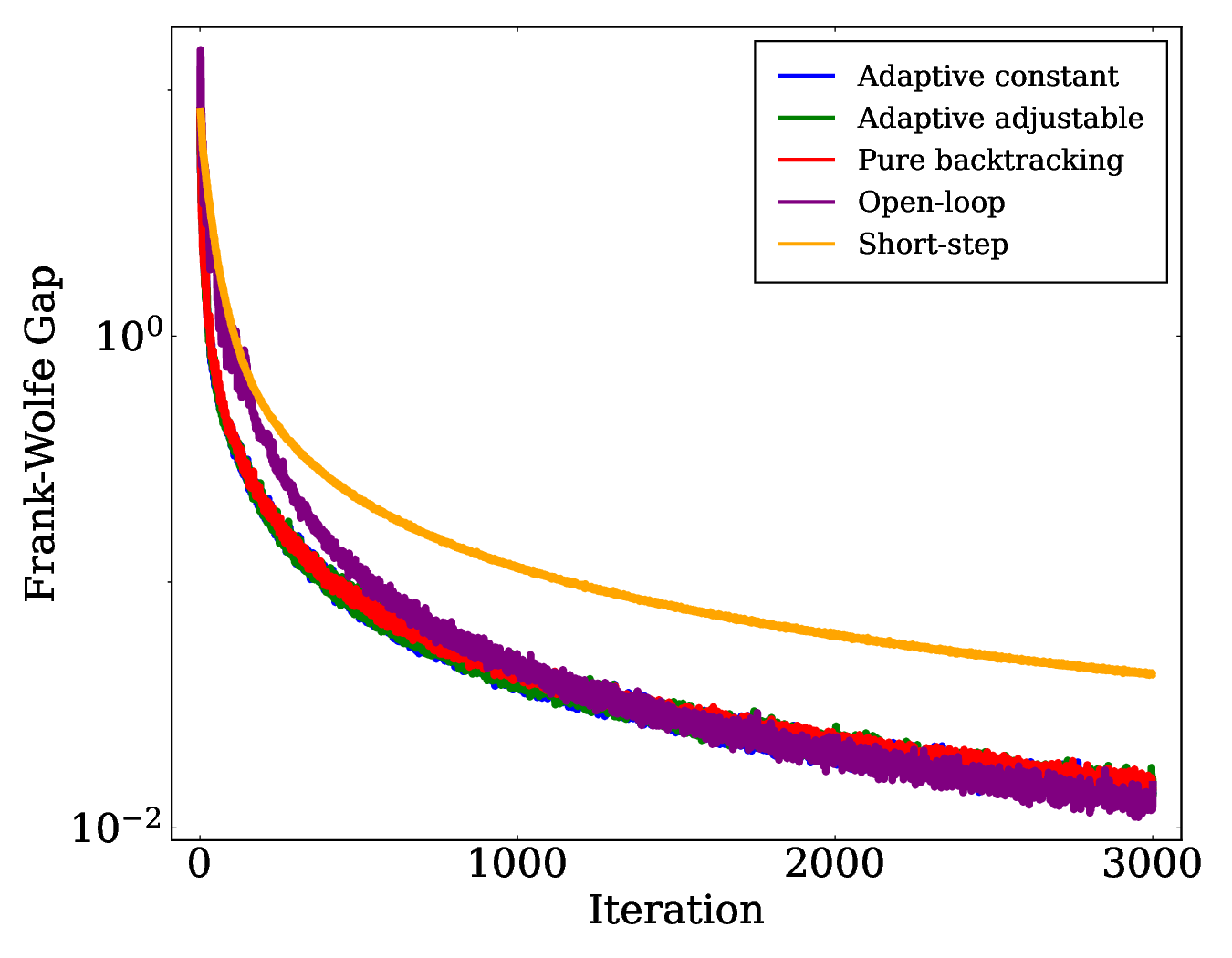}
}\hfill
\subfloat[Lipschitz Constant Evolution]{
 \includegraphics[width=0.3\linewidth]{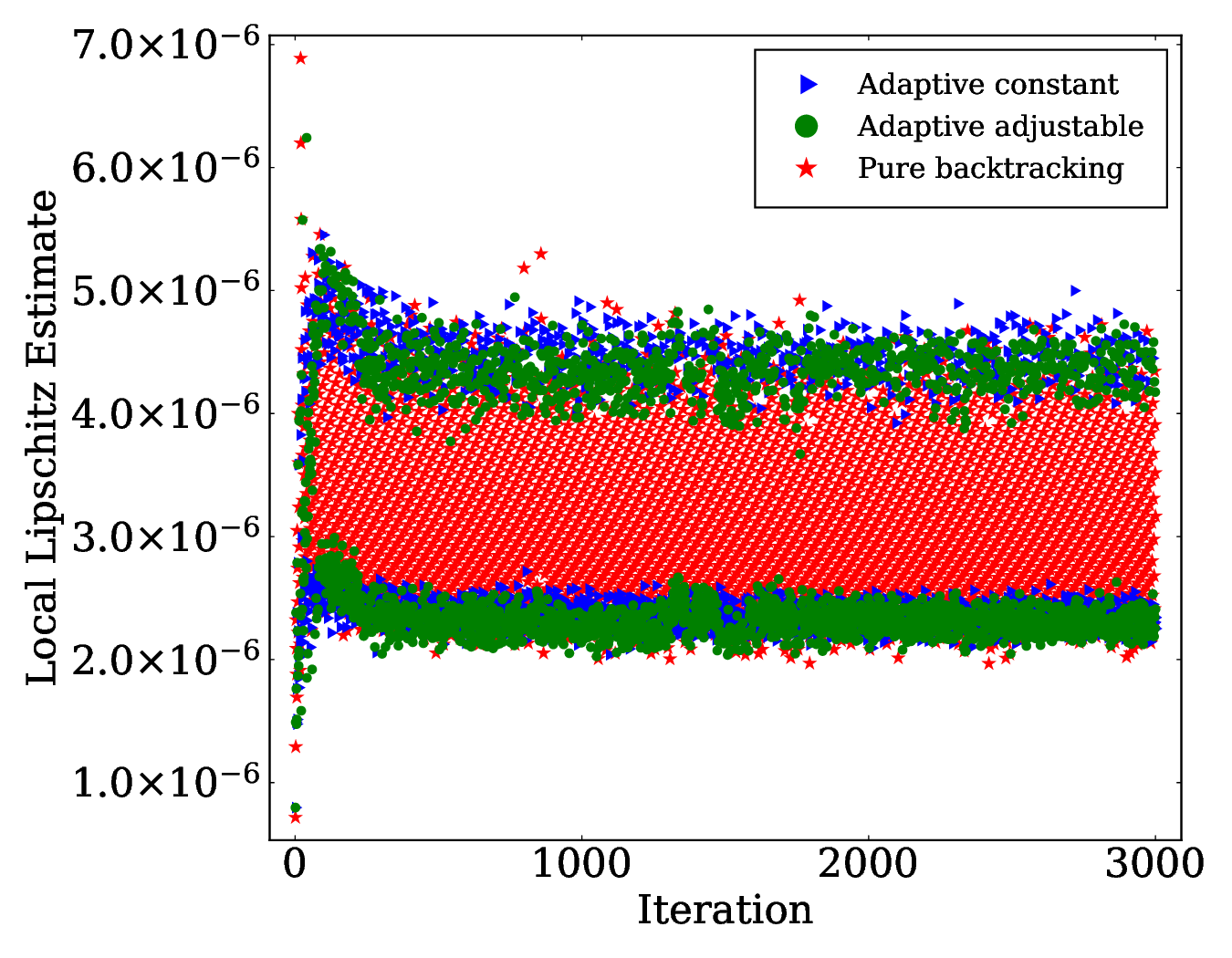}
}\hfill
\subfloat[Gaps over Time]{
  \includegraphics[width=0.3\linewidth]{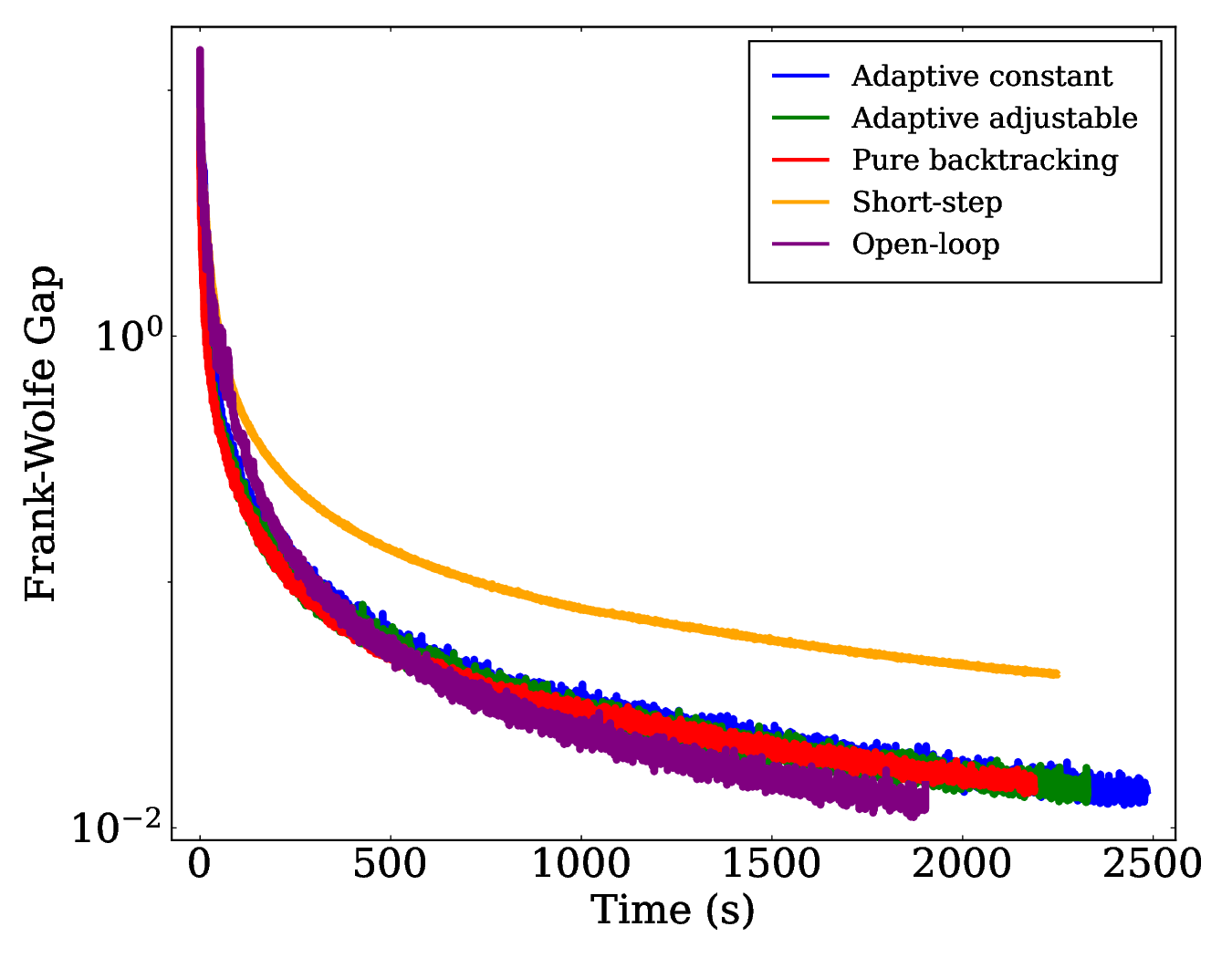}
}
\caption{Convergence behavior and computational efficiency for Collaborative Filtering problem}
\label{fig:CollaborativeFiltering}
\end{figure}


\subsubsection{Simplex Quadratic Optimization Problem (non-convex)}
%

In this part, we consider the standard quadratic optimization problem, which entails minimizing a homogeneous quadratic objective function over the unit simplex set, formulated as
\begin{equation}
\begin{array}{cl}
\min\limits_{x\geq 0} & x^\top Q x \\
\mathrm{s.t.} & \sum\limits_{i=1}^n x_i=1
\end{array}
\end{equation}
where $Q \in \mathbb{R}^{n \times n}$ is symmetric. When $Q$ is indefinite (with both positive and negative eigenvalues), the problem is non-convex and NP-hard \cite{ahmadi2022complexity,Khademi2025}. We use a symmetric indefinite matrix $Q$ of size $n = 900$ from \cite{Khademi2025}  to evaluate the Conditional Gradient algorithm using different step-size strategies.

The results are summarized in Table~\ref{tab:Standard_Quadratic_Problem}. The Adaptive adjustable strategy achieves the best objective value and the smallest Frank-Wolfe gap, demonstrating superior solution quality and convergence. The Pure backtracking method delivers nearly identical accuracy in terms of objective value and Frank-Wolfe gap, at a moderately higher computational cost. The Adaptive constant strategy produces a competitive objective but is slower than Adaptive adjustable. In contrast, the Short-step approach is the fastest but fails to converge to a meaningful solution, yielding a poor objective value and a very large Frank-Wolfe gap. The open-loop method is nearly as fast as Short-step but produces the worst Frank-Wolfe gap of all the methods. These performance trends are further illustrated in Figure~\ref{fig:StQO}, which depicts the convergence behavior and comparative efficiency of the different step-size strategies.

\begin{table}[htpb!]
\centering
\caption{Performance comparison of step-size strategies on non-convex Standard Quadratic Optimization problem}
\label{tab:Standard_Quadratic_Problem}
{
\begin{tabular}{lcccc}
\toprule
Step-size Strategy & Iterations & Time (s) & Objective Value & Frank-Wolfe gap \\
\midrule
Adaptive constant    & 3000 & 1.90 & -6.9860 & 0.0221 \\
Adaptive adjustable  & 3000 & 1.68 & \textbf{-7.0014} & \textbf{0.0065} \\
Pure backtracking    & 3000 & 2.06 & -7.0009 & 0.0068 \\
Short-step           & 3000 & \textbf{1.29} & -5.9066 & 0.9923 \\
Open-loop            & 3000 & 1.30 & -6.9158 & 0.1495 \\
\bottomrule
\end{tabular}
}
\end{table}

\begin{figure}[htpb!]
\centering
\subfloat[Gaps in Log Scale]{
 \includegraphics[width=0.3\linewidth]{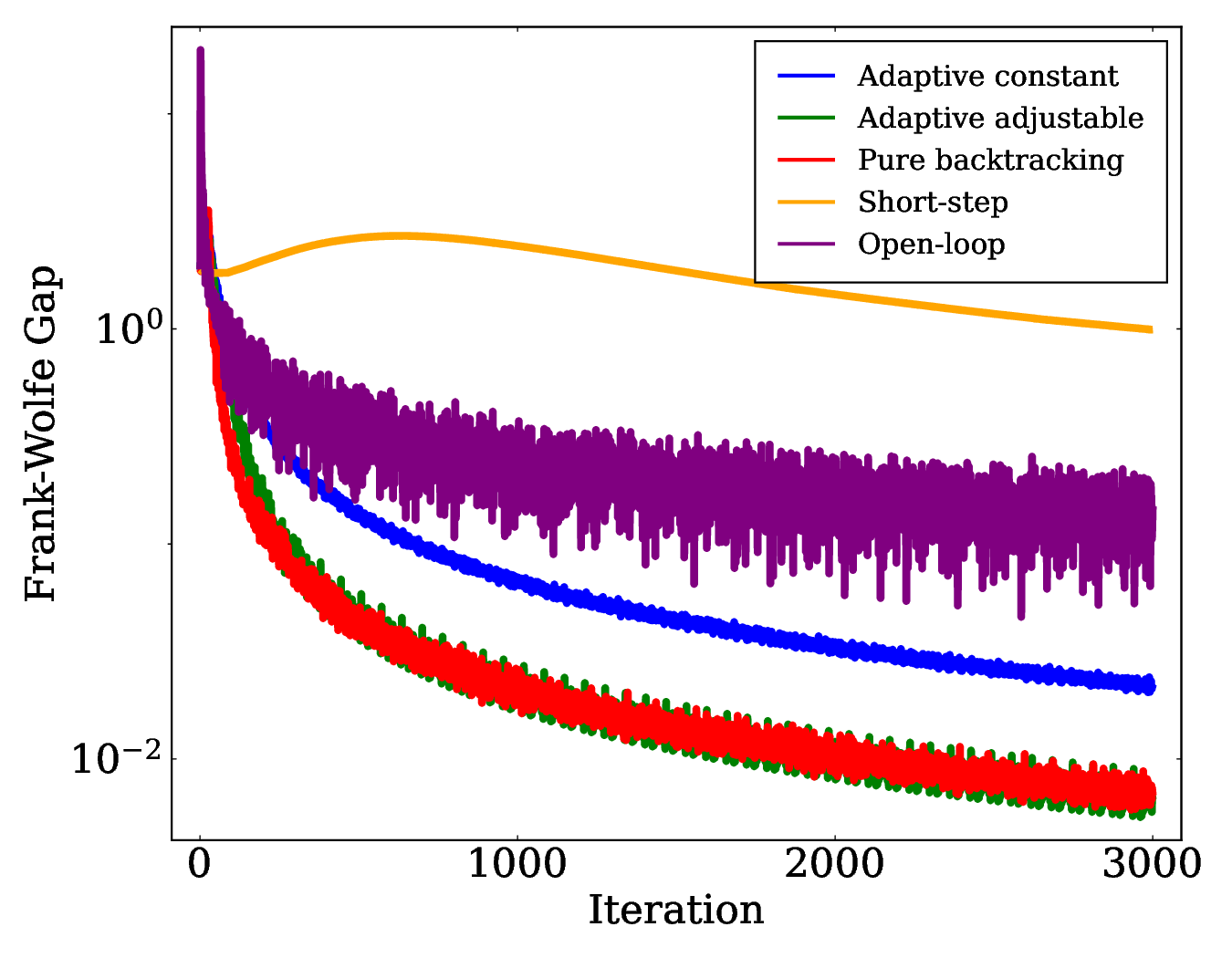}
}\hfill
\subfloat[Lipschitz Constant Evolution]{
 \includegraphics[width=0.3\linewidth]{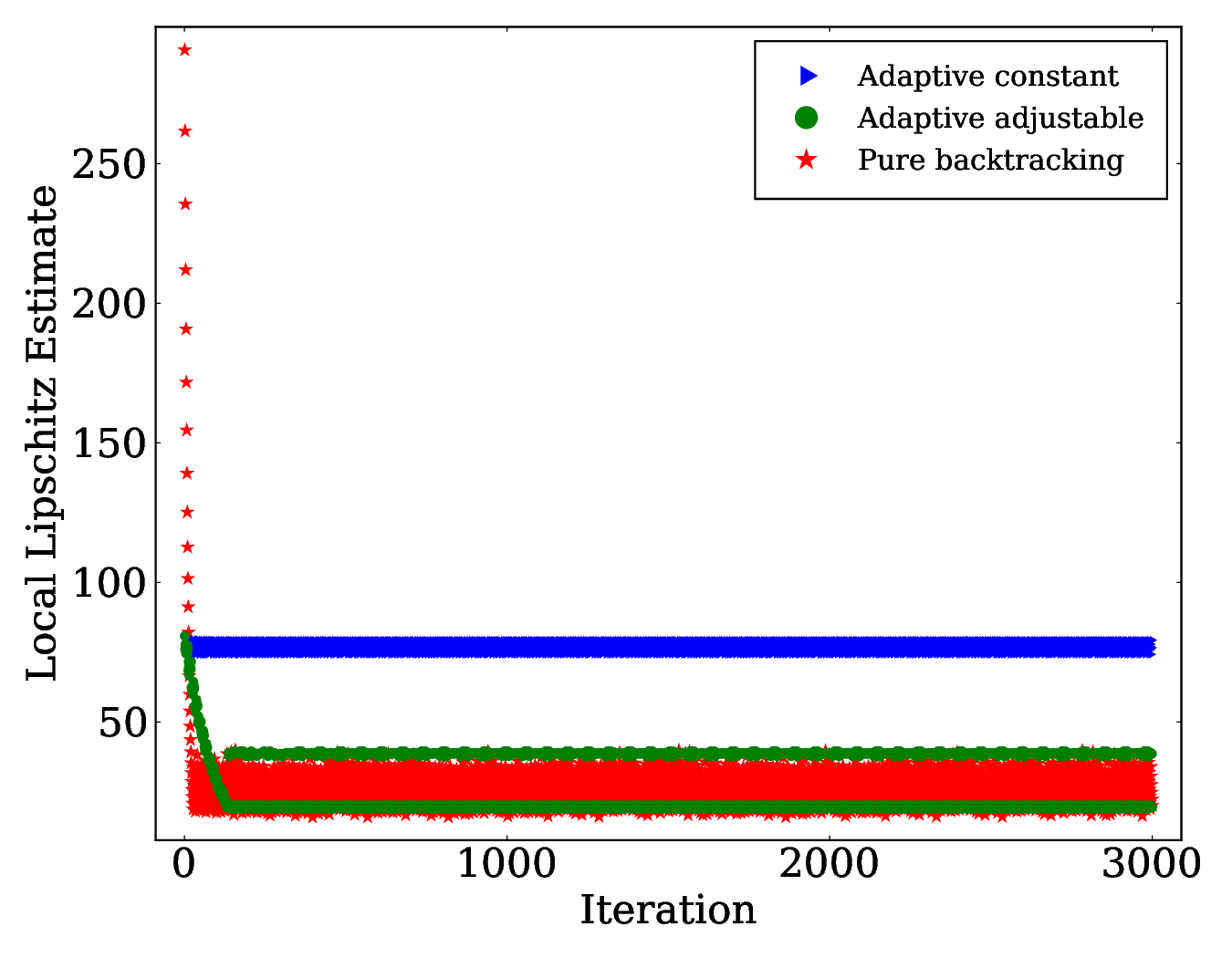}
}\hfill
\subfloat[Gaps over Time]{
 \includegraphics[width=0.3\linewidth]{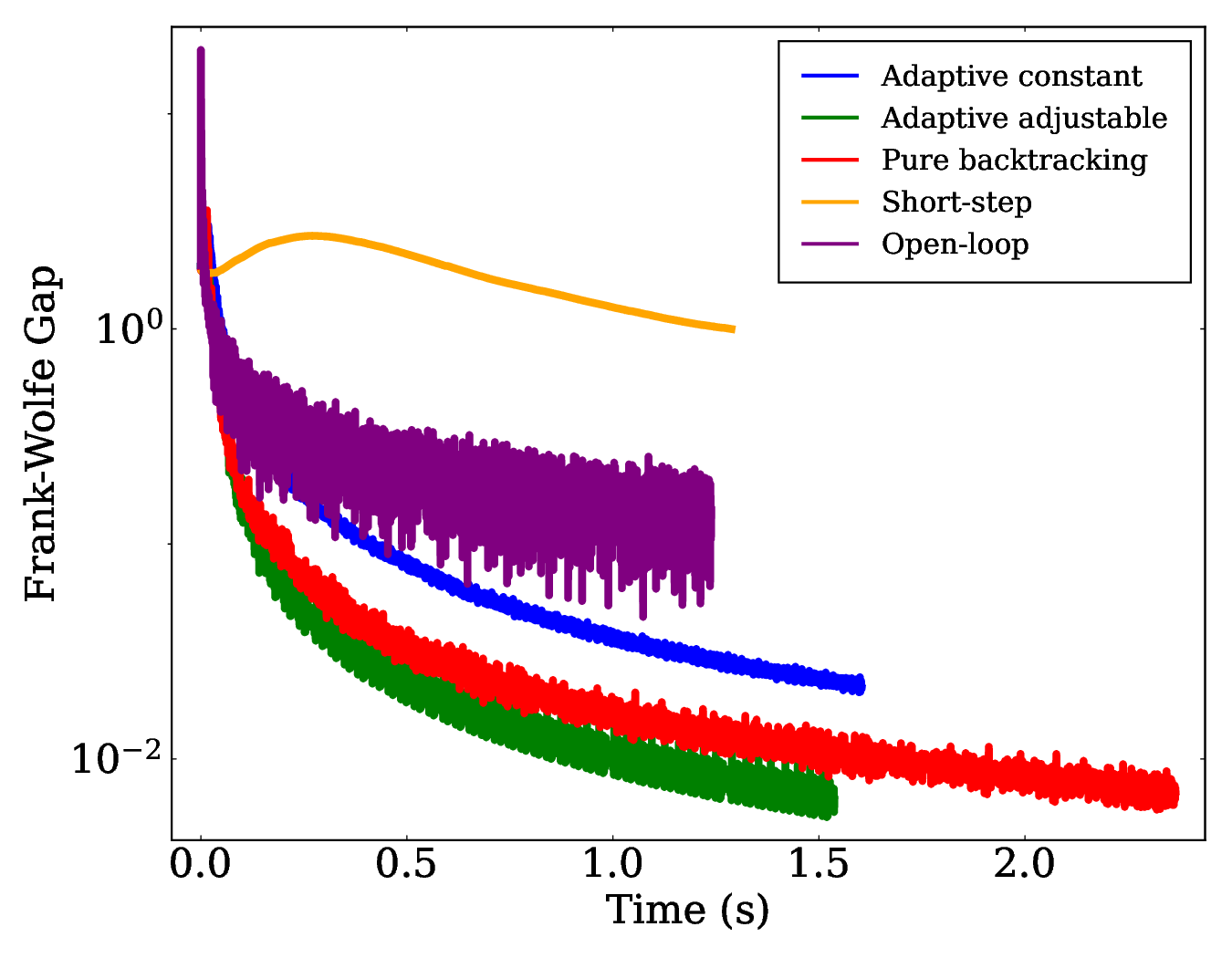}
}
\caption{Convergence behavior and computational efficiency for Standard Quadratic Optimization problem}
\label{fig:StQO}
\end{figure}


Until now, we have observed that for the Conditional Gradient method, the adaptive adjustable step-size strategy outperforms other adaptive step-size methods (pure backtracking and constant) in most cases, achieving faster convergence and competitive gaps. While the open-loop step-size method occasionally performs well, its inconsistency makes it less reliable. The short-step-size method consistently underperforms and fails to converge due to overly conservative steps. This makes the adaptive adjustable method the most reliable and efficient choice for both convex and non-convex problems that we considered in this part.

\subsection{Unconstrained Problems}
We analyze the impact of step-size on the performance of \nameref{alg:ACGD} applied to unconstrained optimization problems. We evaluate three step-size strategies—adaptive constant, adaptive adjustable, and pure backtracking. The process terminates when the gradient norm falls below a tolerance of $10^{-5}$ or after a maximum of 3,000 iterations, whichever occurs first.

\subsubsection{Least Squares Problem Under Different Geometries}
%
%
The least squares problem is a fundamental concept in linear algebra. It is typically used to find the best (approximate) solution to a system of linear equations. The problem is formally stated as
\begin{equation}
\min\limits_{x \in \mathbb{R}^n} \|b - Ax\|_2^2
\end{equation}
where $A \in \mathbb{R}^{m \times n}$ and $b \in \mathbb{R}^m$. For our numerical experiments, we synthetically generate $A$ and $b$ with dimensions $m=20,000$ and $n=1,000$. 

Table \ref{tab:Least_Squares_Problem} compares the performance of various step-size strategies, while Figure \ref{fig:Least Squares} illustrates that adaptive methods efficiently reduce the gradient norm. These results demonstrate that adaptive step-size strategies achieve optimal solutions with superior computational efficiency compared to pure backtracking in the Normalized Steepest Descent algorithm.

For estimating a local Lipschitz constant and performing backtracking, we used $\ell^2$ norm. We observe that, in terms of runtime, objective value, and gradient (LMO dual) norm, the $\ell^2$ geometry for the LMO yields better performance.
\begin{table}[htpb!]
\centering
\caption{Performance comparison of step-size strategies on the unconstrained Least Squares problem with different Geometry.}
\label{tab:Least_Squares_Problem}
\begin{tabular}{clcccc}
\toprule
LMO Norm & Step-size Strategy & Iterations & Time (s) & Objective Value & Gradient (LMO Dual) Norm  \\
\midrule
& Adaptive constant   & 698  & 24.47 & \textbf{1.8857} & \textbf{0.0000} \\
$\ell^2$ & Adaptive adjustable & 742  & \textbf{19.95} & \textbf{1.8857} & \textbf{0.0000} \\
& Pure backtracking   & 3000 & 86.39 & \textbf{1.8857} & 0.0001 \\
\midrule
& Adaptive constant   & 3000 & 95.40 & 1.9217 & \textbf{0.0015} \\
$\ell^1$ & Adaptive adjustable & 3000 & \textbf{68.65} & \textbf{1.9193} & 0.0016 \\
& Pure backtracking   & 3000 & 85.08 & 1.9216 & 0.0016 \\
\midrule
& Adaptive constant   & 3000 & 103.40 & \textbf{1.9039} & 0.0953 \\
$\ell^{\infty}$ & Adaptive adjustable & 3000 & \textbf{69.59} & 1.9046 & 0.0950 \\
& Pure backtracking   & 3000 & 86.01 & 1.9094 & \textbf{0.0868} \\
\bottomrule
\end{tabular}
\end{table}


\begin{figure}[htpb!]
\centering
\subfloat[Gradient Norms in Log Scale]{
  \includegraphics[width=0.3\linewidth]{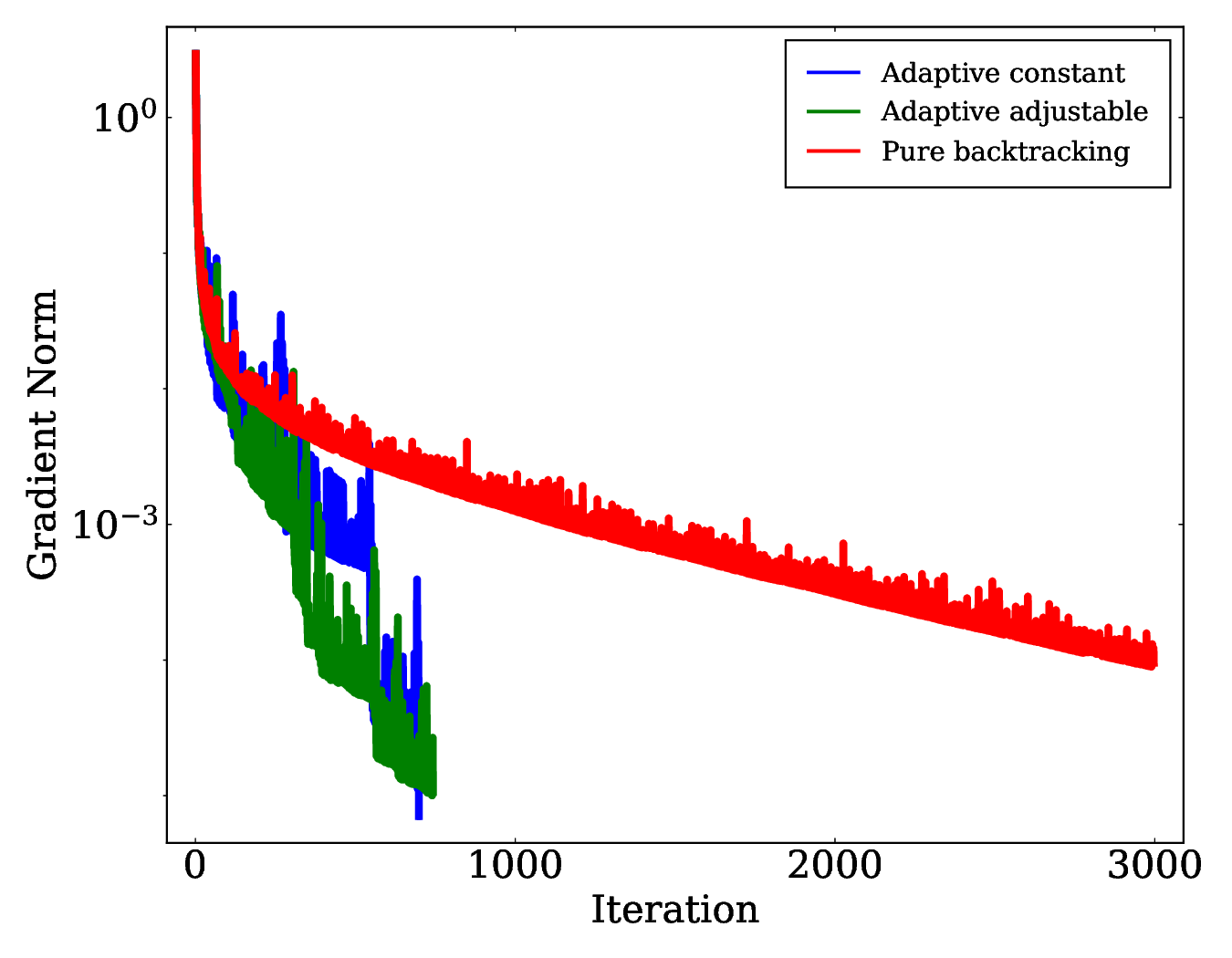}
}\hfill
\subfloat[Lipschitz Constant Evolution]{
   \includegraphics[width=0.3\linewidth]{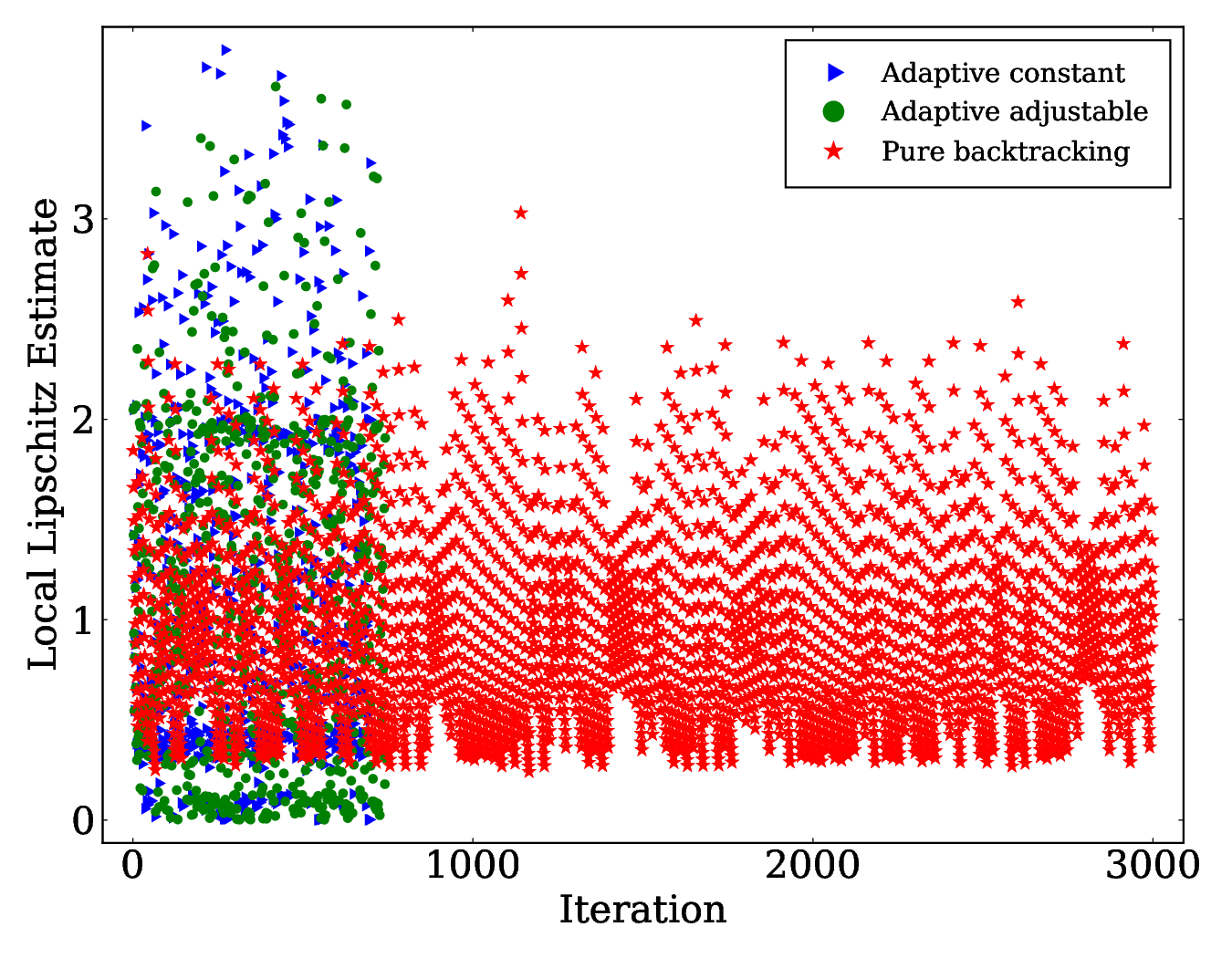}
}\hfill
\subfloat[Gradient Norms over Time]{
    \includegraphics[width=0.3\linewidth]{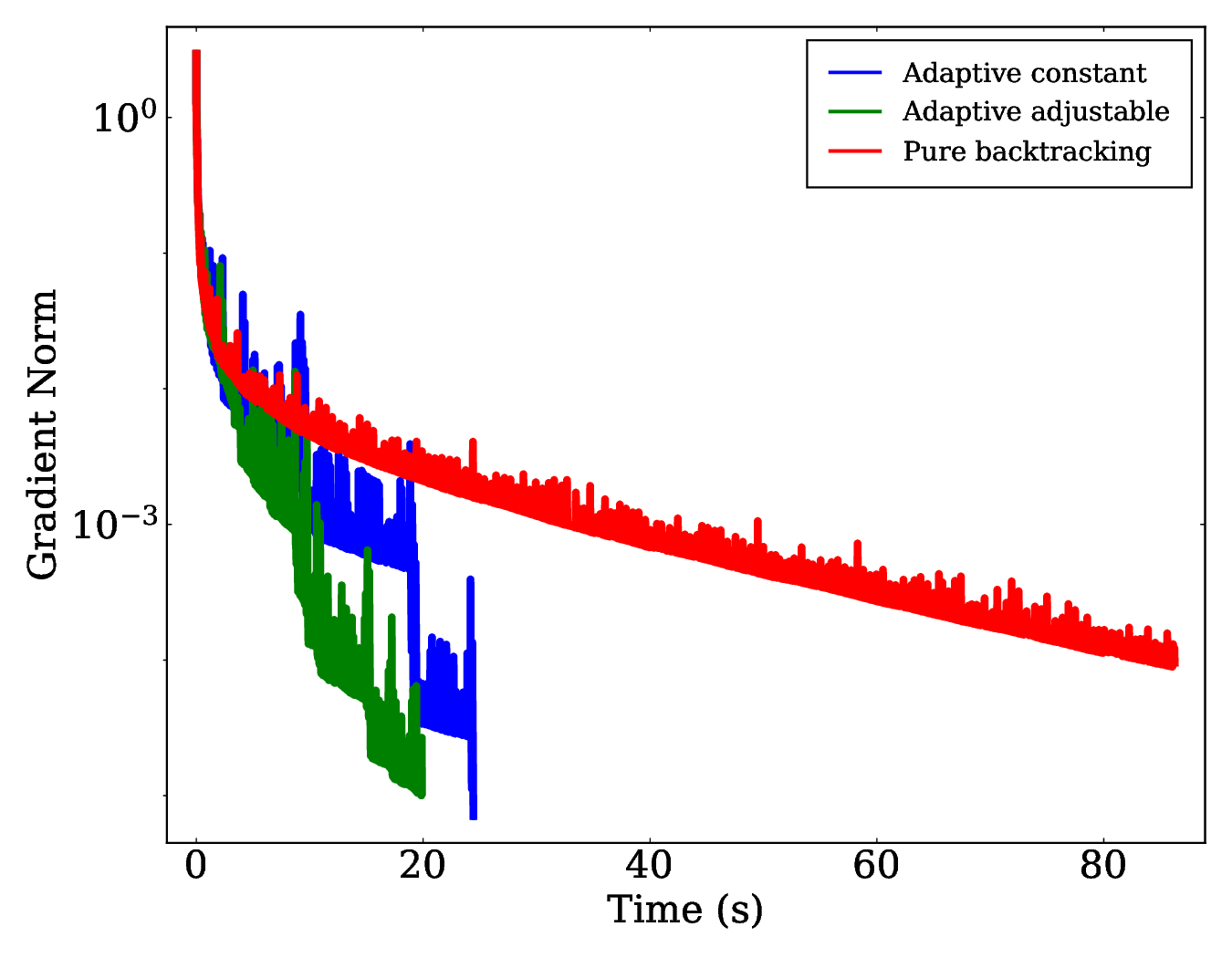}
}
\caption{Convergence behavior and computational efficiency for the Least Squares problem (Backtracking and LMO $\ell^2$).}
\label{fig:Least Squares}
\end{figure}

\subsubsection{Least Squares Sigmoid  Problem (non-convex)}
%
The least squares sigmoid problem involves minimizing a non-convex loss function, specifically the squared error between observed outputs and a sigmoid function. It is formulated as
\begin{equation}
\begin{array}{cl}
\min\limits_{x\in \mathbb{R}^n} & \dfrac{1}{m} \sum\limits_{i=1}^m \left(y_i-\dfrac{1}{1+\exp \left(-x^{\top} a_i\right)}\right)^2.
\end{array}
\end{equation}
For our numerical experiments, we use the LIBSVM \texttt{a1a} dataset \cite{chang2011libsvm} with $m=1,618$ samples and $n=119$ features.

Table \ref{tab:LLS_a1a} compares the performance of different step-size strategies. The adaptive adjustable method is the most time-efficient, achieving a better objective value in 0.88 seconds. The adaptive constant method matches the same objective value of 0.0941 but requires more computational time (1.28 seconds). In contrast, pure backtracking is less efficient, taking 1.25 seconds and yielding a slightly higher objective value of 0.0956.

Figure \ref{fig:LLS_a1a} visually complements the numerical findings in Table \ref{tab:LLS_a1a}, illustrating the convergence behavior and computational efficiency of our adaptive step-size strategies.
\begin{table}[htpb!]
\centering
\caption{Performance comparison of step-size strategies on non-convex least squares sigmoid problem}
\label{tab:LLS_a1a}
{
\begin{tabular}{lcccc}
\toprule
Step-size Strategy & Iterations & Time (s) & Objective Value & Gradient Norm  \\
\midrule
Adaptive constant  & 3000 & 1.28 & \textbf{0.0941} & 0.0001 \\
Adaptive adjustable & 3000 & \textbf{0.88} & \textbf{0.0941} & \textbf{0.0000} \\
Pure backtracking  & 3000 & 1.25 & 0.0956 & 0.0001 \\
\bottomrule
\end{tabular}
}
\end{table}

\begin{figure}[htpb!]
\centering
\subfloat[Gradient Norms in Log Scale]{
  \includegraphics[width=0.3\linewidth]{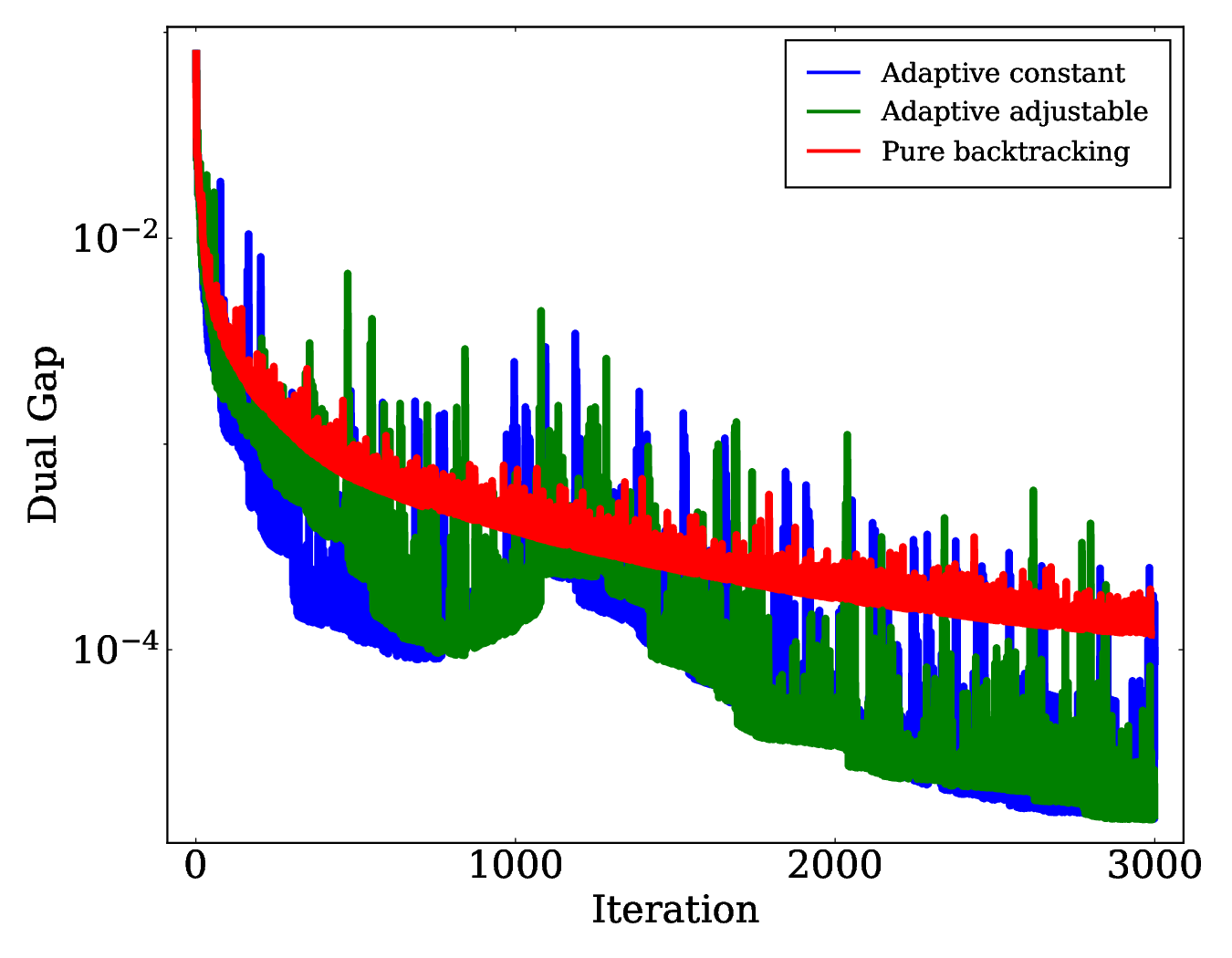}
}\hfill
\subfloat[Lipschitz Constant Evolution]{
   \includegraphics[width=0.3\linewidth]{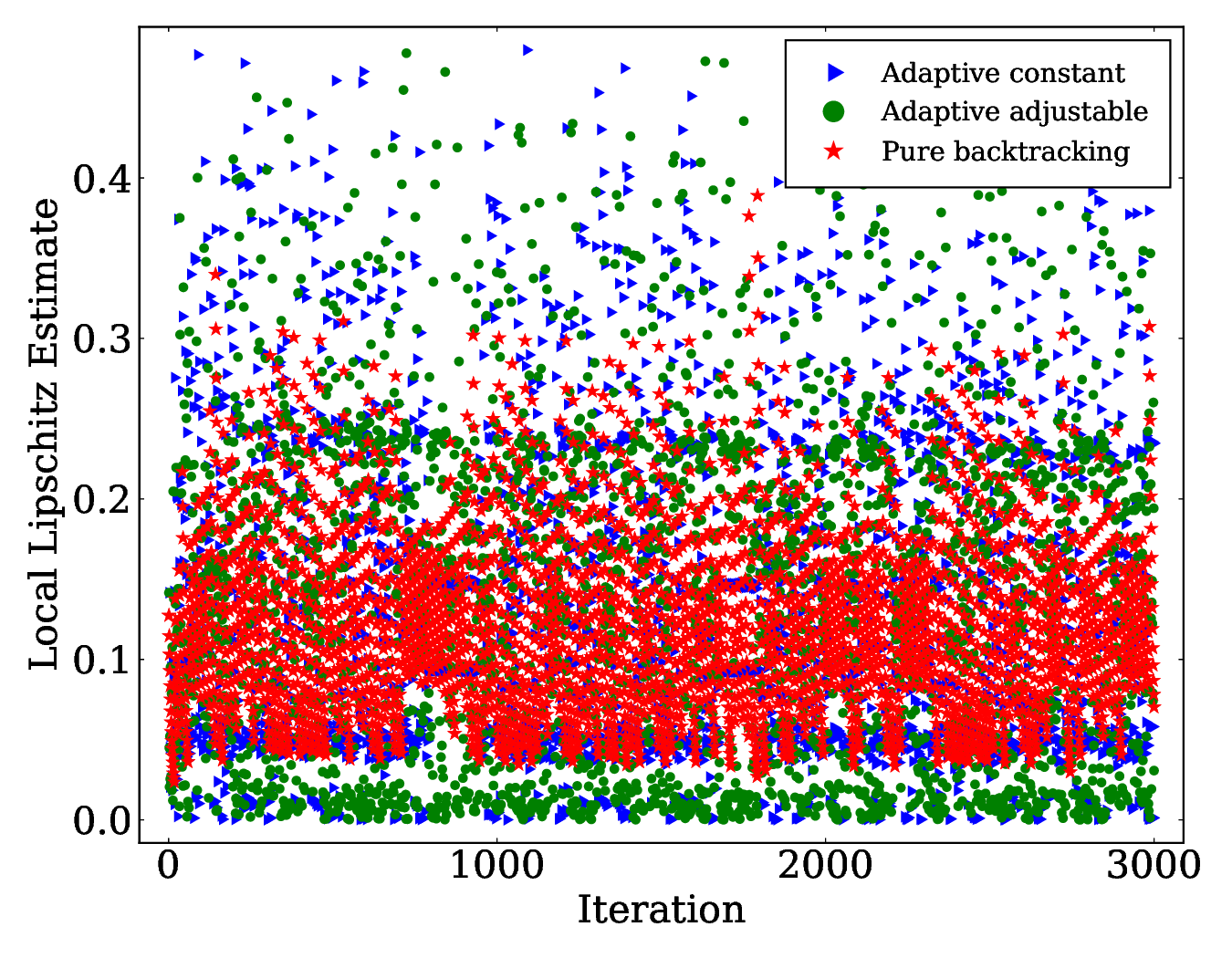}
}\hfill
\subfloat[Gradient Norms over Time]{
    \includegraphics[width=0.3\linewidth]{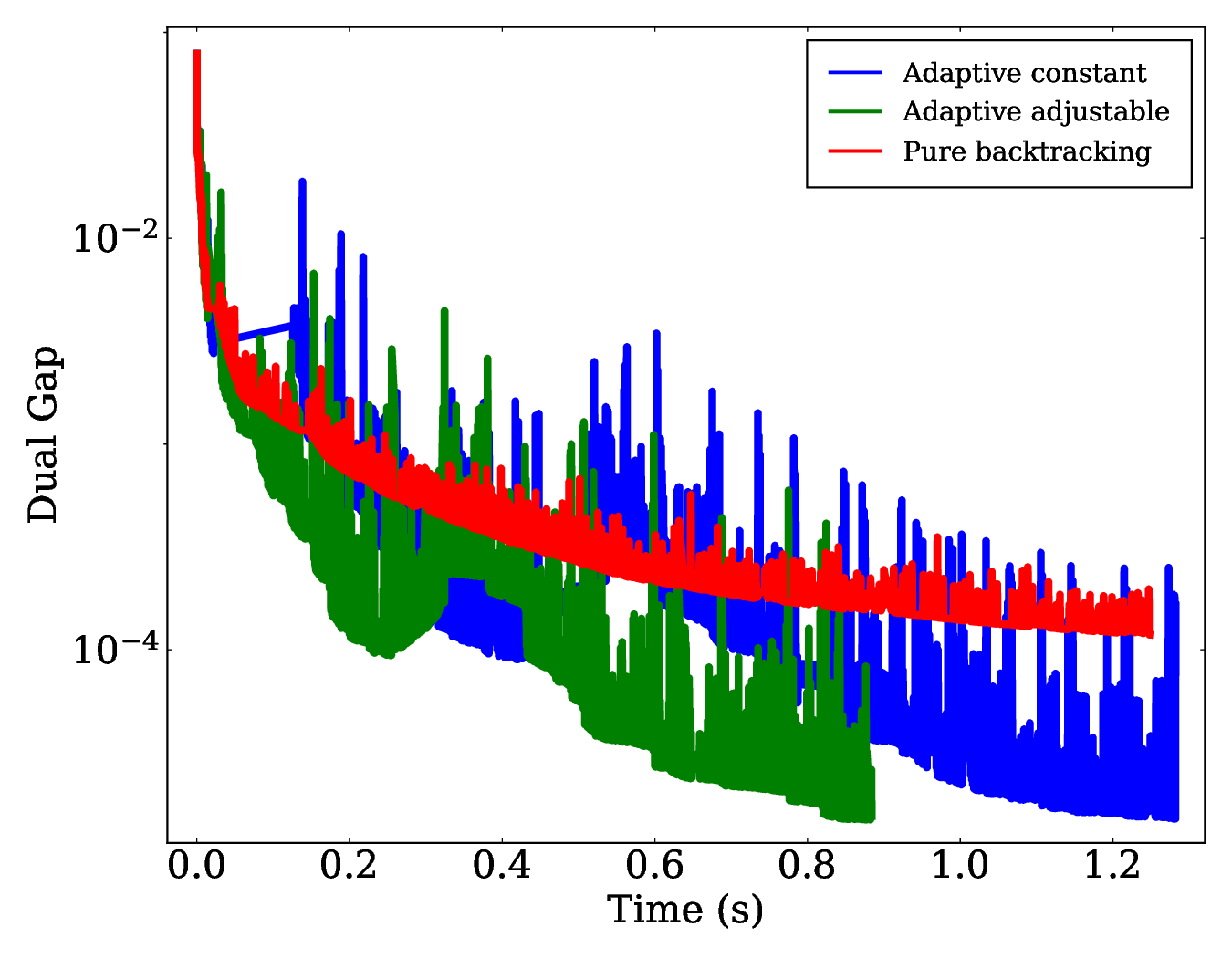}
}
\caption{Convergence behavior and computational efficiency for non-convex least squares sigmoid problem}
\label{fig:LLS_a1a}
\end{figure}

\subsubsection{Sum of Squares Problem (non-convex)}
%

Next, we consider a non-convex sum of squares problem as follows
\begin{equation}
\begin{array}{cl}
\min\limits_{x\in \mathbb{R}^n} & \left(x_1-3\right)^2+\sum\limits_{i=2}^n\left(x_1-3-2\left(x_1+\ldots+x_i\right)^2\right)^2. 
\end{array}
\end{equation}
For our numerical experiments, we set the dimension to $n=200$.

Table \ref{tab:Sos} presents the performance comparison results. The adaptive constant method achieves the best objective value and the lowest gradient norm, indicating superior solution quality. However, it is the slowest in terms of runtime. The adaptive adjustable method offers a faster runtime but results in a higher objective value and gradient norm. The pure backtracking method is the fastest but yields the highest objective value and gradient norm, suggesting convergence to a significantly poorer local minimum. The consistently large maximum and mean Lipschitz constant values across all methods reflect the challenging nature of this non-convex problem.

\begin{table}[htpb!]
\centering
\caption{Performance comparison of step-size strategies on the non-convex Sum of Squares problem.}
\label{tab:Sos}
{
\begin{tabular}{lcccc}
\toprule
Step-size Strategy & Iterations & Time (s) & Objective Value & Gradient Norm \\
\midrule
Adaptive constant  & 3000 & 13.17 & \textbf{0.6239} & \textbf{0.9193} \\
Adaptive adjustable& 3000 & 12.45 & 4.3445 & 6.0382  \\
Pure backtracking  & 3000 & \textbf{12.38} & 66.9148 & 47.8900  \\
\bottomrule
\end{tabular}
}
\end{table}

Figure \ref{fig:Sos} illustrates the convergence behavior. Subfigure \ref{fig:Sos}(b) shows that, for $k \geq 100$, the adaptive constant method yields more stable and lower Lipschitz estimates compared to the pure backtracking method, which exhibits greater variability and higher values.

\begin{figure}[htpb!]
\centering
\subfloat[Gradient Norms in Log Scale]{
  \includegraphics[width=0.3\linewidth]{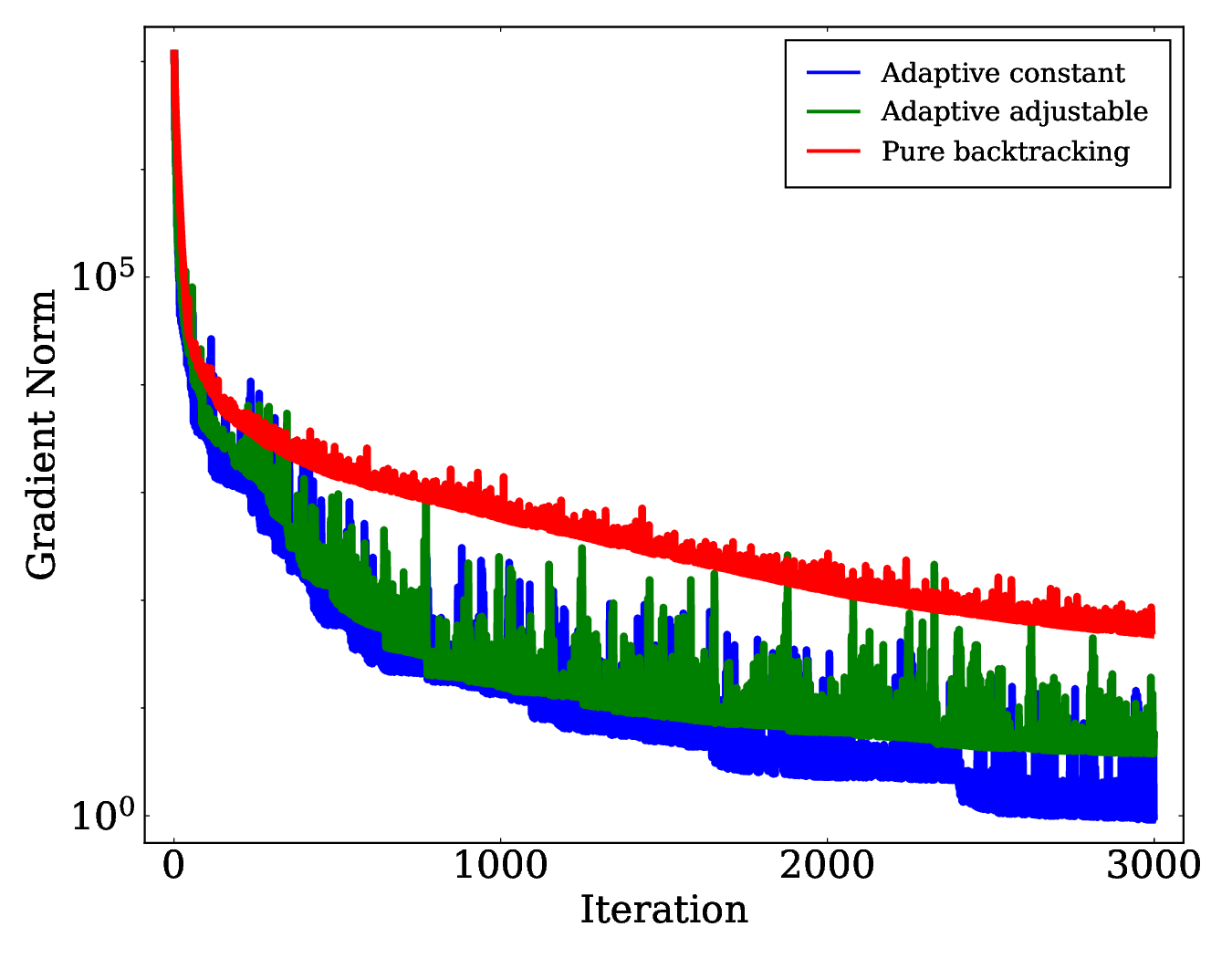}
}\hfill
\subfloat[Lipschitz Constant Evolution ($k\geq 100$)]{
  \includegraphics[width=0.3\linewidth]{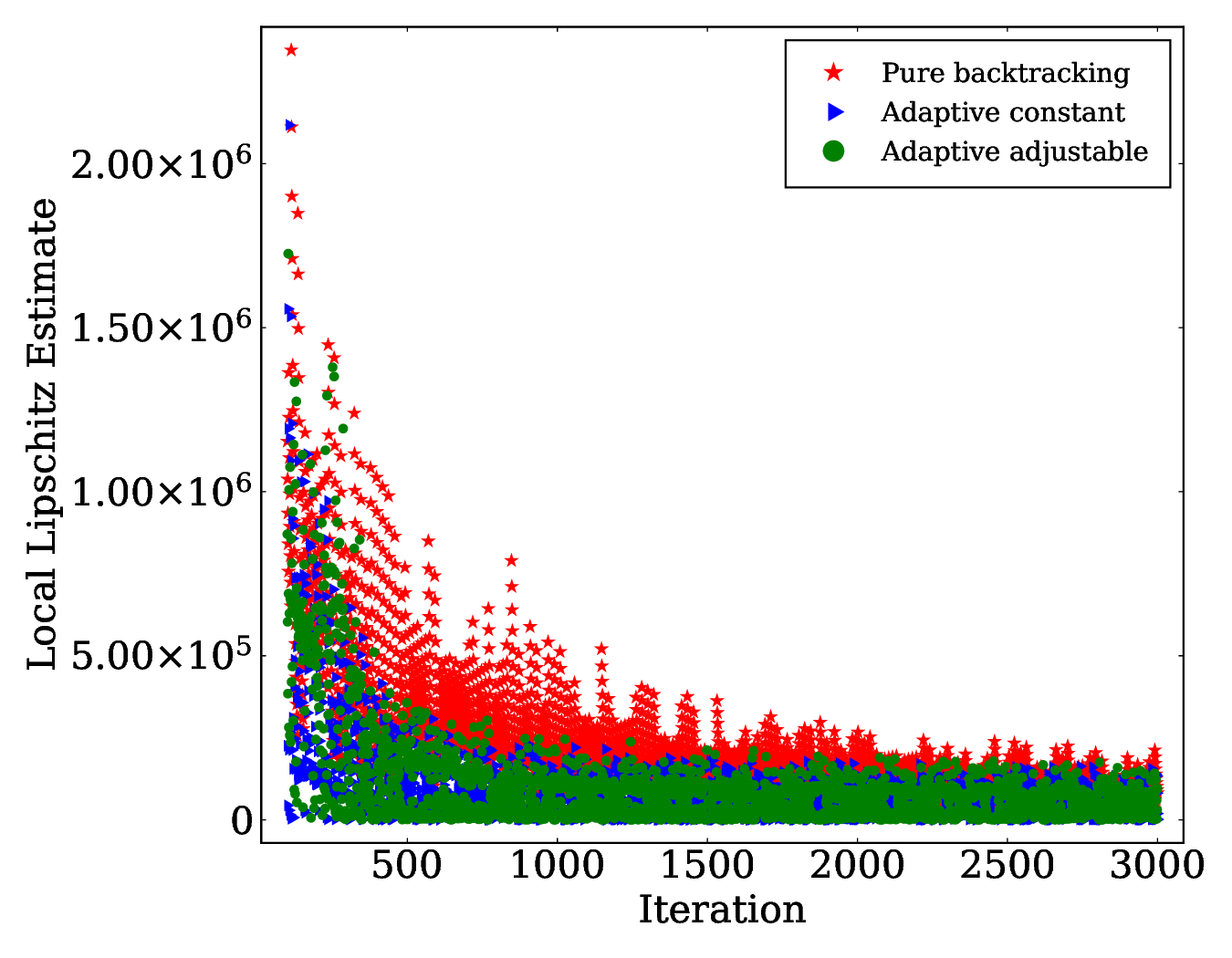}
}\hfill
\subfloat[Gradient Norms over Time]{
   \includegraphics[width=0.3\linewidth]{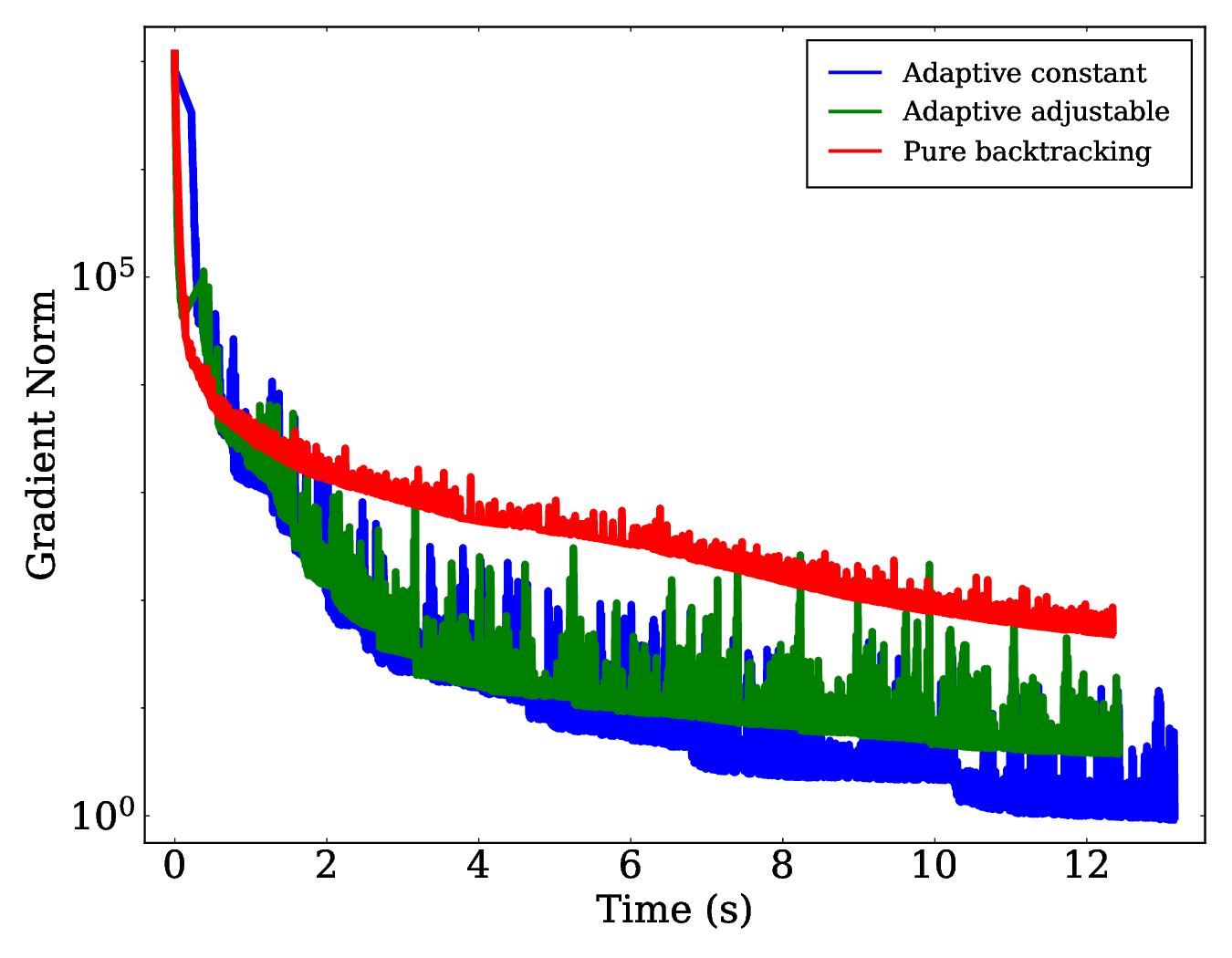}
}
\caption{Convergence behavior and computational efficiency for the non-convex Sum of Squares problem.}
\label{fig:Sos}
\end{figure}

\subsubsection{Rosenbrock Function (non-convex)}

%

We also evaluate our method on the Rosenbrock function, a classic benchmark problem renowned for its challenging non-convex structure \cite{rosenbrock1960automatic}. Consider the following problem
\begin{equation} \label{eq:rosenbrock}
\min\limits_{x \in \mathbb{R}^n} \sum_{i=1}^{n-1} \left( 100 (x_{i+1} - x_i^2)^2 + (x_i - 1)^2 \right).
\end{equation}
Its narrow, parabolic valley poses significant convergence difficulties. For numerical experiments, we set $n=1,000$.

Table \ref{tab:Rosenbrock_Function} shows that the adaptive adjustable method achieves the best performance with the lowest objective value and gradient norm in the fastest time. The adaptive constant method is slower with a higher objective value and gradient norm. Pure backtracking is slightly faster than adaptive constant but yields the worst objective value and gradient norm.

These observations are visually supported by Figure \ref{fig:Rosenbrock_Function}, which illustrates the convergence behavior and computational efficiency.
\begin{table}[htpb!]
\centering
\caption{Performance comparison of step-size strategies on the non-convex Rosenbrock function.}
\label{tab:Rosenbrock_Function}
{
\begin{tabular}{lcccc}
\toprule
Step-size Strategy & Iterations & Time (s) & Objective Value & Gradient Norm  \\
\midrule
Adaptive constant  & 3000 & 82.37 & 928.6291 & 3.2603 \\
Adaptive adjustable& 3000 & \textbf{81.19} & \textbf{923.9158} & \textbf{2.6786}  \\
Pure backtracking  & 3000 & 81.80 & 979.7630 & 3.5251  \\
\bottomrule
\end{tabular}
}
\end{table}

\begin{figure}[htpb!]
\centering
\subfloat[Gradient Norms in Log Scale]{
  \includegraphics[width=0.3\linewidth]{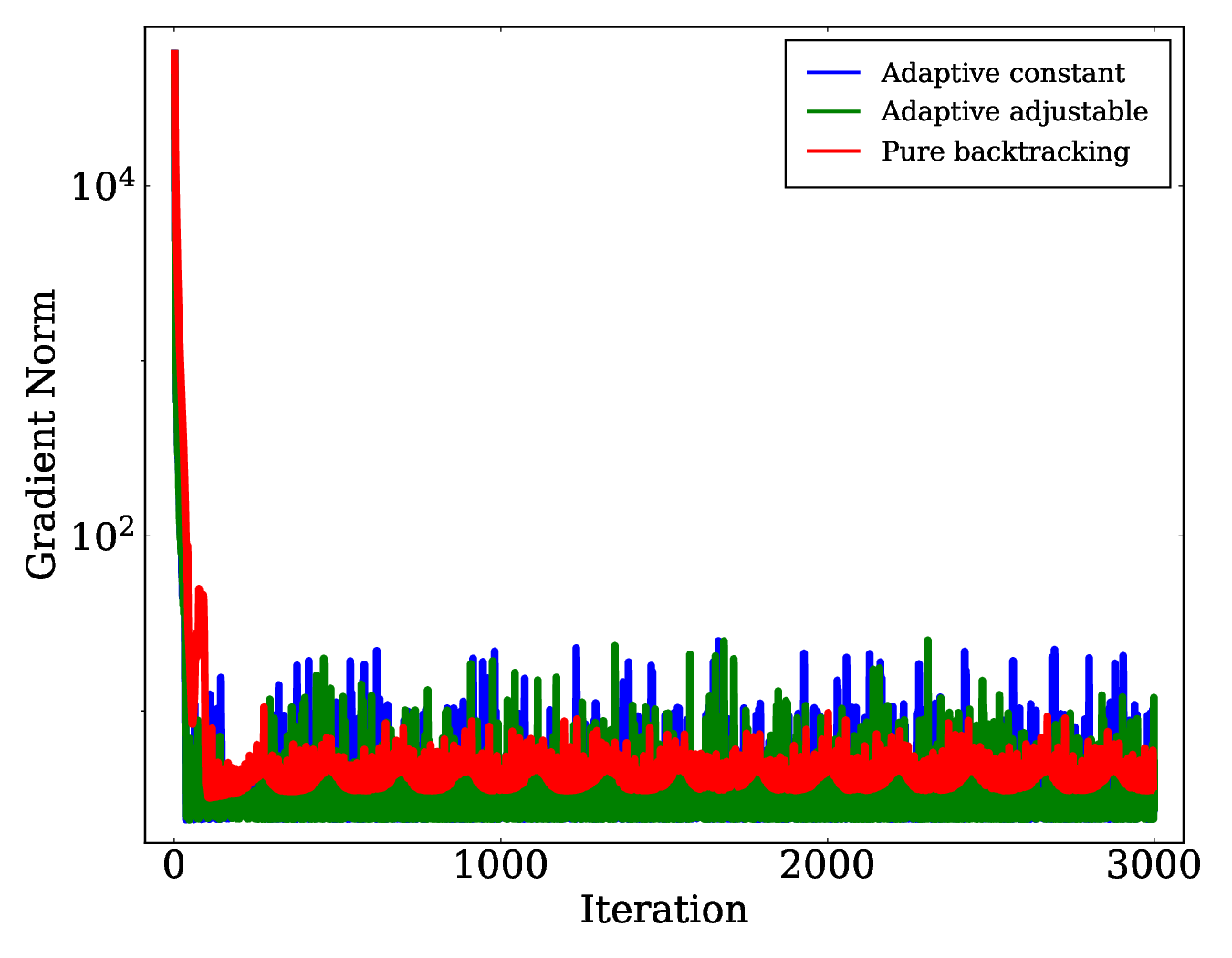}
}\hfill
\subfloat[Lipschitz Constant Evolution]{
  \includegraphics[width=0.3\linewidth]{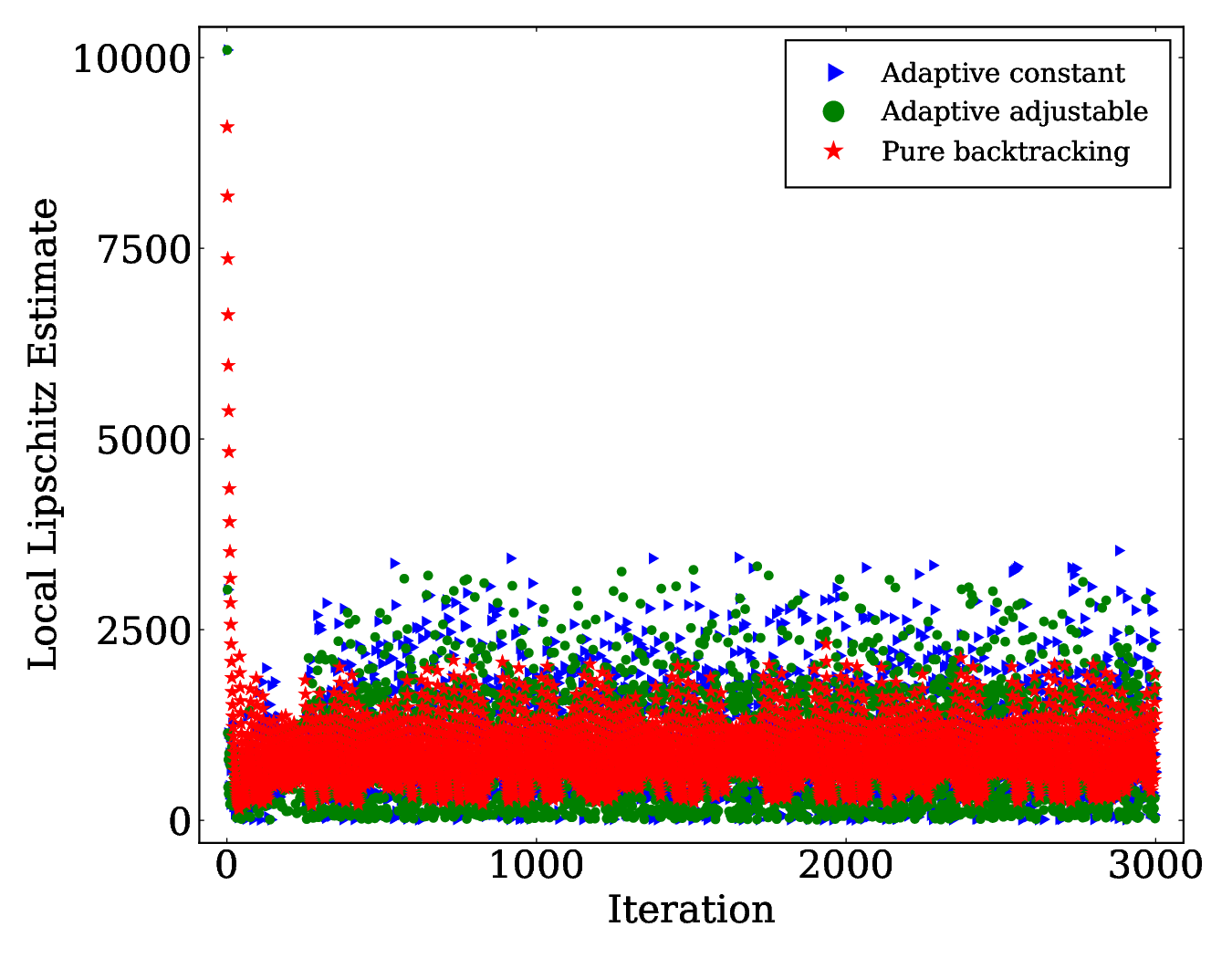}
}\hfill
\subfloat[Gradient Norms over Time]{
  \includegraphics[width=0.3\linewidth]{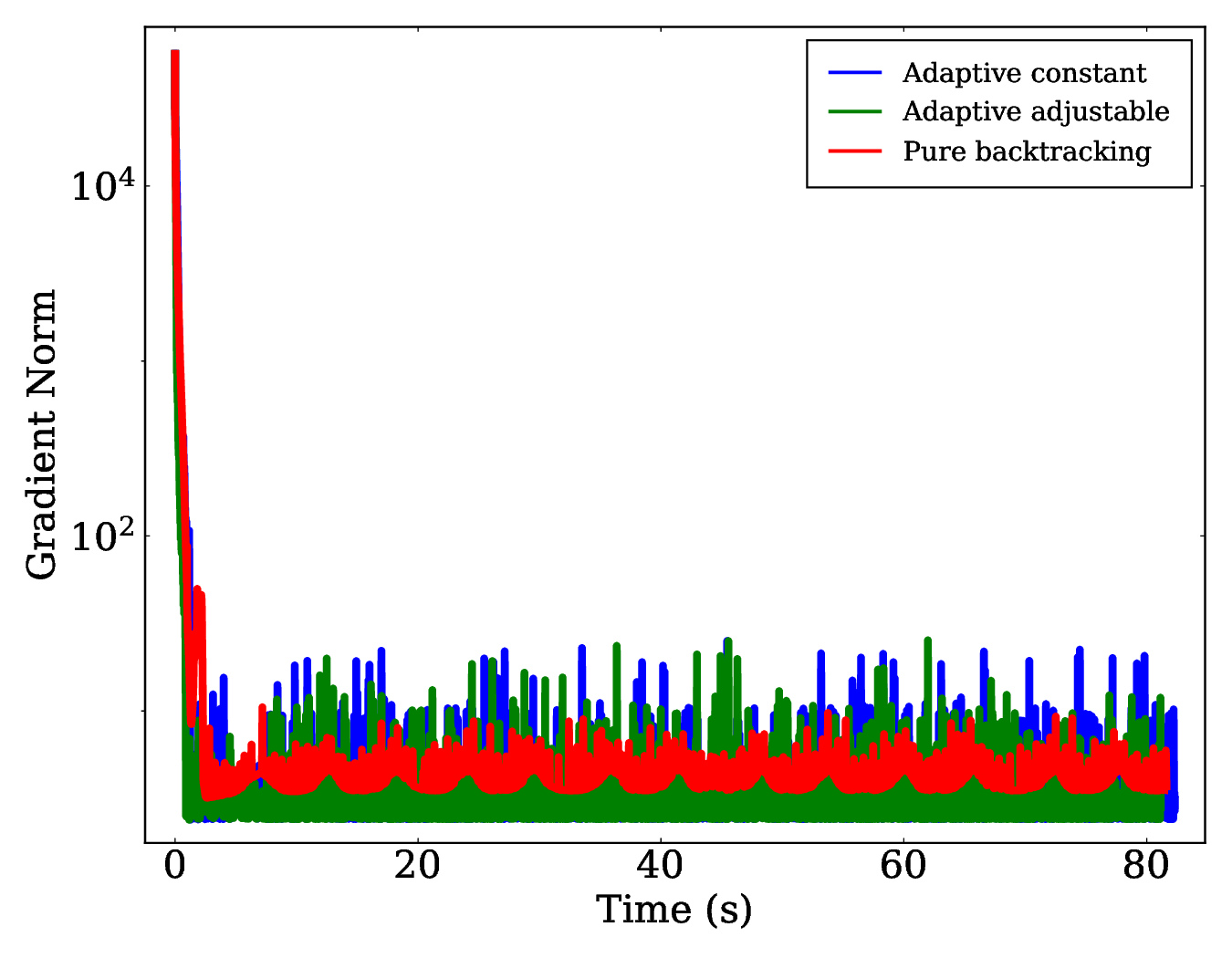}
}
\caption{Convergence behavior and computational efficiency for the non-convex Rosenbrock function.}
\label{fig:Rosenbrock_Function}
\end{figure}

\subsubsection{Lévy Function}

We evaluate the non-convex multimodal Lévy function (a common benchmark for global optimization algorithms) defined as

\begin{equation}
\begin{array}{cl}
 \min\limits_{x \in \mathbb{R}^n} & \sin^2(\pi w_1) + \sum\limits_{i=1}^{n-1} (w_i - 1)^2 \left(1 + 10 \sin^2(\pi w_{i+1})\right) + (w_n - 1)^2 \left(1 + \sin^2(2\pi w_n)\right)
\end{array}
\end{equation}
where $ w_i = 1 + \frac{x_i - 1}{4} $ for all $ i = 1,\dots,n $. We set $n = 1,000$

Table \ref{tab:Levy_Function} compares step-size strategies for the Normalized Gradient Descent algorithm. The adaptive constant method achieves the lowest objective value and gradient norm with fewer iterations. The adaptive adjustable method is slightly faster but yields higher objective value and gradient norm. Pure backtracking is the fastest but has the highest objective value and gradient norm.
Figure \ref{fig:Levy_Function} illustrates the convergence behavior and computational efficiency.

\begin{table}[htpb!]
\centering
\caption{Performance comparison of step-size strategies on the non-convex Lévy function.}
\label{tab:Levy_Function}
{
\begin{tabular}{lcccc}
\toprule
Step-size Strategy & Iterations & Time (s) & Objective Value & Gradient Norm  \\
\midrule
Adaptive constant  & 2808 & 40.34 & \textbf{24.7208} & \textbf{0.0000}     \\
Adaptive adjustable& 3000 & 39.98 & 25.2061 & 0.3505  \\
Pure backtracking  & 3000 & \textbf{39.02} & 522.6736 & 3.2346 \\
\bottomrule
\end{tabular}
}
\end{table}

\begin{figure}[htpb!]
\centering
\subfloat[Gradient Norms in Log Scale]{
   \includegraphics[width=0.3\linewidth]{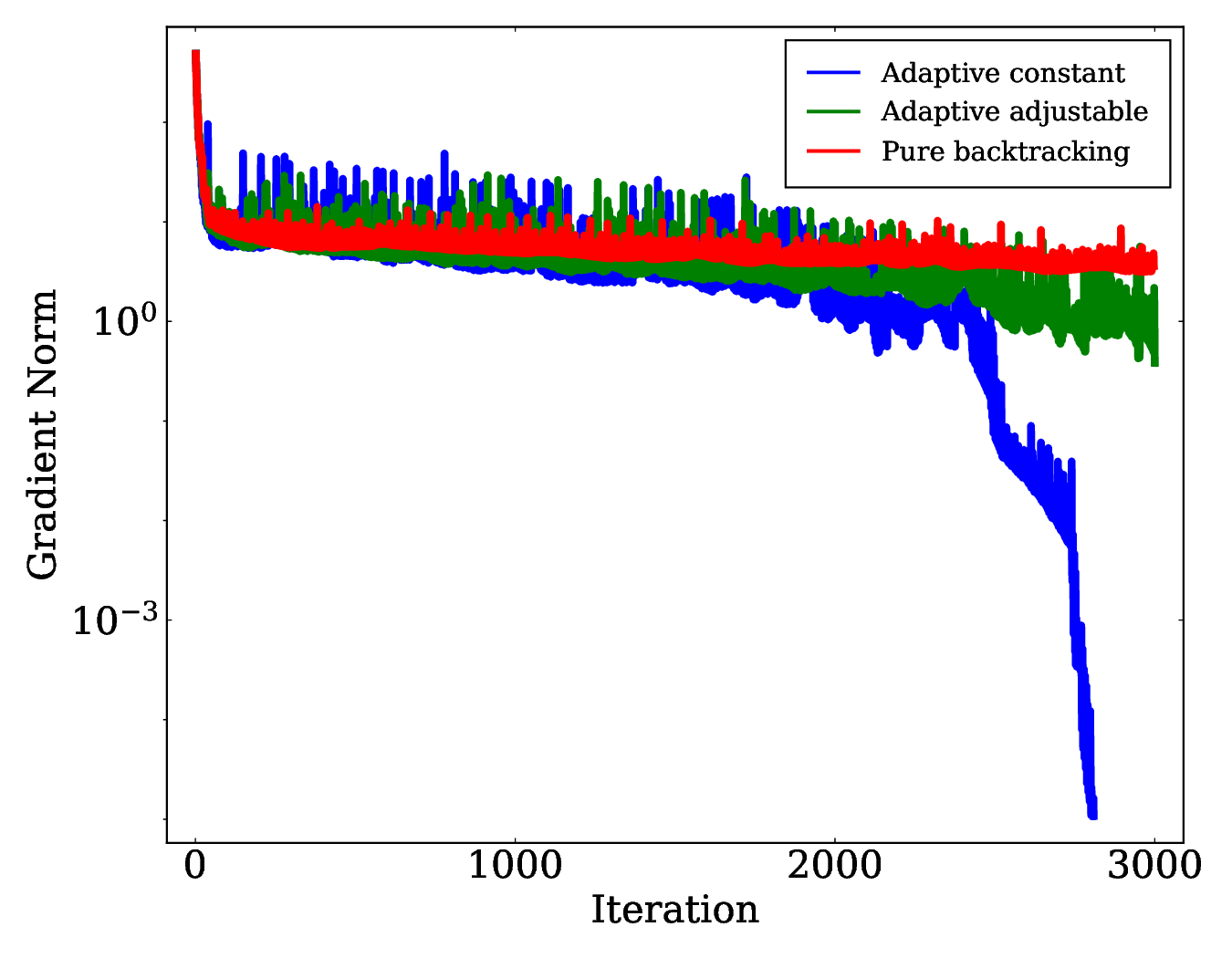}
}\hfill
\subfloat[Lipschitz Constant Evolution]{
   \includegraphics[width=0.3\linewidth]{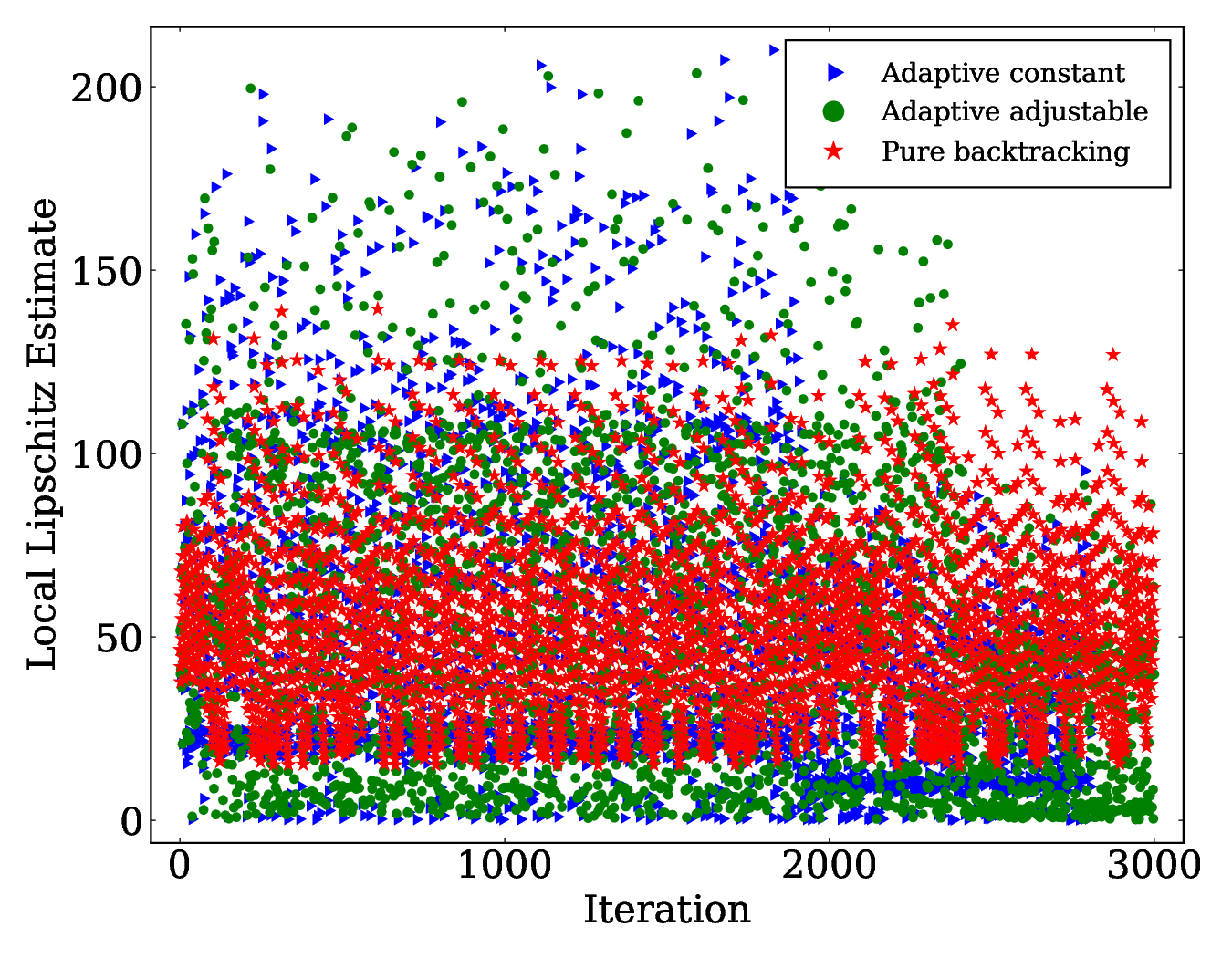}
}\hfill
\subfloat[Gradient Norms over Time]{
   \includegraphics[width=0.3\linewidth]{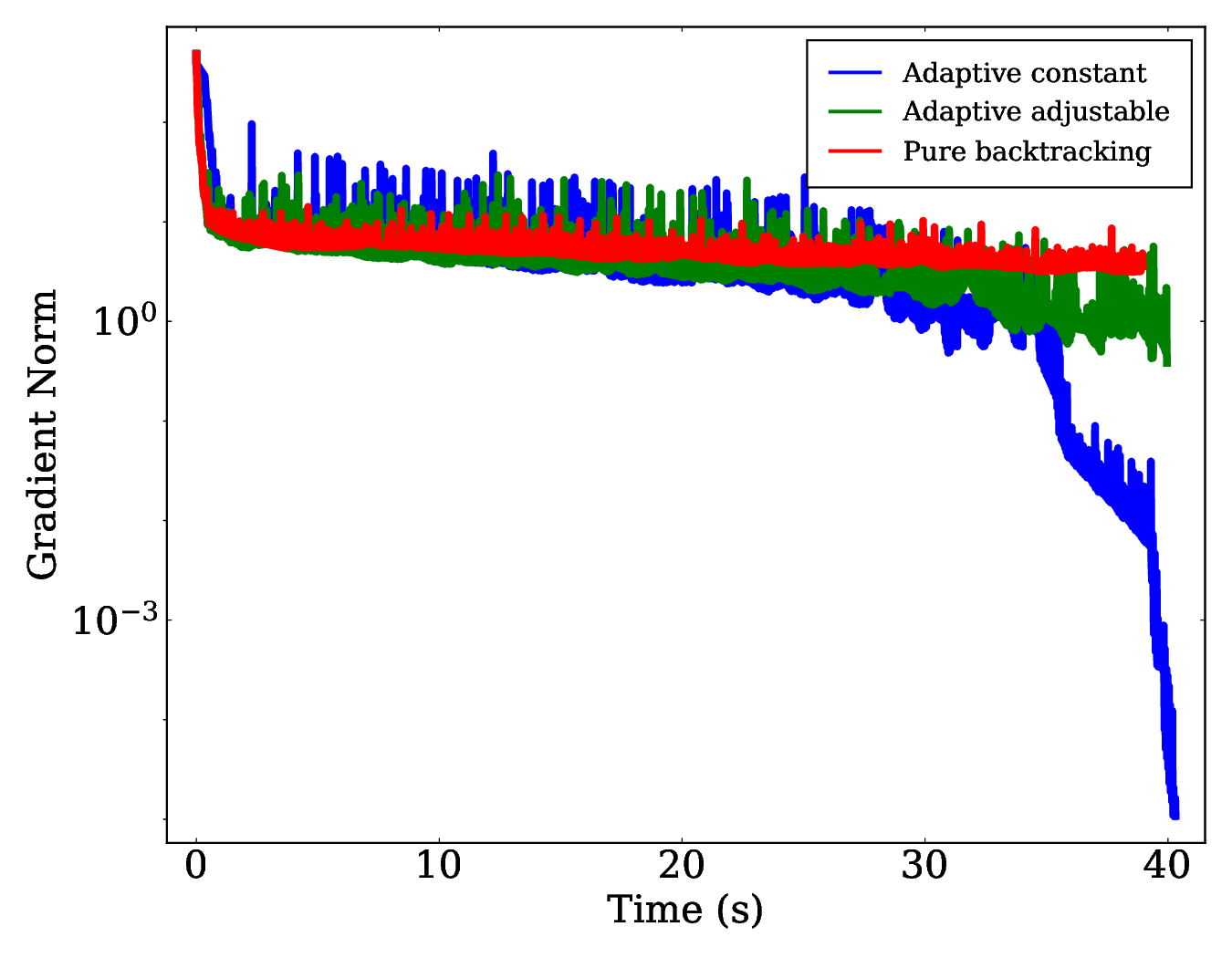}
}
\caption{Convergence behavior and computational efficiency for the non-convex Lévy function.}
\label{fig:Levy_Function}
\end{figure}

\subsubsection{Zakharov Function}

%

We evaluate the Zakharov function, a convex benchmark for optimization algorithms, defined as
\begin{equation}
\begin{array}{cl}
 \min\limits_{x \in \mathbb{R}^n} & \sum\limits_{i=1}^{n} x_i^2 + \left( \frac{1}{2} \sum\limits_{i=1}^{n} i x_i \right)^2 + \left( \frac{1}{2} \sum\limits_{i=1}^{n} i x_i \right)^4.
\end{array}
\end{equation}
We set $n=1,000$.

Table \ref{tab:Zakharov_Function} compares step-size strategies for the Normalized Gradient Descent algorithm. The adaptive adjustable method achieves the lowest objective value and gradient norm. The adaptive constant method is slightly slower with a higher objective value and gradient norm. Pure backtracking is the fastest but yields the highest objective value and a comparable gradient norm.
Figure \ref{fig:Zakharov_Function} illustrates the convergence behavior and local Lipschitz constant estimation for $k \geq 250$.
\begin{table}[htpb!]
\centering
\caption{Performance comparison of step-size strategies on the convex Zakharov function}
\label{tab:Zakharov_Function}
{
\begin{tabular}{lcccc}
\toprule
Step-size Strategy & Iterations & Time (s) & Objective Value & Gradient Norm \\
\midrule
Adaptive constant  & 3000 & 46.27 & 1041.89 & 112.34 \\
Adaptive adjustable& 3000 & 45.73 & \textbf{1039.44} & \textbf{64.58}   \\
Pure backtracking  & 3000 & \textbf{44.71} & 1044.64 & 64.64   \\
\bottomrule
\end{tabular}
}
\end{table}

\begin{figure}[htpb!]
\centering
\subfloat[Gradient Norms in Log Scale]{
   \includegraphics[width=0.3\linewidth]{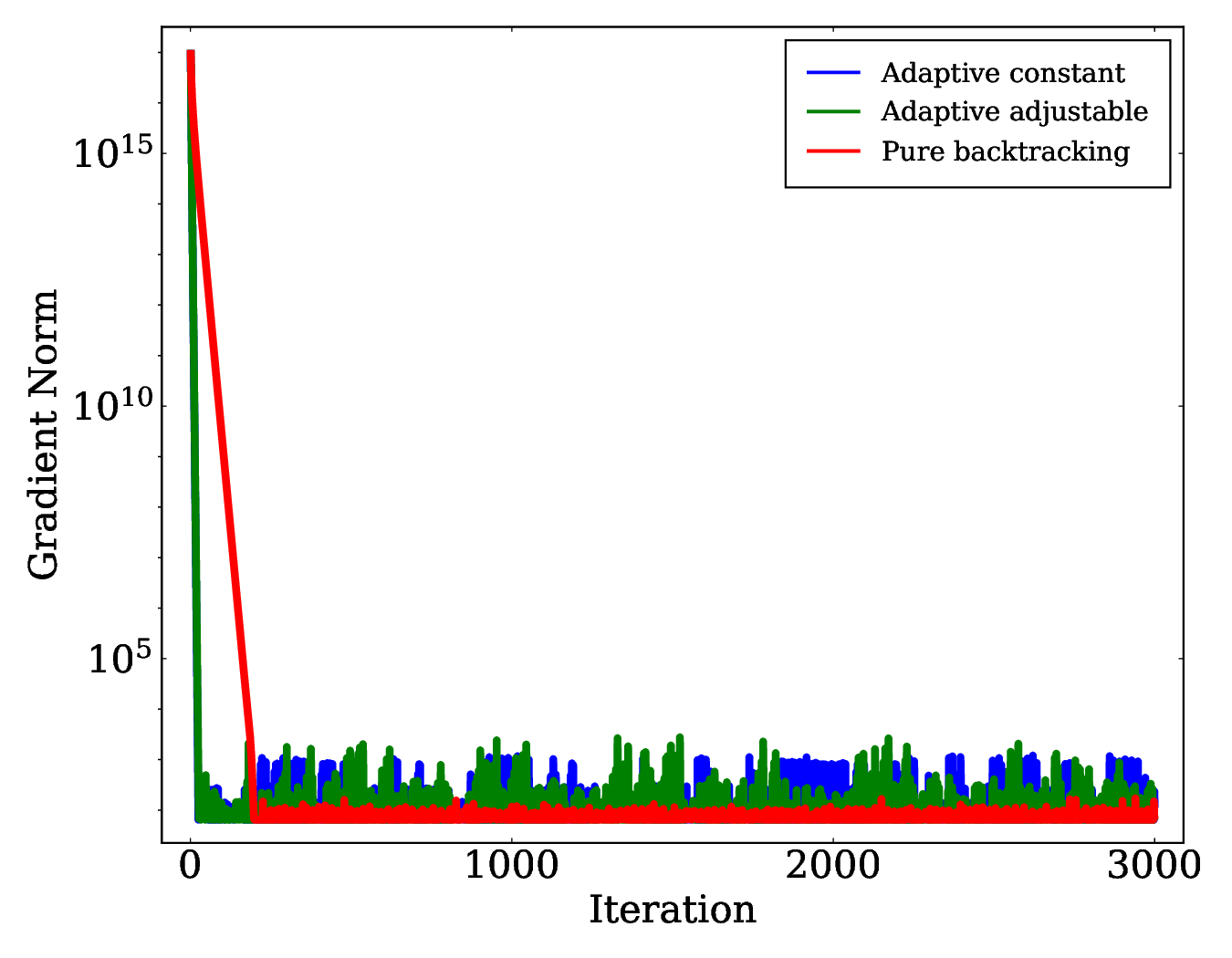}
}\hfill
\subfloat[Lipschitz Constant Evolution ($k\geq250$)]{
   \includegraphics[width=0.3\linewidth]{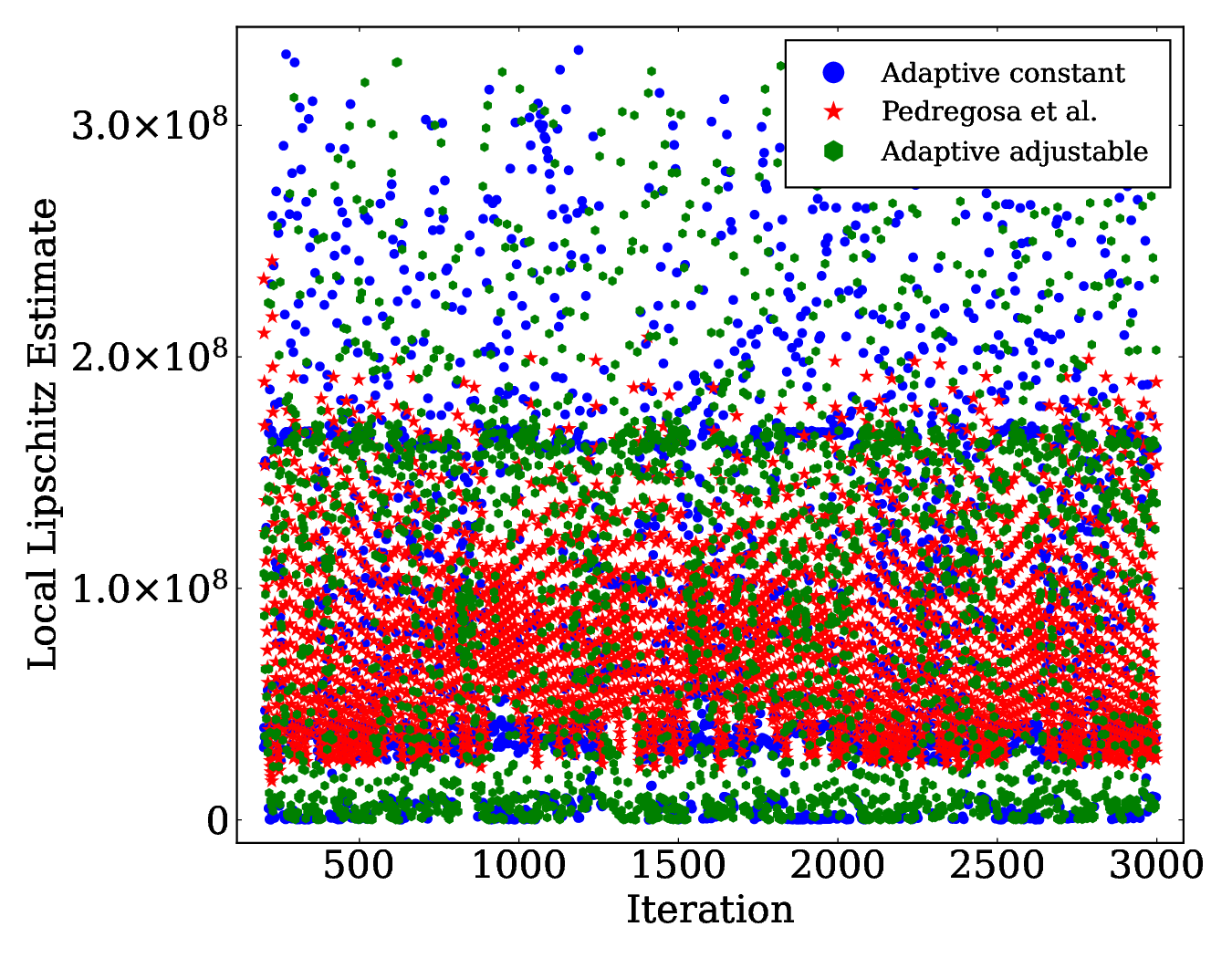}
}\hfill
\subfloat[Gradient Norms over Time]{
   \includegraphics[width=0.3\linewidth]{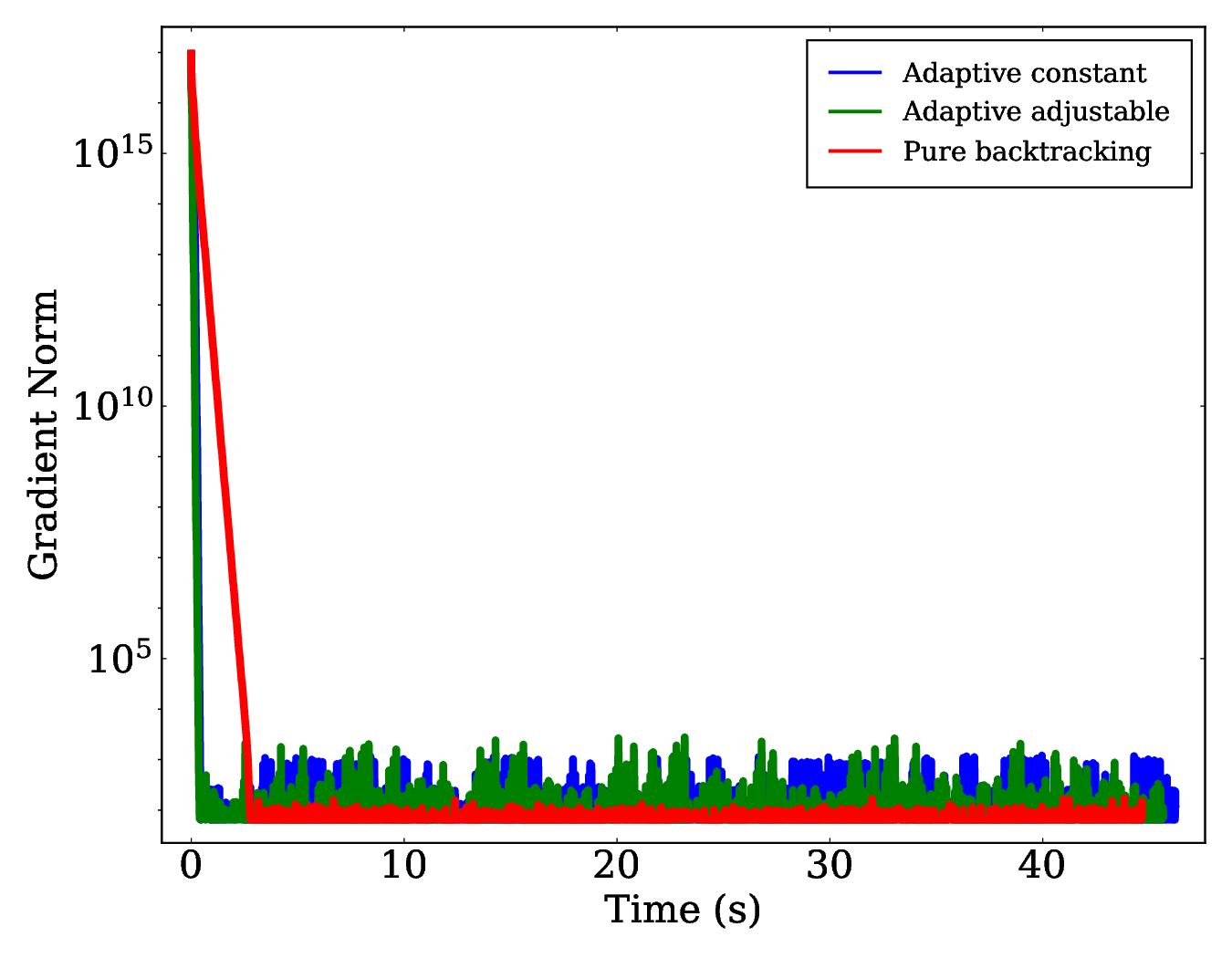}
}
\caption{Convergence behavior and computational efficiency for the convex Zakharov function. Note that we truncate the first 250 iterations, which are quite noisy.}
\label{fig:Zakharov_Function}
\end{figure}


\section{Conclusion}\label{sec:conclusion}

In this paper, we have presented adaptive step-size strategies for two important classes of first-order optimization algorithms that rely on LMOs: the Frank-Wolfe or Conditional Gradient algorithm for constrained optimization and the Normalized Steepest Descent method for unconstrained optimization. This culminated in the Adaptive Conditional Gradient Descent algorithm, \nameref{alg:ACGD}, for which we provided theoretically sound, practical methods for automatically selecting step-sizes without requiring prior knowledge of the Lipschitz constant.

For constrained settings, we demonstrated that \nameref{alg:ACGD} directly estimates the local Lipschitz constant from gradient differences at each iteration, adapting more rapidly to the local curvature compared to previous approaches \cite{pedregosa2020linearly}. We established convergence rates of $\mathcal{O}(1/k)$ for quasar-convex problems and $\mathcal{O}(1/\sqrt{k})$ for general non-convex problems, matching the worst-case theoretical guarantees of the classical Conditional Gradient method while eliminating the need for manual parameter tuning.

Similarly, for unconstrained settings, we proved a linear convergence rate for strongly convex problems and convergence rates of $\mathcal{O}(1/k)$ in the quasar-convex case and $\mathcal{O}(1/\sqrt{k})$ in the non-convex case. Our unified analysis framework gives a harmonized presentation of both methods, highlighting their theoretical connections despite addressing different optimization scenarios.

The key advantage of our approach lies in its practicality and simplified theoretical analysis. By dynamically adjusting step-sizes based on local function behavior, our methods maintain robust performance across diverse problem instances without requiring problem-specific parameter tuning. Allowing for fine-grained choices of the norms involved in both the backtracking procedure and the LMO computation gives more versatility, as shown in the numerical experiments. This adaptivity is particularly valuable in machine learning applications where problem characteristics may be initially unknown or change during the optimization process.

\begin{acknowledgements}
 The research of the first author was in part supported by a grant from IPM (No.1404900034).

\end{acknowledgements}

 \begin{appendix}
 \section*{Appendices}

 \begin{table}[htbp]
\centering
\caption{Closed-Form Solutions for selections of the LMO over Specified Feasible Regions} 
\label{tab:lmo}
\begin{tabular}{l l l}
\toprule
Name & Feasible Region $\mathcal{C}$ & Closed Form $\lmo\limits_{\mathcal{C}}(x)$ \\
\midrule
Scaled $\ell^2$-Ball & $\{ v \mid \|v\|_2 \leq \tau \}$ & $-\frac{\tau x}{\|x\|_2} $ if $x \neq 0 $, else 0 \\
Scaled Simplex & $\{ v \mid \sum v_i = \tau, v_i \geq 0 \}$ & $\tau e_j $, where $j = \argmin_i x_i $ \\
Scaled $\ell^1$-Ball & $\{ v \mid \|v\|_1 \leq \tau \}$ & $-\tau e_j \cdot \text{sign}(x_j) $, where $j = \arg\max_i |x_i| $ \\
General Box & $\{ v \mid -l_i \leq v_i \leq u_i \}$ & $v_i = -l_i $ if $x_i > 0 $, $u_i $ if $x_i < 0 $, 0 if $x_i = 0 $ \\
Scaled $\ell^\infty$-Ball & $\{ v \mid \|v\|_\infty \leq \tau \}$ & $v_i = -\tau \text{sign}(x_i) $ \\
Nuclear Norm-Ball & $\{ V \mid \|V\|_{\text{nuc}} \le \tau \}$ & $-\tau u v^\top, \text{ where } u, v \text{ are top singular vectors of } X$ \\
Spectral Norm Ball & $\{ V \mid \|V\|_{\text{op}} \le \tau \}$ & $-\tau UV^\top, \text{ where } X=U\Sigma V^\top \text{ is the reduced-SVD of } X$ \\
\bottomrule
\end{tabular}
\end{table}

 \end{appendix}

\bibliographystyle{spmpsci}
\bibliography{bibliography}

\end{document}